\documentclass{article}

\usepackage[margin=3cm]{geometry}
\usepackage{microtype}
\usepackage[nottoc]{tocbibind}
\usepackage[utf8]{inputenc}
\usepackage{graphicx}
\usepackage{float}
\usepackage{bm}
\usepackage{amsthm}
\usepackage{amsmath}
\usepackage{amssymb}
\usepackage{mathrsfs}
\usepackage{bbm}
\usepackage[all]{xy}
\usepackage{tikz-cd}
\usepackage{stmaryrd}
\usepackage{enumitem}
\usepackage{tabularx}
\usepackage{placeins}
\usepackage{pdflscape}
\usepackage{hyperref}
\usepackage{xcolor}
\usepackage{import}

\newcommand{\Inv}{\Dot{\mathcal{S}}}
\newcommand{\Tz}{\mathbb{T}_\zeta}
\newcommand{\Sp}{\mathbb{S}}

\renewcommand{\SS}{\mathcal{S}}

\DeclareMathOperator{\Ver}{\mathrm{Ver}}
\DeclareMathOperator{\Tilt}{\mathbf{Tilt}}
\DeclareMathOperator{\TL}{\mathbf{TL}}

\DeclareMathOperator{\HD}{\mathbf{HD}}
\DeclareMathOperator{\cob}{\mathbf{Cob}}
\DeclareMathOperator{\Vect}{Vec}
\DeclareMathOperator{\End}{End}
\DeclareMathOperator{\Hom}{Hom}

\theoremstyle{definition}
\newtheorem{Definition}{Definition}[section]

\theoremstyle{plain}
\newtheorem{Theorem}[Definition]{Theorem}

\theoremstyle{plain}
\newtheorem*{theorem}{Theorem}

\newcounter{mainthm}

\newtheorem{maintheorem}[mainthm]{Theorem}

\theoremstyle{plain}
\newtheorem{Proposition}[Definition]{Proposition}

\theoremstyle{plain}
\newtheorem{Lemma}[Definition]{Lemma}

\theoremstyle{plain}
\newtheorem{Corollary}[Definition]{Corollary}

\theoremstyle{plain}
\newtheorem*{corollary}{Corollary}

\theoremstyle{plain}

\newtheorem{mainconjecture}[mainthm]{Conjecture}

\theoremstyle{definition}
\newtheorem{Question}[Definition]{Question}

\newtheorem{mainquestion}[mainthm]{Question}

\theoremstyle{plain}

\theoremstyle{definition}
\newtheorem{Notation}[Definition]{Notation}

\theoremstyle{definition}

\theoremstyle{definition}
\newtheorem{Example}[Definition]{Example}

\theoremstyle{remark}
\newtheorem{Remark}[Definition]{Remark}

\definecolor{mygray}{gray}{0.6}
\definecolor{mygraydark}{gray}{0.4}
\definecolor{mygraylight}{gray}{0.8}

\definecolor{cherry}{RGB}{222,49,99}
\definecolor{cream}{RGB}{255,253,208}
\definecolor{corn}{RGB}{251,236,93}
\definecolor{citron}{RGB}{190,180,90}

\definecolor{spinach}{RGB}{46,139,87}
\definecolor{tomato}{RGB}{255,99,71}
\definecolor{pumpkin}{RGB}{224,180,80}

\definecolor{orchid}{RGB}{143,40,194}
\definecolor{lava}{RGB}{207,16,32}
\definecolor{mydarkblue}{RGB}{10,10,150}

\definecolor{myorange}{RGB}{225,127,0}
\definecolor{mygreen}{RGB}{0,225,0}
\definecolor{mypurple}{RGB}{128,0,128}
\definecolor{myred}{RGB}{255,0,0}
\definecolor{myblue}{RGB}{0,0,195}
\definecolor{myyellow}{RGB}{210,210,0}

\tikzstyle{densely dotted}=[dash pattern=on \pgflinewidth off .5pt]
\tikzset{anchorbase/.style={baseline={([yshift=-0.5ex]current bounding box.center)}},
tinynodes/.style={font=\tiny, text height=0.25ex, text depth=0.05ex},
smallnodes/.style={font=\scriptsize, text height=0.75ex, text depth=0.15ex},
crossline/.style={preaction={draw=white,line width=5.0pt,-},preaction={draw=black,line width=0.9pt,-}},
usual/.style={line width=1.0,color=black},
mor/.style={line width=1.0,black,fill=mygray,fill opacity=0.35},
morl/.style={draw,rectangle,minimum height=0.5cm,minimum width=0.75cm,
line width=1.0,fill=mygray,fill opacity=0.35,path picture={\draw[solid,line width=5.0,black]
([xshift=0pt] current path bounding box.south west)--([xshift=0pt] current path bounding box.north west)
;}},
morr/.style={draw,rectangle,minimum height=0.5cm,minimum width=0.75cm,
line width=1.0,black,fill=mygray,fill opacity=0.35,path picture={
\draw[solid,line width=5.0,black]
(current path bounding box.south east)--(current path bounding box.north east);
}},
JW/.style={line width=1.0,color=black},
pQJWl/.style={draw,rectangle,minimum height=0.5cm,minimum width=0.75cm,
line width=1.0,color=black,fill=corn!60,path picture={
\draw[solid,line width=5.0,black]
(current path bounding box.south east)--(current path bounding box.north east);
}},
pJWl/.style={draw,rectangle,minimum height=0.5cm,minimum width=0.75cm,
line width=1.0,color=black,fill=orchid!70,path picture={\draw[solid,line width=5.0,black]
([xshift=0pt] current path bounding box.south east)--([xshift=0pt] current path bounding box.north east)
;}},
pQJW/.style={draw,rectangle,minimum height=0.5cm,minimum width=0.75cm,
line width=1.0,color=black,fill=corn!60,path picture={\draw[solid,line width=5.0,black]
([xshift=0pt] current path bounding box.south west)--([xshift=0pt] current path bounding box.north west)
;}},
pJW/.style={draw,rectangle,minimum height=0.5cm,minimum width=0.75cm,
line width=1.0,color=black,fill=orchid!70,path picture={\draw[solid,line width=5.0,black]
([xshift=0pt] current path bounding box.south west)--([xshift=0pt] current path bounding box.north west)
;}},
cJW/.style={line width=1.0,color=black,fill=cherry!70},
lJW/.style={line width=1.0,color=black,fill=pumpkin!70},
}

\newcommand{\qjwm}[1][v{-}1]{\mathtt{#1}}
\newcommand{\pjwm}[1][v{-}1]{\mathtt{#1}}

\newcommand{\fusidem}[2]{\mathrm{B}^{#2}_{#1}}
\newcommand{\fusidema}[2]{\mathrm{A}^{#2}_{#1}} 

\newcommand{\ptru}[5]{
\draw [pJW] (-#4-#2/2-#2/4,#5) to (-#4-#1,#5) to (-#4-#1,#5+#2) to (-#4,#5+#2) to (-#4-#2/2-#2/4,#5);
\node at (-#4-#1/2-#1/16,#5+#2/2-0.05) {#3};
}

\newcommand{\ptrd}[5]{
\draw [pJW] (-#4-#2/2-#2/4,#5) to (-#4-#1,#5) to (-#4-#1,#5-#2) to (-#4,#5-#2) to (-#4-#2/2-#2/4,#5);
\node at (-#4-#1/2-#1/16,#5-#2/2-0.05) {#3};
}

\newcommand{\ptr}[5]{
\draw [pJW] (-#4-#1,#5) to (-#4-#1,#5-#2) 
to (-#4,#5-#2) to (-#4-#2/2-#2/4,#5) to (-#4,#5+#2) to (-#4-#1,#5+#2) to (-#4-#1,#5);
\node at (-#4-#1/2-#1/16,#5-0.05) {#3};
}

\title{\vskip-40pt Invariants of 4-Dimensional 2-Handlebodies from\\ the Temperley-Lieb Category in Positive Characteristic}
\author{Thibault D. D{\'e}coppet and Benjamin Ha{\"i}oun}
\date{December 2025}

\begin{document}

\bibliographystyle{alpha}

\maketitle
    \hspace{1cm}
    \begin{abstract}
    We investigate invariants of 4-dimensional 2-handlebodies associated to the Temperley-Lieb category in characteristic $p>2$ and at a primitive fourth root of unity. These invariants depend additionally on a height parameter $n$, and we focus on the case $n=2$.
    Provided that $p>3$, we show that the height $n=2$ invariant associated to the Temperley-Lieb category at a primitive fourth root of unity vanishes on $\mathbb{C}P^2$, $\overline{\mathbb{C}P}^2$, and $S^2\!\times\! S^2$. In particular, it has the potential to detect exotic smooth structures.
    \end{abstract}

\tableofcontents

\section*{Introduction}
\addcontentsline{toc}{section}{Introduction}

In \cite{W}, Witten proposed the existence of a 3-dimensional quantum field theory associated to the Jones polynomial. Realizing his vision, Reshetikhin and Turaev famously constructed such a quantum field theory from a highly structured Hopf algebra \cite{RT}. Specifically, the relation with the Jones polynomial is achieved by considering Lusztig's divided power quantum group for $\mathfrak{sl}_2$ at a complex root of unity. The corresponding invariants of oriented 3-manifolds were extensively studied in \cite{KM}, where they were shown to recover many classical invariants. In a different direction, Lickorish put forward a skein-theoretic construction of the Reshetikhin-Turaev invariant \cite{Lic:TL, Lic:invariants, Lic:calculations} (see also \cite{BHMV:invariant}). His approach relies on the close relationship between the representation theory of quantum $\mathfrak{sl}_2$ and the Temperley-Lieb algebras. The combinatorics of the Temperley-Lieb category was then used in \cite{KL} to explicitly compute the value of the Reshetikhin-Turaev invariants on small 3-manifolds. Finally, a skein-theoretic description of the 3-dimensional topological quantum field theory of Reshetikhin and Turaev was obtained in \cite{BHMV:TQFT}.

More generally, Turaev associated a 3-dimensional TQFT to any modular fusion category \cite{Tur:invariant,Tur:TQFT}. This construction was subsequently expanded further to allow for modular tensor categories that are not necessarily semisimple \cite{DGGPR}. In this case, the output turns out to be a so called non-compact 3-dimensional TQFT, meaning that it can generally only be evaluated on 3-manifold with non-empty incoming boundary.
In a different direction, it was shown in \cite{CY} that there is a 4-dimensional TQFT associated to any modular fusion category, and, more generally, to any ribbon fusion category \cite{CKY}.
This latter construction was recently generalized to accommodate for ribbon tensor categories that are not necessarily semisimple.
One such construction, introduced in \cite{CGHPM}, proceeds via skein-theoretic methods utilizing a convenient presentation of the category of oriented 4-manifolds from \cite{Juh}.
More precisely, given a ribbon tensor category satisfying a technical assumption called \emph{chromatic non-degeneracy}, they produce a non-compact 4-dimensional TQFT. Furthermore, under the stronger assumption that the input ribbon tensor category is \emph{chromatic compact}, the corresponding 4-dimensional TQFT is in fact defined on all 4-manifolds.
A distinct construction was laid out in \cite{BdR}.
Given a ribbon tensor category, they produce an invariant of 4-dimensional 2-handlebodies, which are presentations of 4-manifolds with boundary consisting of 0-, 1-, and 2-handles only.

A primary interest in 4-dimensional topology is the detection of exotic smooth structures, that is, the ability to distinguish 4-manifolds that are homeomorphic but not diffeomorphic.
Currently, the only known algebraic invariant with this property is Khovanov homology \cite{RW}.
We find it natural to ask whether the previously mentioned 4-dimensional TQFTs derived from ribbon tensor categories are able to detect exotic smooth structures.
Various properties of such TQFTs are known to prevent this behaviour.
For instance, Gompf showed in \cite{Gom:stable} that any two homeomorphic 4-manifolds (potentially with boundary) become diffeomorphic after a finite sequence of connected sums with $S^2\!\times\! S^2$. A similar result also holds with the pair $\mathbb{C}P^2$ and $\overline{\mathbb{C}P}^2$.
It follows that, for a 4-dimensional TQFT to have the potential to detect exotic smooth structures, it must vanish on $S^2\times S^2$ and at least one of $\mathbb{C}P^2$ and $\overline{\mathbb{C}P}^2$.
In particular, it follows from results of \cite{Reu,H:WRTdelCY} that the 4-dimensional TQFTs associated to chromatic compact ribbon tensor categories via the construction of \cite{CGHPM} cannot vanish on $S^2\!\times\! S^2$, and so cannot detect exotic smooth structures.

There is also ample motivation for exhibiting invariants of 4-dimensional 2-handlebodies, which are presentations of 4-manifolds with boundary that consist exclusively of 0-, 1-, and 2-handles. Recall that there are two natural notions equivalence between those:\ diffeomorphism and 2-equivalence. Two 2-handlebodies are said to be 2-equivalent if they can be transformed into one another without introducing a 3-handle.
It is believed that 2-equivalence is a strictly stronger equivalence relation than diffeomorphism, but no counterexample is currently known. Additional motivation for tackling this question stems from its intimate relationship with the celebrated Andrews-Curtis conjecture \cite{Gom:AkbulutKirby}.

The strength of the invariants associated to ribbon tensor categories is closely related to the properties of the ribbon tensor category.
For instance, it is known that finite ribbon tensor categories that are semisimple or, more generally, chromatic compact produce via the construction of \cite{CGHPM} invariants that only depend on the homeomorphism type \cite{Reu}.
Natural candidates avoiding these obstructions include the categories of representations of small quantum groups. Even though the corresponding ribbon tensor categories are not chromatic compact, nevertheless the corresponding invariants are subject to stabilization as they do not vanish on $S^2\!\times\! S^2$ as was shown in \cite{CGHPM}, see also \cite{BdR}. Examples derived from less familiar Hopf algebras were also considered in \cite{FM, Man}, but they also fail at detecting exotic smooth structures. It therefore seems necessary to provide truly exotic ribbon tensor categories as input.

Now, chromatic compactness is a skein-theoretic condition, and is therefore quite difficult to check in practice.
It is however believed that chromatic compactness admits a purely algebraic characterization.
Recall that the symmetric, also called M\"uger, center of a braided tensor category is the full symmetric tensor subcategory on those objects that are transparent to the double braiding.
In particular, it is well-known that a finite ribbon tensor category is modular if and only if its symmetric center is trivial.
Conjecturally, a finite ribbon tensor category is chromatic compact if and only if its symmetric center is \textit{separable}, a refinement of the notion of semisimple that is necessary in positive characteristic.
This last characterization is known to hold for symmetric tensor categories \cite{CGHPM}.
Additional evidence for its validity is supplied by higher algebraic considerations \cite{Dec:relative}. 

The discussion of the previous paragraph motivates that, in order to yield an interesting invariant, one has to consider ribbon tensor categories whose symmetric center is not separable or, more generally, as exotic as possible.
The most exotic symmetric tensor categories are defined over fields of positive characteristic $p>0$. Specifically, in \cite{BEO, C:monoidal}, the authors introduced a family $\mathrm{Ver}_{p^n}$ of remarkable symmetric tensor categories built from quotients of the Temperley-Lieb category at parameter $\pm 1$.
Here, the natural number $n$ should be, roughly speaking, thought of as a measure of the complexity of the category $\mathrm{Ver}_{p^n}$.
To be specific, we have that $\mathrm{Ver}_{p^n}$ is not separable provided that $p\neq 2,3$ if $n=1$.
This construction was subsequently generalized in \cite{STWZ}, where ribbon tensor categories were built from quotients of the Temperley-Lieb category in positive characteristic at an arbitrary root of unity $\zeta$. We refer to these ribbon tensor categories as the mixed higher Verlinde categories, which we denote by $\Ver_{p^{(n)}}^{\zeta^{1/2}}$. It was shown in \cite{Dec} that, for $\zeta\neq\pm 1$ a root of unity of even order, the symmetric center of $\Ver_{p^{(n)}}^{\zeta^{1/2}}$ is $\mathrm{Ver}^+_{p^{n-1}}$, a full symmetric tensor subcategory of $\mathrm{Ver}_{p^{n-1}}$.
In particular, if $n=2$ and $p\neq 2,3$, or if $n\geq 3$, the symmetric center of $\Ver_{p^{(n)}}^{\zeta^{1/2}}$ is not separable. Consequently, these ribbon tensor categories are very natural choices of input for the construction of \cite{CGHPM}.

\subsection*{Results}
\addcontentsline{toc}{subsection}{Results}

Let us now fix an algebraically closed field $\mathbbm{k}$, and let $\mathcal{B}$ be a unimodular finite ribbon tensor category in the sense of \cite{EGNO}.
As already mentioned above, there are two constructions of an invariant of 4-dimensional 2-handlebodies up to 2-equivalence associated to $\mathcal{B}$. The first one was laid out in \cite{BdR} and utilizes Lyubashenko's universal Hopf algebra in $\mathcal{B}$.
One difficulty with this approach is that, in practice, the universal Hopf algebra is often very difficult to describe explicitly.
On the other hand, a second construction, proceeding via skein-theory, was presented in \cite{CGHPM}.
The main ingredient is, in this case, the ideal of projective objects in $\mathcal{B}$ equipped together with a choice of modified trace.\footnote{On a unimodular finite spherical tensor category, (non-zero) modified traces on the ideal of projective objects are parametrized by (non-zero) scalars \cite{GKP}.}
For our purposes, that is when $\mathcal{B}$ is a mixed higher Verlinde category, this ideal turns out to be a much more tractable set of algebraic data than the universal Hopf algebra.
To be more accurate, the construction of \cite{CGHPM} produces a non-compact 4d TQFT under the additional provision that $\mathcal{B}$ is \emph{chromatic non-degenerate} see Definition \ref{def:chrnondegen} below.
Without this assumption, their construction may still be partially carried out so as to produce an invariant of 4-dimensional 2-handlebodies up to 2-equivalence as we record below.

\newpage

\begin{theorem}[\cite{CGHPM}]
Let $\mathcal{B}$ be a unimodular finite ribbon tensor category equipped with a modified trace.
Then, there is a symmetric monoidal functor 
$$\SS_{\mathcal{B}}: \mathrm{\HD}_4^{2,3,4}\to \Vect\,,$$
from the category of 4-dimensional 2-handlebodies\footnote{The construction of \cite{CGHPM} is defined on 2-, 3- and 4-handles, whereas it is usual to define 2-handlebodies as consisting of 0-, 1- and 2-handles. The two notions correspond to each other under orientation reversal.}, see Definition \ref{def:HDcategory}.
Furthermore, provided that $\mathcal{B}$ is chromatic non-degenerate, this assignment can be extended to a symmetric monoidal functor
$$\SS_{\mathcal{B}}: \mathrm{\HD}_4^{1,2,3,4}\to \Vect\,,$$
from the category of non-compact 4-dimensional cobordisms, that is, to a non-compact 4d TQFT.
\end{theorem}

Let us now assume that $\mathbbm{k}$ has characteristic $p>0$.
We apply the above construction to the mixed higher Verlinde categories introduced in \cite{STWZ}, generalizing the symmetric higher Verlinde categories from \cite{BEO, C:monoidal}.
These finite ribbon tensor categories are derived from the Temperley-Lieb category $\mathbf{TL}^{\zeta^{1/2}}$ with parameter $\zeta$, a root of unity in $\mathbbm{k}$.
Via the Kauffman bracket, a choice of square root $\zeta^{1/2}$ for $\zeta$ endows $\mathbf{TL}^{\zeta^{1/2}}$ with a ribbon structure.
The category $\mathbf{TL}^{\zeta^{1/2}}$ is closely related to the category $\Tilt^{\zeta^{1/2}}$ of tilting modules for Lusztig's divided power quantum group for $\mathfrak{sl}_2$ at root of unity $\zeta$. Namely, the category $\Tilt^{\zeta^{1/2}}$ is the Cauchy completion of $\mathbf{TL}^{\zeta^{1/2}}$.
It was established in \cite{STWZ} that the ribbon category $\Tilt^{\zeta^{1/2}}$ contains an infinite sequence of ideals $\mathbf{J}_{p^{(n)}}$ parametrized by the natural number $n$.
Then, the quotient $\Tilt^{\zeta^{1/2}}\!\!\!/\mathbf{J}_{p^{(n)}}$ admits an abelian envelope, that is, a universal embedding into a finite tensor category, which we denote by $\Ver_{p^{(n)}}^{\zeta^{1/2}}$. Crucial to our subsequent considerations is the fact that the ideal of projective objects of $\Ver_{p^{(n)}}^{\zeta^{1/2}}$ is precisely $\mathbf{J}_{p^{(n-1)}}/\mathbf{J}_{p^{(n)}}$, and is therefore completely described by the combinatorics of the Temperley-Lieb category $\mathbf{TL}^{\zeta^{1/2}}$. The next result follows from the preceding discussion.

\begin{corollary}
For any natural number $n$ and root of unity $\zeta$ in $\mathbbm{k}$, there is an invariant $\Inv_{p^{(n)}}^{\zeta^{1/2}}$ of 4-dimensional 2-handlebodies up to 2-equivalence associated to the mixed higher Verlinde category $\Ver_{p^{(n)}}^{\zeta^{1/2}}$. Moreover, there is a symmetric monoidal functor 
$\mathcal{S}_{p^{(n)}}^{\zeta^{1/2}}: \mathrm{\HD}_4^{2,3,4}\to \Vect.$
\end{corollary}

The algebraic properties of the mixed higher Verlinde categories $\Ver_{p^{(n)}}^{\zeta^{1/2}}$ were investigated in \cite{Dec}, building on the thorough study of the symmetric higher Verlinde categories undergone in \cite{BEO}.
At height $n = 1$, the mixed Verlinde categories $\Ver_{p^{(1)}}^{\zeta^{1/2}}$ are separable.
In particular, they yield fully defined 4d TQFTs \cite{CGHPM}, and therefore relatively weak invariants of 4-manifolds \cite{Reu}.
At height $n=2$, the mixed Verlinde categories $\Ver_{p^{(2)}}^{\zeta^{1/2}}$ are no longer separable.
For definitiveness, we consider the case when $\zeta$ is a primitive fourth root of unity, so that $p \neq 2$.
In fact, we shall assume that $p>3$.
Namely, in the case $p=3$, the finite ribbon tensor category $\Ver_{3^{(2)}}^{\zeta^{1/2}}$ happens to be non-degenerate.\footnote{With $\zeta$ a fourth root of unity, the symmetric center of $\Ver_{p^{(2)}}^{\zeta^{1/2}}$ is $\Ver_p^+$. It so happens that $\Ver_3^+ = \Vect$.} The corresponding 4d TQFT is therefore not only defined on all 4-manifolds, but is also invertible \cite[Thm. 5.8]{CGHPM}. By contrast, if $p>3$, we have the following result.

\begin{maintheorem}\label{thm:234butnot1}
Let $p>3$, and let $\zeta$ be a fourth root of unity. The finite ribbon tensor category $\Ver_{p^{(2)}}^{\zeta^{1/2}}$ is chromatic degenerate. Furthermore, the functor $\mathcal{S}_{p^{(2)}}^{\zeta^{1/2}}:\mathrm{\HD}_4^{2,3,4}\to \Vect$ cannot be extended to a symmetric monoidal functor $\mathrm{\HD}_4^{1,2,3,4}\to \Vect$.
\end{maintheorem}

\noindent This suggests that the invariant $\Inv_{p^{(n)}}^{\zeta^{1/2}}$ might be able to distinguish 4-dimensional 2-handlebodies that are diffeomorphic but not 2-equivalent.

In order to establish the above result and, more generally, so as to be able to evaluate our invariant, we study the ribbon structure on the ideal of projective objects of $\Ver_{p^{(2)}}^{\zeta^{1/2}}$.
This amounts precisely to investigating the structure of the ribbon category $\Tilt^{\zeta^{1/2}}\!\!\!/\mathbf{J}_{p^{(2)}}$ with $\zeta$ a primitive fourth root of unity.
In characteristic zero, the ribbon structure of the Temperley-Lieb category is completely understood, and conveniently summarized in \cite{KL}.
By contrast, the situation in positive characteristic is drastically more challenging.
For instance, the fusion rules for arbitrary indecomposable objects of $\Tilt^{\zeta^{1/2}}$ is unknown. Various partial results in this direction were obtained in \cite{STWZ}.
Expanding their work, we obtain explicit descriptions for braidings and twists on the indecomposable objects of $\Tilt^{\zeta^{1/2}}\!\!\!/\mathbf{J}_{p^{(2)}}$ at a fourth root of unity.
With these pieces of algebraic data at our disposal, we compute the value of the corresponding invariants $\Inv_{p^{(n)}}^{\zeta^{1/2}}$ on various 4-dimensional 2-handlebodies. Of particular interest is the value of the invariant on $\mathbb{C}P^2$, $\overline{\mathbb{C}P}^2$, and $S^2\!\times\! S^2$. Namely, as we have already mentioned above, an invariant that does not vanish on these 4-manifolds is subject to stabilization, henceforth cannot detect exotic smooth structures.

\begin{maintheorem}
Let $p>3$, and let $\zeta$ be a primitive fourth root of unity. Then, we have
$$\Inv_{p^{(2)}}^{\zeta^{1/2}}(\mathbb{C}P^2) = 0\,,\quad \Inv_{p^{(2)}}^{\zeta^{1/2}}(\overline{\mathbb{C}P}^2) = 0\,,\quad \Inv_{p^{(2)}}^{\zeta^{1/2}}(S^2\!\times\! S^2) = 0\,.$$
In particular, the invariant $\Inv_{p^{(2)}}^{\zeta^{1/2}}$ is not subject to stabilization.
\end{maintheorem}

\noindent Together with those of \cite{FM}, our invariants are some of the few derived from tensor categories that are known to exhibit these properties.
While many finite ribbon tensor categories are known to yield invariants vanishing on $\mathbb{C}P^2$ and $\overline{\mathbb{C}P}^2$ \cite{BdR, CGHPM, Man}, in all of these examples, the corresponding invariants do not vanish on $S^2\!\times\! S^2$. They are therefore subject to stabilization and cannot detect exotic smooth structures.

A well-known example of 4-dimensional 2-handlebody exhibiting exotic behavior is the Mazur manifold \cite{Ak}. More precisely, the Mazur manifold admits a self-diffeomorphism of its boundary that can be extended to a self-homeomorphism of the interior, but not to a self-diffeomorphism.
An advantage of the construction of \cite{CGHPM} is that it allows us to probe such phenomena using skeins.
We compute the value of our invariants on the Mazur manifold decorated by various skeins.
This exotic phenomenon is however not detected by our computations.

\subsubsection*{Partially Defined TQFTs and Partial Dualizability}

We presently give a terse higher-algebraic perspective on our main results. Recall that the Cobordism Hypothesis \cite{BD,Lur:TFT} posits an algebraic classification of fully extended framed $n$-dimensional TQFTs.
More precisely, a fully extended framed TQFT with value in a symmetric monoidal higher category is completely determined by its value at the point, which is necessarily an $n$-dualizable object of the target higher category.
This classification can be extended to partially-defined TQFTs, which are then classified by partially dualizable objects. In more detail, there is a notion of $(n+\frac{k}{n+1})$-dualizable object, implicit in \cite{Lur:TFT} and formalized in \cite{H:unit}.
The corresponding $(n+1)$d TQFT is, a priori, only defined on $(n+1)$-cobordism obtained by attaching handles of index between $0$ and $k$.

The 4-dimensional TQFTs of \cite{CGHPM} are expected to be fully extended. As evidence in this direction, these TQFTs are shown to be once-extended in \cite{H:WRTdelCY}.
Furthermore, it is very natural to guess that the putative corresponding fully extended TQFT has value on the point given by the input ribbon tensor category seen as an object of the Morita 4-category $\mathrm{Mor}_2(\Pr)$ of presentable braided monoidal categories \cite{JFS, BJS}. 
One therefore suspects that the properties of the TQFTs of \cite{CGHPM} reflect the dualizability properties of the input ribbon tensor category, thought of as an object of $\mathrm{Mor}_2(\Pr)$. 
In fact, it was shown in \cite{Dec:relative} that the dualizability of a braided tensor category is completely controlled by the dualizability of its symmetric center.
As a consequence, from the construction of \cite{CGHPM} and Theorem \ref{thm:234butnot1}, we expect:

\begin{mainconjecture}
Every unimodular finite braided tensor category, and, in particular every mixed higher Verlinde category $\Ver^{\zeta^{1/2}}_{p^{(n)}}$, is $(3+\frac{2}{4})$-dualizable in $\mathrm{Mor}_2(\Pr)$.
Furthermore, with $p>3$ and $n=\ell=2$, the mixed higher Verlinde category $\Ver^{\zeta^{1/2}}_{p^{(2)}}$ is not $(3+\frac{3}{4})$-dualizable.
\end{mainconjecture}

\noindent We plan on coming back to this as well as related questions in future work.

\subsection*{Further Directions}
\addcontentsline{toc}{subsection}{Further Directions}

We have focused on the invariant $\Inv_{p^{(n)}}^{\zeta^{1/2}}$ associated to the mixed higher Verlinde categories $\Ver_{p^{(2)}}^{\zeta^{1/2}}$ with $\zeta$ a primitive fourth root of unity. However, we suspect that the invariant $\Inv_{p^{(n)}}^{\zeta^{1/2}}$ associated to the mixed higher Verlinde categories $\Ver_{p^{(n)}}^{\zeta^{1/2}}$ with $\zeta$ an arbitrary root of unity and $n\geq 2$ have a similar behavior.

\begin{mainconjecture}
Let $n\geq 2$, and assume $p>3$ if $n=2$. For any root of unity $\zeta$, the finite ribbon tensor category $\Ver_{p^{(n)}}^{\zeta^{1/2}}$ is chromatic degenerate. Moreover, we have
$$\Inv_{p^{(n)}}^{\zeta^{1/2}}(\mathbb{C}P^2) = 0\,,\quad \Inv_{p^{(n)}}^{\zeta^{1/2}}(\overline{\mathbb{C}P}^2) = 0\,,\quad \Inv_{p^{(n)}}^{\zeta^{1/2}}(S^2\!\times\! S^2) = 0\,.$$
\end{mainconjecture}

\noindent As a consequence, we expect that the invariant $\Inv_{p^{(n)}}^{\zeta^{1/2}}$ is not subject to stabilization so long as $n\geq 2$, and $p>3$ if $n=2$.
Additionally, given that the structure of the ribbon tensor categories $\Ver_{p^{(n)}}^{\zeta^{1/2}}$ gets more complicated as $n$ increases, we expect that the corresponding invariants $\Inv_{p^{(n)}}^{\zeta^{1/2}}$ become more interesting. Relatedly, the computations also become more intricate. In particular, tackling the above conjecture surely necessitates a thorough study of the ribbon structure of the Temperley-Lieb category in positive characteristic.

As we have emphasized above, the mixed higher Verlinde categories are the most exotic finite ribbon tensor categories currently known. We therefore find it intriguing to ask whether the corresponding invariants can shed light on classical problems in 4-dimensional topology.

\begin{mainquestion}
Is the family of invariants $\Inv_{p^{(n)}}^{\zeta^{1/2}}$ able to distinguish 4-dimensional 2-handlebodies that are diffeomorphic but not 2-equivalent? If this is not the case, that is, if $\Inv_{p^{(n)}}^{\zeta^{1/2}}$ are diffeomorphism invariants, are they able to detect exotic smooth structures?
\end{mainquestion}

\noindent The two questions above fit into the general program of studying the sensitivity of the invariants associated to 4d TQFTs. This question has been completely answered for once-extended completely defined 4d TQFTs \cite{Reu,RSP} which detect exactly stable-diffeomorphism types. It is known that once-categorified 4d TQFTs, that is, TQFTs that assign vector spaces to 4-manifolds, are able to detect exotic pairs \cite{RW}.

\subsection*{Acknowledgments}

T.\ D.\ would like to thank Anna Beliakova, and Ivelina Bobtcheva for discussions. B.\ H.\ would like to thank Francesco Costantino, Nathan Geer, and Bertrand Patureau-Mirand for many helpful conversations. Both authors are especially thankful towards Maksymilian Manko for many conversations on the topic of 4-dimensional 2-handlebodies.
Both authors acknowledge support from the Simons Collaboration on Global Categorical Symmetries.

\section{Topological Background}\label{sec:topology}

We review the construction of a 4-manifold invariant from a finite unimodular ribbon tensor category from \cite{CGHPM}. In particular, we explain why it also yields invariants of 4-dimensional 2-handlebodies.

\subsection{Handle Decompositions}

The invariants of 4-manifolds that we will employ are based on handle decompositions. We therefore begin by recalling how the category 4-manifolds can be presented using handle attachments following \cite{Juh}.

\begin{Definition}\label{def:HDcategory}
    The category $\HD_n^{k,\dots,n}$ of $n$-dimensional handlebodies of handles of indices $k,\dots, n$ up to $k$-equivalence is the category whose objects are closed oriented smooth $(n-1)$-manifolds and whose morphisms are generated by:
    \begin{itemize}
        \item[(a)] Morphisms $e_d: M\to M'$ for every diffeomorphism $d:M\to M'$,
        \item[(b)] Morphisms $e_\Sp: M\to M(\Sp)$ for every framed attaching sphere $\Sp:S^{p-1}\times D^{n-p}\hookrightarrow M$, where $p\in \{k,\dots,n\}$ and $M(\Sp) = (M\smallsetminus \Sp)\cup (D^p \times S^{n-p-1})$ is the surgered manifold, 
    \end{itemize}
    subject to the following relations:
    \begin{itemize}
            \item[(1)] We have $e_d \circ e_{d'} \sim e_{d\circ d'}$ whenever $d$ and $d'$ compose, and $e_d \sim \operatorname{Id}$ whenever $d$ is isotopic to the identity,
            \item[(2)] We have $e_{M', d\circ \Sp}\circ e_d \sim e_{d^\Sp}\circ e_{M,\Sp}$ for $d:M\to M'$ and $\Sp\subseteq M$, where $d^\Sp: M(\Sp) \to M'(d\circ \Sp)$ is the diffeomorphism induced by $d$,
            \item[(3)] We have $e_{M(\Sp),\Sp'}\circ e_{M,\Sp}\sim e_{M(\Sp'),\Sp}\circ e_{M,\Sp'}$ whenever $\Sp,\Sp' \subseteq M$ are disjoint, so $\Sp'$ can be pushed to $M(\Sp)$ and $\Sp$ to $M(\Sp')$,
            \item[(4)] We have $e_{M(\Sp),\Sp'}\circ e_{M,\Sp} \sim e_\varphi$ whenever the attaching sphere of $\Sp'$ and the belt sphere of $\Sp$ intersect transversely exactly once in $M(\Sp)$, and $\varphi$ is the induced diffeomorphism $M\simeq M(\Sp)(\Sp')$, 
            \item[(5)] We have $e_{M,\Sp} \sim e_{M,\overline{\Sp}}$ where $\Sp$ is a framed $(p-1)$-sphere for some $p\in \{k,\dots,n-1\}$ and $\overline{\Sp}$ is obtained from $\Sp$ by reversing both the $S^{p-1}$ and the $D^{n-p}$ factors (which adds to an orientation-preserving automorphism of $S^{p-1}\times D^{n-p}$).
        \end{itemize}
\end{Definition}

Diffeomorphism provides another, more familiar, notion of equivalence on handlebodies.

\begin{Definition}
    The category $\cob_n^{k,\dots,n}$ of $n$-dimensional handlebodies of handles of indices $k,\dots, n$ up to diffeomorphisms is the category whose objects are closed oriented smooth $(n-1)$-manifolds and whose morphisms are diffeomorphism classes of $n$-cobordisms which admit a handle decomposition involving only handles of indices $k,\dots, n$.
\end{Definition}

\noindent For $k=0$, we have that $\cob_n^{0,\dots,n}$ is the standard cobordism category as defined, for instance, in \cite[Section 1]{Milnor_hcobordism}. This follows from the fact that every cobordism admits a handle decomposition. With $k=1$, we have that $\cob_n^{1,\dots,n}$ is the non-compact cobordism category, where every cobordism must have nonempty incoming boundary in every connected component \cite[Def. 1.3.5]{H:Thesis}.

For any $0\leq k\leq n$, the two categories $\HD_n^{k,\dots,n}$ and $\cob_n^{k,\dots,n}$ are symmetric monoidal under disjoint union. Moreover, there is a symmetric monoidal functor $$F:\HD_n^{k,\dots,n} \to \cob_n^{k,\dots,n}\,.$$ This functor is the identity on objects:\ For any $(n-1)$-manifold $M$, we have $F(M)=M$. It sends the generating morphism $e_d$ to the mapping cylinder of $d$, that is, we have $$F(e_d) = (M\times [0,1]) \underset{(x,1)\sim d(x)}{\cup}M'.$$ Finally, the functor $F$ sends the generating morphism $e_\Sp$ to the handle attachment of $\Sp$, that is, we have $$F(e_\Sp) = (M\times [0,1]) \underset{\Sp}{\cup} (D^p \times D^{n-p}).$$
This functor is manifestly essentially surjective. It is also full by construction. More precisely, choosing a handle presentation of a given cobordism is closely related to choosing a Morse function on that cobordism. Then, for $k=0$, the relations (1) -- (5) above are exactly the relations between different Morse functions coming from Cerf theory. This can be used to prove the following result.

\begin{Theorem}[\cite{Juh}]
    The functor $F:\HD_n^{k,\dots,n} \to \cob_n^{k,\dots,n}$ is an equivalence for $k=0$ or $k=1$.
\end{Theorem}

More generally, it is natural to wonder for which values of $k$ the functor $F:\HD_n^{k,\dots,n} \to \cob_n^{k,\dots,n}$ is an equivalence. In the case $n=4$ and $k=2$, this question was first raised in \cite{Gom:AkbulutKirby}, where it was related to the famous Andrews-Curtis conjecture.

\begin{Question}
    Is the functor $F:\HD_4^{2,3,4} \to \cob_4^{2,3,4}$ an equivalence? Said differently, are there diffeomorphic 4-dimensional 2-handlebodies that are not 2-equivalent?
\end{Question}

\subsection{Skein Theory}\label{sub:skeintheory}
Let $\mathcal{B}$ be a ribbon tensor category over an algebraically closed field $\mathbbm{k}$ in the sense of \cite{EGNO}.
Given any oriented 3-manifold $M$, we can consider the admissible skein module $\SS_{\mathcal{B}}(M)$ in the sense of \cite{CGP:AdmissibleSkein}.
We refer the reader to this reference for details. For our purposes, it will suffice to know that $\SS_{\mathcal{B}}(M)$ is a vector space over $\mathbbm{k}$ spanned by $\mathcal{B}$-colored ribbon graphs in $M$ such that there is at least one strand labeled by a projective object in every connected component of $M$. Two such $\mathcal{B}$-colored ribbon graphs are equivalent if they are isotopic. Additionally, they are subject to certain local admissible skein relations, see \cite[Sect. 2.1]{CGP:AdmissibleSkein}.

While we will refrain for recalling all the details of admissible skein theory. We wish to point out a notable subtlety:\ A closed diagram labeled by objects and morphisms in $\mathcal{B}$ can be interpreted in two different ways.
As a concrete example, consider a circle
$$
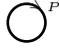
\begin{tikzpicture}[baseline=-5pt,scale=0.25]
\draw[usual] (0,0) arc(180:540:1) node[pos = 0.66, sloped]{\tiny $>$} node[pos = 0.66,right]{\tiny $P$};
\end{tikzpicture}$$ colored by a projective object $P$ in $\mathcal{B}$. The first way to interpret this diagram is as a morphism in $\mathcal{B}$, in which case we would use the standard graphical calculus of ribbon tensor categories to compute the corresponding morphism in $\mathcal{B}$.
The second way to interpret the above picture is as an admissible skein in $S^3$. These two notions are drastically different. For instance, provided that $\mathcal{B}$ is not semisimple, we have
$$\begin{tikzpicture}[baseline=-5pt,scale=0.25]
\draw[usual] (0,0) arc(180:540:1) node[pos = 0.66, sloped]{\tiny $>$} node[pos = 0.66,right]{\tiny $P$};
\end{tikzpicture} \raisebox{3pt}{$= 0 \in \End_{\mathcal{B}}(\mathbbm 1)\,.$}$$
On the other hand, it is very often the case that the corresponding admissible skein in $S^3$ does not evaluate to zero, that is
$$\begin{tikzpicture}[baseline=-5pt,scale=0.25]
\draw[usual, blue] (0,0) arc(180:540:1) node[pos = 0.66, sloped]{\tiny $>$} node[pos = 0.66,right]{\tiny $P$};
\end{tikzpicture} \raisebox{3pt}{$\neq 0 \in \SS_{\mathcal{B}}(S^3)\,.$}$$
In order to distinguish morphisms in $\mathcal{B}$ from admissible skeins in $S^3$, the later will be depicted in blue.

We now explain how to upgrade admissible skein modules to a partially defined TQFT.
Firstly, it follows from their construction that admissible skeins are functorial with respect to diffeomorphisms.
Then, the possible values for the evaluation of the 4-handle correspond precisely to linear maps
$$\widetilde{\mathsf{t}}: \SS_{\mathcal{B}}(S^3)\to \mathbbm{k}\,.$$ 
But, it was proven in \cite[Cor. 3.2]{CGP:AdmissibleSkein} that such linear maps are in bijective correspondence with \emph{modified traces} on $\mathcal{B}$. Recall that a modified trace $\mathsf{t}$ on $\mathcal{B}$ is a trace defined on the ideal of projective objects $\mathrm{Proj}$ of $\mathcal{B}$ satisfying various axioms \cite[Def. 3.2.1]{GKP11}.
In particular, the modified trace is compatible with left and right partial traces in the following sense:
For any projective object $P$, any object $X$, and any morphism $f\in\mathrm{End}(P\otimes X)$, using $\mathsf{pTr}^R(f)\in\mathrm{End}(P)$ to denote the right partial trace of $f$, we have
\begin{equation}\label{eq:modifiedpartial}
\mathsf{t}(f) = \mathsf{t}\big(\mathsf{pTr}^R(f)\big)\,.
\end{equation}
The modified trace $\mathsf t$ yields an invariant of admissible skeins in $S^3$, by cutting open a given admissible skein in $S^3$ along any strand labeled by a projective object and then applying the modified trace on the resulting endomorphism:
\begin{equation}\label{eq:4handlemtr}
    \widetilde{\mathsf t} \left( \begin{tikzpicture}[baseline = -5pt, scale = 0.3]
    \draw[usual, blue] (0,0) arc(0:360:1 and 2) node[midway, left, yshift=0.1cm]{\small $P$} node[midway, sloped]{\tiny $>$};
    \node[draw, rectangle, blue, fill=white] at (0,0) {\small \textcolor{blue}{T}};
\end{tikzpicture} \right) 
:=
\mathsf t  \left(\begin{tikzpicture}[baseline = -5pt, scale = 0.3]
    \draw[usual] (0,-2) --(0,2) node[pos = 0.9, left]{\small $P$} node[pos = 0.9, sloped]{\tiny $>$};
    \node[draw, rectangle, fill=white] at (0,0) {\small T};
\end{tikzpicture} \right)\,.
\end{equation}
Summarizing the above discussion, we have seen that a ribbon tensor category with a chosen modified trace on its ideal of projective objects induces a functor $$\SS: \HD_4^4\to \Vect\,,$$ which maps a 3-manifold to its admissible skein module, a diffeomorphism $e_d$ to the induced linear map, and a 4-handle $e_\Sp$ to the linear map $\widetilde{\mathsf{t}}$.

In order to be able to assign a value to 3-handles, we need to recall a notion from \cite{CGHPM}, to which we refer the reader for details. Let us now assume that $\mathcal{B}$ is finite and unimodular in the sense of \cite{EGNO}. Under these assumptions, the modified trace on the ideal of projective objects is unique up to scalar \cite{GKP}.
In particular, for any projective object $P\in\mathcal{B}$, the modified trace induces a pairing $$\mathsf{t}: \Hom(P, \mathbbm{1}) \otimes \Hom(\mathbbm{1}, P) \to \mathbbm{k}\,,$$ which is non-degenerate provided that $\mathsf{t}$ is nonzero \cite[Cor. 5.6]{GKP}.

\begin{Definition}\label{def:cutting}
Assume that the modified trace $\mathsf{t}$ on the ideal of projective objects of $\mathcal{B}$ is nonzero. The \emph{cutting morphism} $\Lambda_P$ based at a projective object $P\in\mathcal{B}$, is $$\Lambda_P = \sum_i x_i\circ x^i \in \End(P)\,,$$ where $(x^i)_i\in\Hom(P, \mathbbm{1})$ and $(x_i)_i\in\Hom(\mathbbm{1}, P)$ are dual bases with respect to $\mathsf{t}$
\end{Definition}

We record for later use the following property of the cutting morphism.

\begin{Lemma}\label{lem:dimHom}
For any projective object $P$ in $\mathcal{B}$, we have that $$\mathsf t (\Lambda_P)  = \dim \Hom (P,\mathbbm{1})\,.$$
\end{Lemma}
\begin{proof}
We have that $$\mathsf t (\Lambda_P) = \mathsf t (\sum_i x_i\circ x^i) = \sum_i 1 = \dim \Hom (P,\mathbbm{1})\,,$$ as $\{x_i\}$ and $\{x^i\}$ are dual bases of $\Hom(\mathbbm{1}, P)$ and $\Hom(P,\mathbbm{1})$ with respect to the pairing $\mathsf{t}$. 
\end{proof}

\noindent The cutting morphism is used to define $\SS$ on 3-handles by cutting skeins going through the attaching $S^2$. Representing the attaching $S^2$ by surgery on an unknot (which we draw with a dot), we have the following graphical depiction of the 3-handle: 
\begin{equation}\label{eq:3handleCutting}
    \begin{tikzpicture}[anchorbase,scale=0.25]
\draw[usual,crossline, blue] (0,5) -- (0,8);
\draw[usual,crossline] (-1.5,5) to[out=90,in=90] (1.5,5) node{\small $\bullet$} to[out=270,in=270] (-1.5,5);
\draw[usual, crossline, blue] (0,2) -- (0,5) node[pos = 0.3, right]{\small $P$} ;
\end{tikzpicture} \mapsto
    \begin{tikzpicture}[anchorbase,scale=0.25]
\draw[usual,crossline, blue] (0,2) -- (0,8);
\node[rectangle, draw, lJW] at (0,5) {\small $\Lambda_P$};
\end{tikzpicture}\,. 
\end{equation}

In order to evaluate 2-handles, we need to recall another notion from \cite{CGHPM}. Given that we have assumed in particular that $\mathcal{B}$ is a finite tensor category, it admits a projective generator $G$.

\begin{Definition}\label{def:chromatic}
A \emph{chromatic morphism} $\mathsf{c}_P$ based at a projective object $P\in\mathcal{B}$ is an endomorphism 
$$\mathsf{c}_P \in \End(G\otimes P)$$
satisfying 
\begin{equation}\label{eq:chromatic}
\begin{tikzpicture}[anchorbase,scale=0.3,tinynodes]
\draw[lJW] (-2.5,3) rectangle (2.5,5);
\node at (0,4) {$\Lambda_{V\otimes G^*}$};
\draw[cJW] (3.5,3) rectangle (8.5,5);
\node at (6,3.9) {\normalsize $\mathsf{c}_P$};
\draw[usual] (2,5) to[out=90,in=180] (3,6)node[above,xshift=0.2cm,yshift=-0.15cm]{$G$} to[out=0,in=90] (4,5);
\draw[usual] (-1.5,5) to (-1.5,7)node[left,yshift=-0.15cm]{$V$};
\draw[usual] (7.5,5) to (7.5,7)node[right,yshift=-0.15cm]{$P$};
\draw[usual] (2,3) to[out=270,in=180] (3,2)node[below,xshift=0.15cm,yshift=0cm]{$G$} to[out=0,in=270] (4,3);
\draw[usual] (-1.5,3) to (-1.5,1)node[left,xshift=-0.00cm,yshift=-0.00cm]{$V$};
\draw[usual] (7.5,3) to (7.5,1)node[right,xshift=0.0cm,yshift=0.00cm]{$P$};
\end{tikzpicture}
\quad \ \  = \quad
\begin{tikzpicture}[anchorbase,scale=0.25,tinynodes]
\draw[usual] (-1.5,1) to (-1.5,7)node[left,yshift=-0.15cm]{$V$};
\draw[usual] (1.5,1) to (1.5,7)node[right,yshift=-0.15cm]{$P$};
\end{tikzpicture}\,,
\end{equation} 
for any projective object $V$ of $\mathcal{B}$.
\end{Definition}

\noindent We emphasize that the chromatic morphism is not unique. However, all possible choices lead to the same invariant as was shown in \cite[Lemma 2.3]{CGHPM}.

The chromatic morphism is used to define $\SS$ on 2-handles by adding the belt sphere as a skein labeled ``red", which is turned to an actual skein using the following red-to-blue operation:
\begin{equation}\label{eq:redtoblue}
\begin{tikzpicture}[anchorbase,tinynodes,scale=0.3]
\draw[usual, red] (2,5) to[out=90,in=180] (3,6)node[above,xshift=0.2cm,yshift=-0cm]{\normalsize red} to[out=0,in=90] (4,5) to (4,3);
\draw[usual, blue] (5.5,1) to (5.5,7)node[right,xshift=-0.1cm,yshift=-0.15cm]{$P$};
\draw[usual, red] (2,3) to[out=270,in=180] (3,2)to[out=0,in=270] (4,3);
\draw[usual, red, dashed] (2,3) to (2,5);
\end{tikzpicture}
\quad \ \  := \quad
\begin{tikzpicture}[anchorbase,scale=0.3,tinynodes]
\draw[cJW] (3.5,3) rectangle (6.5,5);
\node at (5,3.9) {\normalsize $\mathsf{c}_P$};
\draw[usual, blue] (2,5) to[out=90,in=180] (3,6)node[above,xshift=0.2cm,yshift=-0.12cm]{$G$} to[out=0,in=90] (4,5);
\draw[usual, blue] (5.5,5) to (5.5,7)node[right,yshift=-0.15cm]{$P$};
\draw[usual, blue] (2,3) to[out=270,in=180] (3,2) to[out=0,in=270] (4,3);
\draw[usual, blue] (5.5,3) to (5.5,1)node[right,xshift=0.0cm,yshift=0.00cm]{$P$};
\draw[usual, dashed, blue] (2,3) to (2,5);
\end{tikzpicture}\,.
\end{equation} 

\noindent The chromatic morphism generalizes the Kirby color associated to a modular fusion category that arises in the Reshetikhin--Turaev construction \cite{RT}.
More generally, the above chromatic morphisms is closely related to Lyubashenko's integral on the canonical coend in an arbitrary finite braided tensor category. The chromatic morphism has the advantage of not necessitating the identification of the canonical coend, which can be difficult in practice.

\begin{Definition}\label{def:chrnondegen}
A unimodular finite ribbon tensor category $\mathcal B$ is said to be \textit{chromatic non-degenerate} if the morphism
$$\Delta_0^{P_\mathbbm{1}}=
\begin{tikzpicture}[anchorbase,scale=0.25]
\draw[usual,crossline] (0,5) -- (0,8);
\draw[usual,crossline, red] (-1.5,5) to[out=90,in=90] (1.5,5) node[right=-2pt]{\small red} to[out=270,in=270] (-1.5,5);
\draw[usual, crossline] (0,2) -- (0,5)  node[pos = 0.3, right]{\small $P_\mathbbm{1}$};
\end{tikzpicture} =\begin{tikzpicture}[anchorbase,scale=0.25]
\draw[usual,crossline] (0,5) -- (0,8);
\draw[usual,crossline] (-1.5,6) to[out=90,in=90] (2,6) -- (2,4) node[pos = 0.2, sloped]{\small $>$} node[pos = 0.2, above right]{$G$}  to[out=270,in=270] (-1.5,4);
\draw[usual, crossline] (0,2) -- (0,5)  node[pos = 0, right]{\small $P_\mathbbm{1}$};
\draw[cJW] (-2,4) rectangle (1,6);
\node at (-0.5, 5) {\normalsize $\mathsf{c}_{P_\mathbbm{1}}$};
\end{tikzpicture}  \in \End_{\mathcal{B}}(P_\mathbbm{1})$$
    is non-zero. 
\end{Definition}

Summarizing the above discussion, the following result holds.

\begin{Theorem}[\cite{CGHPM}]\label{thm:partialCGHPM}
Let $\mathcal{B}$ be a ribbon tensor category with a chosen non-zero modified trace on its ideal of projectives. Then the admissible skein module functor $\SS_{\mathcal{B}}: \HD_4^4\to \Vect$ extends to 
\begin{enumerate}
    \item A functor $\SS_{\mathcal{B}}: \mathrm{\HD}_4^{2,3,4}\to \Vect$ as soon as $\mathcal{B}$ is finite and unimodular,
    \item A functor $\SS_{\mathcal{B}}: \cob_4^{1,\dots,4} \simeq \HD_4^{1,\dots,4}\to \Vect$ as soon as $\mathcal{B}$ is chromatic non-degenerate,
    \item A functor $\SS_{\mathcal{B}}: \cob_4 \simeq \HD_4^{0,\dots,4}\to \Vect$ as soon as $\mathcal{B}$ is chromatic compact.
\end{enumerate}
\end{Theorem}
\begin{proof}
    The first item is not stated as such in \cite{CGHPM}. However, the fact that the 2-, 3-, and 4-handle attachments are well-defined and satisfy the relations not involving 1-handles is checked under the stated assumption in \cite[Propositions 5.2, 5.3 and Theorem 5.4]{CGHPM}.
\end{proof}

\begin{Remark}
    It is shown in \cite[Section 4]{H:WRTdelCY} that these partially-defined TQFTs can be made once-extended, i.e.\ defined on a 2-category of cobordisms with corners. In this setting, there is a unique extension from $\HD_4^{k+1,\dots,4}$ to $\HD_4^{k,\dots,4}$, when it exists \cite[Section 6]{H:handle}, so the extensions above are indeed unique in an appropriate sense.
\end{Remark}

\subsection{Invariants of 4-Dimensional 2-Handlebodies}\label{sub:invariant}

We now review how to compute explicitly the invariant of 4-dimensional 2-handlebodies up to 2-equivalence obtained from the construction outlined in the previous section.
In the literature, it is customary to consider a 2-handlebody as being obtained from 0-, 1-, and 2-handles. Such handles can be turned into 2-, 3-, and 4-handles by taking the opposite of a given Morse function. Specifically, let $W$ be a connected 4-dimensional 2-handlebody. One can present $W$ by a certain number of 1-handles, represented as dotted (trivial) circles, and of 2-handles, represented as plain links, in Akbulut's dotted notation, see e.g. \cite{AkbulutKirbyMazur}. A relevant example is given in Figure \ref{fig:2hbInv}.

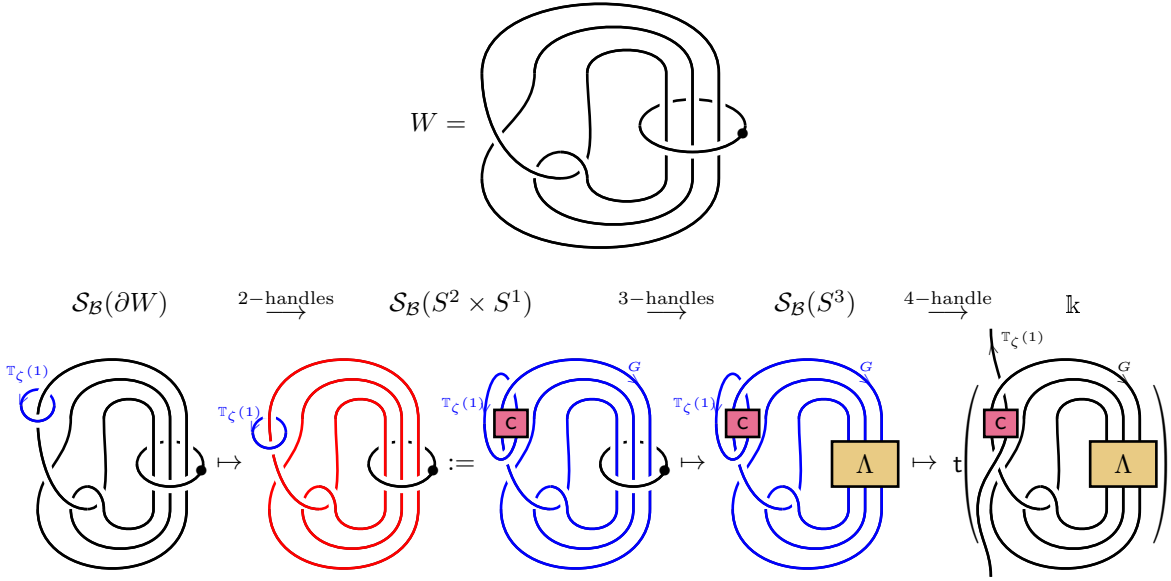
\begin{figure}[h]
    \centering $W=$
\begin{tikzpicture}[anchorbase,scale=0.7]
\draw[usual] (3,0) arc (180:0:1 and 0.5); 
\draw[usual] (0,-1) to[out=90,in=270] (1,1);
\draw[usual] (1.5,-1) to[out = 0, in=270] (2,1);
\draw[usual] (1.5,-0.5) to[out = 180, in=90] (1,-1);
\draw[usual, crossline] (1.5,-0.5) to[out = 0, in=90] (2,-1);
\draw[usual, crossline](1.5,-1) to[out = 180, in=270] (0,1);
\draw[usual, crossline] (0,1) to[out = 90, in=90] (4.5,1) -- (4.5,-1) to[out=270, in=270] (0,-1);
\draw[usual, crossline] (1,1) to[out = 90, in=90] (4,1) -- (4,-1) to[out=270, in=270] (1,-1);
\draw[usual, crossline] (2,1) to[out = 90, in=90] (3.5,1) -- (3.5,-1) to[out=270, in=270] (2,-1);
\draw[usual, crossline] (3,0) arc (-180:0:1 and 0.5) node[pos = 0.9]{$\bullet$}; 
\end{tikzpicture}
\begin{tikzpicture}[anchorbase,xscale=0.44, yscale = 0.6]
    \begin{scope}[xshift = 0cm]

    \node at (2.5,3.55) {$\SS_{\mathcal{B}}(\partial W)$};
    \node at (7.5,3.55) {$\overset{\mathrm{2-handles}}\longrightarrow$};    
    \node at (5.8,0) {$\mapsto$};
\draw[usual, blue] (-0.5,1.3) arc (180:0:0.5 and 0.35) node[pos = 0.1, sloped]{\tiny $<$}node[pos = 0.3, above = 0pt]{\tiny $\mathbb{T}_{\zeta}(1)$}; 
\draw[usual] (3,0) arc (180:0:1 and 0.5); 
\draw[usual] (0,-1) to[out=90,in=270] (1,1);
\draw[usual] (1.5,-1) to[out = 0, in=270] (2,1);
\draw[usual] (1.5,-0.5) to[out = 180, in=90] (1,-1);
\draw[usual, crossline] (1.5,-0.5) to[out = 0, in=90] (2,-1);
\draw[usual, crossline](1.5,-1) to[out = 180, in=270] (0,1);
\draw[usual, crossline] (0,1) to[out = 90, in=90] (4.5,1) -- (4.5,-1) to[out=270, in=270] (0,-1);
\draw[usual, crossline] (1,1) to[out = 90, in=90] (4,1) -- (4,-1) to[out=270, in=270] (1,-1);
\draw[usual, crossline] (2,1) to[out = 90, in=90] (3.5,1) -- (3.5,-1) to[out=270, in=270] (2,-1);
\draw[usual, crossline] (3,0) arc (-180:0:1 and 0.5) node[pos = 0.9]{$\bullet$}; 
\draw[usual, crossline, blue] (-0.5,1.3) arc (-180:0:0.5 and 0.35); 
    \end{scope}    
    \begin{scope}[xshift = 7cm]
    \node at (5.8,3.55) {$\SS_{\mathcal{B}}(S^2\times S^1)$};
    \node at (12,3.55) {$\overset{\mathrm{3-handles}}\longrightarrow$};
    \node at (5.8,0) {$:=$};
\draw[usual, blue] (-0.5,0.7) arc (180:0:0.5 and 0.35) node[pos = 0.1, sloped]{\tiny $<$}; 
\draw[usual] (3,0) arc (180:0:1 and 0.5); 
\draw[usual, red] (0,-1) to[out=90,in=270] (1,1);
\draw[usual, red] (1.5,-1) to[out = 0, in=270] (2,1);
\draw[usual, red] (1.5,-0.5) to[out = 180, in=90] (1,-1);
\draw[usual, red, crossline] (1.5,-0.5) to[out = 0, in=90] (2,-1);
\draw[usual, red, crossline](1.5,-1) to[out = 180, in=270] (0,1);
\draw[usual, red, crossline] (0,1) to[out = 90, in=90] (4.5,1) -- (4.5,-1) to[out=270, in=270] (0,-1);
\draw[usual, red, crossline] (1,1) to[out = 90, in=90] (4,1) -- (4,-1) to[out=270, in=270] (1,-1);
\draw[usual, red, crossline] (2,1) to[out = 90, in=90] (3.5,1) -- (3.5,-1) to[out=270, in=270] (2,-1);
\draw[usual, crossline] (3,0) arc (-180:0:1 and 0.5) node[pos = 0.9]{$\bullet$}; 
\draw[usual, crossline, blue] (-0.5,0.7) arc (-180:0:0.5 and 0.35)node[pos = 0.01, above = 7pt, left = -7pt]{\tiny $\mathbb{T}_{\zeta}(1)$}; 
    \end{scope}
    \begin{scope}[xshift = 14cm]
    \node at (5.8,0) {$\mapsto$};
\draw[usual, blue] (-0.5,0.9) arc (180:0:0.5 and 1.1) node[pos = 0.1, sloped]{\tiny $<$}; 
\draw[usual] (3,0) arc (180:0:1 and 0.5); 
\draw[usual, blue] (0,-1) to[out=90,in=270] (1,1);
\draw[usual, blue] (1.5,-1) to[out = 0, in=270] (2,1);
\draw[usual, blue] (1.5,-0.5) to[out = 180, in=90] (1,-1);
\draw[usual, blue, crossline] (1.5,-0.5) to[out = 0, in=90] (2,-1);
\draw[usual, blue, crossline](1.5,-1) to[out = 180, in=270] (0,1);
\draw[usual, blue, crossline] (0,1) to[out = 90, in=90]node[pos = 0.8, sloped]{\tiny $>$}node[pos = 0.8, above = 0pt]{\tiny $G$} (4.5,1) -- (4.5,-1) to[out=270, in=270] (0,-1) ;
\draw[usual, blue, crossline] (1,1) to[out = 90, in=90] (4,1) -- (4,-1) to[out=270, in=270] (1,-1);
\draw[usual, blue, crossline] (2,1) to[out = 90, in=90] (3.5,1) -- (3.5,-1) to[out=270, in=270] (2,-1);
\draw[usual, crossline] (3,0) arc (-180:0:1 and 0.5) node[pos = 0.9]{$\bullet$}; 
\draw[usual, crossline, blue] (-0.5,0.9) arc (-180:0:0.5 and 0.8)node[pos = 0.01, above = 7pt, left = -4pt]{\tiny $\mathbb{T}_{\zeta}(1)$}; 
\draw[cJW] (-0.2,0.6) rectangle (0.8,1.2);
\node at (0.3,0.9) {\normalsize $\mathsf{c}$};
    \end{scope} 
    \begin{scope}[xshift = 21cm]
    \node at (2.5,3.55) {$\SS_{\mathcal{B}}(S^3)$};
    \node at (6.5,3.55) {$\overset{\mathrm{4-handle}}\longrightarrow$};
    \node at (5.8,0) {$\mapsto$};
\draw[usual, blue] (-0.5,0.9) arc (180:0:0.5 and 1.1) node[pos = 0.1, sloped]{\tiny $<$}; 
\draw[usual, blue] (0,-1) to[out=90,in=270] (1,1);
\draw[usual, blue] (1.5,-1) to[out = 0, in=270] (2,1);
\draw[usual, blue] (1.5,-0.5) to[out = 180, in=90] (1,-1);
\draw[usual, blue, crossline] (1.5,-0.5) to[out = 0, in=90] (2,-1);
\draw[usual, blue, crossline](1.5,-1) to[out = 180, in=270] (0,1);
\draw[usual, blue, crossline] (0,1) to[out = 90, in=90]node[pos = 0.8, sloped]{\tiny $>$}node[pos = 0.8, above = 0pt]{\tiny $G$} (4.5,1) -- (4.5,-1) to[out=270, in=270] (0,-1) ;
\draw[usual, blue, crossline] (1,1) to[out = 90, in=90] (4,1) -- (4,-1) to[out=270, in=270] (1,-1);
\draw[usual, blue, crossline] (2,1) to[out = 90, in=90] (3.5,1) -- (3.5,-1) to[out=270, in=270] (2,-1);
\draw[usual, crossline, blue] (-0.5,0.9) arc (-180:0:0.5 and 0.8)node[pos = 0.01, above = 7pt, left = -4pt]{\tiny $\mathbb{T}_{\zeta}(1)$}; 
\draw[cJW] (-0.2,0.6) rectangle (0.8,1.2);
\node at (0.3,0.9) {\normalsize $\mathsf{c}$};
\draw[lJW] (3,-0.5) rectangle (5,0.5);
\node at (4,0) {$\Lambda$};
    \end{scope} 
    \begin{scope}[xshift = 28.8cm]
    \node at (2.5,3.55) {$\mathbbm{k}$};
\draw[usual] (0.5,0.9) to[out=90,in=270] node[pos = 0.8, sloped]{\tiny $<$}node[pos = 0.9, right]{\tiny $\mathbb{T}_{\zeta}(1)$} (0,3); 
\draw[usual] (0,-1) to[out=90,in=270] (1,1);
\draw[usual] (1.5,-1) to[out = 0, in=270] (2,1);
\draw[usual] (1.5,-0.5) to[out = 180, in=90] (1,-1);
\draw[usual, crossline] (1.5,-0.5) to[out = 0, in=90] (2,-1);
\draw[usual, crossline](1.5,-1) to[out = 180, in=270] (0,1);
\draw[usual, crossline] (0,1) to[out = 90, in=90]node[pos = 0.8, sloped]{\tiny $>$}node[pos = 0.8, above = 0pt]{\tiny $G$} (4.5,1) -- (4.5,-1) to[out=270, in=270] (0,-1) ;
\draw[usual, crossline] (1,1) to[out = 90, in=90] (4,1) -- (4,-1) to[out=270, in=270] (1,-1);
\draw[usual, crossline] (2,1) to[out = 90, in=90] (3.5,1) -- (3.5,-1) to[out=270, in=270] (2,-1);
\draw[usual, crossline] (0.5,0.9) to[out=270,in=90] (-0.4,-1) to[out=270,in=90] (0,-2.5); 
\draw[cJW] (-0.2,0.6) rectangle (0.8,1.2);
\node at (0.3,0.9) {\normalsize $\mathsf{c}$};
\draw[lJW] (3,-0.5) rectangle (5,0.5);
\node at (4,0) {$\Lambda$};
\node at (-1,0){$\mathsf{t}$};
\node[yscale = 2] at (-0.6,0){$\Bigg(\,$};
\node[yscale = 2] at (5.2,0){$\Bigg)$};
    \end{scope} 
\end{tikzpicture}
\caption{Top:\ A 2-handlebody in Akbulut's dotted notation. 
Bottom:\ The associated invariant. We start from an arbitrary admissible skein, drawn in blue, in the boundary of $W$. We get a skein in a connected sum of $S^2\times S^1$ by adding the 2-handles of $W$ in red and using Equation \eqref{eq:redtoblue}. We then obtain a skein in $S^3$ by way of Equation \eqref{eq:3handleCutting}, which can be evaluated to a scalar using Equation \eqref{eq:4handlemtr}.
}
    \label{fig:2hbInv}
\end{figure}

We shall consider the 4-dimensional 2-handlebody $W$ as a cobordism $-W:\partial W \to \emptyset$. We can then view the plain links as belt spheres for 2-handles, and the dotted circles as representing attaching spheres for the 3-handles. The construction of \cite{CGHPM} recalled succinctly in the previous section then provides us with a linear map $\SS_{\mathcal{B}}(-W): \SS_{\mathcal{B}}(\partial W) \to \mathbbm{k}$. Now, a skein $\Gamma$ in $\partial W$ can be represented as a ribbon graph disjoint from the presentation of $W$. The invariant $\SS_{\mathcal{B}}(-W)(\Gamma)$ is then obtained by replacing every plain component of the presentation of $W$ by a red coloured skein, which is turned into an admissible skein using the red-to-blue operation, then cutting all the strands going through a dotted circle using the cutting morphism, and finally taking the modified trace of the resulting skein in $S^3$. An example of this procedure is recorded in Figure \ref{fig:2hbInv}.

A priori, this invariant depends on a parametrization $\partial W \simeq M$ of the boundary
and on a skein $\Gamma$ in $M$. If one changes the parametrization by a diffeomorphism $d:M\to M'$, the invariant is modified by transporting the skein $\Gamma$ under $d$ in the following fashion: $$\SS_{\mathcal{B}}(-W)(\Gamma) = \SS_{\mathcal{B}}(W \circ e_{d^{-1}})(d(\Gamma))\,.$$
So as not to have to worry about such parametrizations, it is convenient to consider the skein
\begin{equation}\label{eq:Gamma0}
\Gamma_0 :=
\begin{tikzpicture}[anchorbase,scale=0.3,tinynodes]
\draw[lJW] (-2,3) rectangle (2,5);
\node at (0,4) {$\Lambda_{P_{\mathbbm{1}}}$};
\draw[usual, blue] (0,5) to[out=90,in=90] (4,5);
\draw[usual, blue] (0,3) to[out=270,in=270] (4,3) to (4,5) node[right, blue]{$P_{\mathbbm{1}}$};
\end{tikzpicture}
\end{equation}
inserted in a small enough ball in $M$. Here $P_{\mathbbm{1}}$ denotes the projective cover of the unit $\mathbbm{1}$ of $\mathcal{B}$. The skein $\Gamma_0$ is manifestly invariant under diffeomorphisms of the boundary. This motivates the next definition.

\begin{Definition}
Let $\mathcal{B}$ be a finite unimodular ribbon tensor category, and let $W$ be a 4-dimensional 2-handlebody. We denote $$\Inv_{\mathcal{B}}(W) := \SS_{\mathcal{B}}(-W)(\Gamma_0) \in \mathbbm{k}$$ the invariant associated to $W$.
\end{Definition}

\section{Temperley-Lieb at Height Two and a Fourth Root of Unity}\label{sec:algebra}

We now fix an algebraically closed field $\mathbbm{k}$ of positive characteristic $p$ along with $\zeta$ a root of unity in $\mathbbm{k}$. We begin by briefly reviewing general properties of the Temperley-Lieb category $\mathbf{TL}^{\zeta}$ from \cite{STWZ}. In particular, a choice of square root $\zeta^{1/2}$ endows $\mathbf{TL}^{\zeta}$ with a ribbon structure, and we write $\mathbf{TL}^{\zeta^{1/2}}$ for the corresponding ribbon category. Additionally, the ribbon category $\mathbf{TL}^{\zeta^{1/2}}$ admits a sequence of tensor ideals $\mathbf{J}_{p^{(n)}}$ indexed by the integer $n$, which we think of as a height parameter.
After having recalled these general facts, we focus on the case when $\zeta$ is a primitive fourth root of unity, and we describe explicitly the ribbon structure of the quotient $\mathbf{TL}^{\zeta^{1/2}}\!\!/\mathbf{J}_{p^{(2)}}$. In the language of \cite{KL}, we study the recoupling theory of $\mathbf{TL}^{\zeta^{1/2}}$ at height $n=2$ and a primitive fourth root of unity.

\subsection{Mixed Temperley-Lieb Categories}

Let $\mathbbm{k}$ be an algebraically closed field of positive characteristic $p$, and let $\zeta$ be a root of unity in $\mathbbm{k}$ of order $N>2$. The Temperley-Lieb category, denoted by $\TL^{\zeta}$, is the $\mathbbm{k}$-linear category whose objects are non-negative integers, and whose morphisms are planar tangle diagrams modulo isotopies and the relation
$$\begin{tikzpicture}[anchorbase,scale=0.25,tinynodes]
\draw[usual] (0,0) circle (1);
\end{tikzpicture} =-(\zeta+\zeta^{-1})\in \mathrm{End}(\mathbbm{1})=\mathbbm{k}\,.$$
Composition in $\TL^{\zeta}$ is given by vertical stacking of diagrams. The category $\TL^{\zeta}$ admits a monoidal structure given by horizontal stacking.
Upon choosing a square root $\zeta^{\frac12}$ for $\zeta$, the monoidal category $\TL^{\zeta}$ can be endowed with a compatible braiding using Kauffman's bracket rule: 
$$\begin{tikzpicture}[anchorbase,scale=0.25,tinynodes]
\draw[usual] (2,0) to (0,2);
\draw[usual,crossline] (0,0) to (2,2);
\end{tikzpicture}
=
\zeta^{\frac12}\cdot
\begin{tikzpicture}[anchorbase,scale=0.25,tinynodes]
\draw[usual] (0,0) to (0,2);
\draw[usual] (2,0) to (2,2);
\end{tikzpicture}
+
\zeta^{-\frac12}\cdot
\begin{tikzpicture}[anchorbase,scale=0.25,tinynodes]
\draw[usual] (0,0) to[out=90,in=180] (1,0.75) to[out=0,in=90] (2,0);
\draw[usual] (0,2) to[out=270,in=180] (1,1.25) to[out=0,in=270] (2,2);
\end{tikzpicture} \,.$$
We shall write $\TL^{\zeta^{1/2}}$ for the Temperley-Lieb category $\TL^{\zeta}$ equipped with the above choice of braiding corresponding to $\zeta^{\frac12}$. Crucially for our purposes, the category $\TL^{\zeta^{1/2}}$ is not only braided,  but also admits a canonical ribbon structure.

One key feature of the Temperley-Lieb category $\TL^{\zeta}$ is that it contains distinguished idempotents \cite{STWZ}. 
More precisely, associated to every integer $v \geq 1$, there is an idempotent $\mathrm{E}_{v-1}$ on $v-1$ strands. We will represent this idempotent graphically by $$\mathrm{E}_{v-1}=  \begin{tikzpicture}[anchorbase,scale=0.25]
\draw[pJW] (-2,-1) rectangle (2,1);
\node at (0.2,-0.1) {\tiny $\pjwm[v{-}1]$};
\end{tikzpicture}\,.$$
We find it instructive to consider some examples.

\begin{Example}
Let $\zeta$ be a primitive fourth root of unity. In this case, we have:
$$\mathrm{E}_{2} = \begin{tikzpicture}[anchorbase,scale=0.25,tinynodes]
\draw[usual] (0,0) to (0,2);
\draw[usual] (2,0) to (2,2);
\end{tikzpicture}\,,\quad \mathrm{E}_{3} = \begin{tikzpicture}[anchorbase,scale=0.25,tinynodes]
\draw[usual] (0,0) to (0,2);
\draw[usual] (2,0) to (2,2);
\draw[usual] (4,0) to (4,2);
\end{tikzpicture}
-
\begin{tikzpicture}[anchorbase,scale=0.25,tinynodes]
\draw[usual] (0,0) to[out=90,in=90] (2,0);
\draw[usual] (4,2) to[out=270,in=270] (2,2);
\draw[usual] (4,0) to[out=90,in=270] (0,2);
\end{tikzpicture}
-
\begin{tikzpicture}[anchorbase,scale=0.25,tinynodes]
\draw[usual] (4,0) to[out=90,in=90] (2,0);
\draw[usual] (0,2) to[out=270,in=270] (2,2);
\draw[usual] (0,0) to[out=90,in=270] (4,2);
\end{tikzpicture}\,,\quad
\mathrm{E}_{4}= \begin{tikzpicture}[anchorbase,scale=0.25,tinynodes]
\draw[usual] (0,0) to (0,2);
\draw[usual] (2,0) to (2,2);
\draw[usual] (4,0) to (4,2);
\draw[usual] (6,0) to (6,2);
\end{tikzpicture}
-
\begin{tikzpicture}[anchorbase,scale=0.25,tinynodes]
\draw[usual] (0,0) to[out=90,in=90] (2,0);
\draw[usual] (4,2) to[out=270,in=270] (2,2);
\draw[usual] (4,0) to[out=90,in=270] (0,2);
\draw[usual] (6,0) to (6,2);
\end{tikzpicture}
-
\begin{tikzpicture}[anchorbase,scale=0.25,tinynodes]
\draw[usual] (4,0) to[out=90,in=90] (2,0);
\draw[usual] (0,2) to[out=270,in=270] (2,2);
\draw[usual] (0,0) to[out=90,in=270] (4,2);
\draw[usual] (6,0) to (6,2);
\end{tikzpicture}\,.$$
In particular, the idempotent $\mathrm{E}_{v-1}$ is not necessarily equal to its left-right mirror $\mathrm{E}_{v-1}^{\leftrightarrow}$. The two idempotents $\mathrm{E}_{v-1}$ and $\mathrm{E}_{v-1}^{\leftrightarrow}$ are nonetheless always isomorphic.
\end{Example}

Every object of $\TL^{\zeta}$ is rigid, in the sense that it has duals. The evaluation and coevaluation morphisms are given by the familiar cap and cup morphisms.
The dual of an idempotent in $\TL^{\zeta}$ is given graphically by its 180 degree rotation. It follows from the construction given in \cite{STWZ} that the idempotents $\mathrm{E}_{v-1}$ are their own up-down mirror, that is, we have $\mathrm{E}_{v-1} = (\mathrm{E}_{v-1})^\updownarrow$. However, they are not in general their own left-right mirror. Since 180 degree rotation is the composition of these two mirror operations, we obtain that the dual of $\mathrm{E}_{v-1}$ is its left-right mirror, which we denote by
$$(\mathrm{E}_{v-1})^* = (\mathrm{E}_{v-1})^\leftrightarrow = 
\begin{tikzpicture}[anchorbase,scale=0.25,tinynodes]
\draw[pJWl] (-2,-1) rectangle (2,1);
\node at (-0.1,-0.2) {\tiny $\pjwm[ v{-}1]$};
\end{tikzpicture} \ .$$
For completeness, we note that $\mathrm{E}_{v-1}$ is always isomorphic to $(\mathrm{E}_{v-1})^*$, but the explicit isomorphism can be difficult to write down, see \cite[Lemma 3.16]{STWZ}.

The Temperley-Lieb category $\TL^{\zeta}$ is closely related to the category $\Tilt^{\zeta}$ of titling modules for Lusztig's divided power quantum group for $\mathfrak{sl}_2$ at root of unity $\zeta$. This relation is discussed in detail in \cite[Section 4.B]{STWZ}, to which we refer the reader for details. Succinctly, the category $\Tilt^{\zeta}$ is a highest weight category, whose indecomposable objects are the tilting modules $\mathrm{T}_{\zeta}(v-1)$ with $v$ a positive integer. Moreover, the canonical monoidal functor
$$\TL^{\zeta}\rightarrow\Tilt^{\zeta}$$ determined by sending $\mathrm{E}_1$ to $\mathrm{T}_{\zeta}(1)$ exhibits $\Tilt^{\zeta}$ as the Cauchy completion of $\TL^{\zeta}$, that is, every object of $\Tilt^{\zeta}$ arises as a direct sum of idempotents in $\TL^{\zeta}$. Here, we have more specifically that the idempotent $\mathrm{E}_{v-1}$ is identified with the indecomposable tilting module $\mathrm{T}_{\zeta}(v-1)$.
As a consequence, a choice of square root $\zeta^{1/2}$ endows that category $\Tilt^{\zeta}$ with an essentially unique ribbon structure compatible with the canonical functor $\TL^{\zeta}\rightarrow\Tilt^{\zeta}$. We write $\Tilt^{\zeta^{1/2}}$ for the corresponding ribbon category.

We now review the notion of $p\ell$-adic expansion, which plays a key role in the structure of the category $\mathrm{T}_{\zeta}(v-1)$ as was observed in \cite{STWZ}.
Recall that $N$ denotes the order of the root of unity $\zeta$. We set $\ell$ to be equal to $N/2$ if $N$ is even and to $N$ otherwise. For any integer $k\geq 1$, we define $p^{(k)} := \ell p^{k-1}$.
Any positive integer $v$ can be uniquely expressed as a sum $v= \sum_{i=0}^k a_i p^{(i)}$ with $0\leq a_0\leq \ell-1$, $0\leq a_i\leq p-1$ for $i\geq 1$ and $a_k \neq 0$. If this is the case, we write $v = [a_k,\dots,a_0]_{p,\ell}$.
For any positive integer $n$, one can consider the full subcategory $$\mathbf{J}_{p^{(n)}} = \{\mathrm{T}_{\zeta}(v-1)| v\geq p^{(n)}\}^{\oplus}\subseteq \Tilt^{\zeta}.$$
Said differently, we have that $\mathbf{J}_{p^{(n)}}$ is the full additive subcategory generated by the indecomposable tilting modules $\mathrm{T}_{\zeta}(v-1)$ for $v = [a_k,\dots,a_0]_{p,\ell}$ with $k\geq n$.
It was shown in \cite[Theorem 5.1]{STWZ} that $\mathbf{J}_{p^{(n)}}$ is a thick tensor ideal in $\Tilt^{\zeta}$ and that every tensor ideal in $\Tilt^{\zeta}$ is of this form.

For specific values of $v$, called Eves, the corresponding idempotent $\mathrm{E}_{v-1}$ enjoys additional properties that are reminiscent of those of the Jones-Wenzl projectors in characteristic zero. More precisely, if we fix $v = [a_j,0,\dots,0]_{p,\ell}$ with $a_j\neq 0$, it follows from \cite[Equation (3-8)]{STWZ} that $\mathrm{E}_{v-1}$ is its own left-right mirror, in which case we will denote it by 
\begin{equation}\label{eq:left=righteves}\begin{tikzpicture}[anchorbase,scale=0.25,tinynodes]
\draw[pJW] (-2,-1) rectangle (2,1);
\node at (0,-0.2) {\tiny $\pjwm[v{-}1]$};
\end{tikzpicture} =
\begin{tikzpicture}[anchorbase,scale=0.25,tinynodes]
\draw[JW] (-2,-1) rectangle (2,1);
\node at (0,-0.2) {\tiny $\qjwm[v{-}1]$};
\end{tikzpicture} =
\begin{tikzpicture}[anchorbase,scale=0.25,tinynodes]
\draw[pJWl] (-2,-1) rectangle (2,1);
\node at (-0.1,-0.2) {\tiny $\pjwm[v{-}1]$};
\end{tikzpicture}\ .\end{equation}
Moreover, it follows from \cite[Theorem 3.19 and Equation (2-3)]{STWZ} that $\mathrm{E}_{v-1}$ is annihilated by cupping or capping any two adjacent strands, which yields a simple formula for the interaction of such idempotents with the braiding of any two adjacent strands:
\begin{equation}\label{eq:classicalbraidingabsorption}
\begin{tikzpicture}[anchorbase,scale=0.25,tinynodes]
\draw[JW] (0,0) rectangle (4,1);
\node at (2,0.33) {$\pjwm[{v{-}1}]$};
\draw[usual] (1.5,1) arc(180:0:0.7 and 1.5);
\end{tikzpicture}
=0\,, \quad \quad
\begin{tikzpicture}[anchorbase,scale=0.25,tinynodes]
\draw[usual] (3,1) to[out=90,in=270] (1,3);
\draw[usual,crossline] (1,1) to[out=90,in=270] (3,3);
\draw[JW] (0,0) rectangle (4,1);
\node at (2,0.33) {$\pjwm[{v{-}1}]$};
\end{tikzpicture}
= \zeta^{\frac12}\cdot
\begin{tikzpicture}[anchorbase,scale=0.25,tinynodes]
\draw[JW] (0,0) rectangle (4,1);
\node at (2,0.33) {$\pjwm[{v{-}1}]$};
\end{tikzpicture}\ .
\end{equation}

In the next section, we will apply the construction of \cite{CGHPM} to the quotient $\mathbf{Tilt}^{\zeta^{1/2}}\!\!/\mathbf{J}_{p^{(n)}}$ for some fixed natural number $n\geq 1$. In order to do so, it is most convenient to embed $\mathbf{Tilt}^{\zeta^{1/2}}\!\!/\mathbf{J}_{p^{(n)}}$ into a finite tensor category. As explained in \cite[Section 5.A]{STWZ}, this can be achieved using the theory of abelianization developed in \cite{BEO}. We write $\Ver^{\zeta}_{p^{(n)}}$ for the finite tensor category given by the abelianization of $\mathbf{Tilt}^{\zeta}\!/\mathbf{J}_{p^{(n)}}$. By construction, there is a monoidal functor $U:\mathbf{Tilt}^{\zeta}\!/\mathbf{J}_{p^{(n)}}\rightarrow \Ver^{\zeta}_{p^{(n)}}$ inducing an equivalence between the ideal of projective objects in $\Ver^{\zeta}_{p^{(n)}}$ and $\mathbf{J}_{p^{(n-1)}}/\mathbf{J}_{p^{(n)}}$. In fact, as explained in \cite{Dec}, the functor $U$ is fully faithful, and therefore has a universal property. It follows that the ribbon structure on $\mathbf{Tilt}^{\zeta}\!/\mathbf{J}_{p^{(n)}}$ induced by a choice of $\zeta^{1/2}$ yields a compatible ribbon structure on $\Ver^{\zeta}_{p^{(n)}}$. We write $\Ver^{\zeta^{1/2}}_{p^{(n)}}$ for the corresponding finite ribbon tensor category. Many general properties of these ribbon tensor categories were obtained in \cite{Dec}, to which we refer the curious reader for details.

Finally, we shall write $\Tz(v-1)$ for the image under $U$ of the indecomposable tilting modules $\mathrm{T}_{\zeta}(v-1)$ with $1\leq v \leq p^{(n)}-1$. In particular, all the objects $\mathrm{T}_{\zeta}(v-1)$ are self-dual, so that $\Ver^{\zeta^{1/2}}_{p^{(n)}}$ is unimodular.
Moreover, the indecomposable projective objects of $\Ver^{\zeta^{1/2}}_{p^{(n)}}$ are $\Tz(v-1)$ with $p^{(n-1)}\leq v \leq p^{(n)}-1$. We note that the projective cover of the unit $\mathbbm{1} = \Tz(0)$ is $P_{\mathbbm{1}} = \Tz(2p^{(n-1)}-2)$. For later use, let us also point out that $\Tz(v-1)$ is a simple object of $\Ver^{\zeta^{1/2}}_{p^{(n)}}$ precisely if $v$ is an Eve, that is, $v = [a_j,0,\dots,0]_{p,\ell}$ for some $a_j\neq 0$ and $0\leq j\leq n-1$.

\subsection{Linear Structure of the Ideal of Projective Objects}\label{sub:linearstructure}

We now fix $\zeta$ to be a primitive fourth root of unity. In this case, we observe that
$$\begin{tikzpicture}[anchorbase,scale=0.25,tinynodes]
\draw[usual] (0,0) circle (1);
\end{tikzpicture} =-(\zeta+\zeta^{-1}) = 0 \in \mathrm{End}(\mathbbm{1})=\mathbbm{k}\,.$$
We will investigate the ribbon structure of the ideal of projective objects of $\Ver^{\zeta^{1/2}}_{p^{(2)}}$, that is, of the category $\mathbf{J}_{p^{(1)}}/\mathbf{J}_{p^{(2)}}$. We begin by reviewing its linear structure from \cite{STWZ}.

The indecomposable projective objects of $\Ver^{\zeta^{1/2}}_{p^{(2)}}$ are 
$$\mathbb{T}_{\zeta}(1),\,\mathbb{T}_{\zeta}(2) = P_{\mathbbm1},\,\mathbb{T}_{\zeta}(3),\dots,\,\mathbb{T}_{\zeta}(2p-2).$$ 
The indecomposable objects $\Tz(v-1)$ for $v=[a_1,0]_{p,\ell}$, i.e.\ the idempotents on an odd number of strands, are simple. In particular, for $v=[a_1,0]_{p,\ell}$, we have
$$\Hom(\Tz(v-1), \Tz(w-1)) = \left\{ \begin{array}{ll}
\mathbbm{k} \mathrm{E}_{v-1}\,, & \text{if } v=w\,,\\
0\,, & \text{otherwise}\,.
\end{array} \right.$$
Above, $\mathrm{E}_{v-1}$ denotes the idempotent defining $\Tz(v-1)$, which we identify with the identity morphism of $\Tz(v-1)$.

The objects $\Tz(v-1)$ for $v=[a_1,1]_{p,\ell}$ are not simple. For any such $v$, we can consider the following morphisms in $\Tilt^{\zeta}$:
\begin{equation*}
\mathrm{U}^0_{v-1} = 
\begin{tikzpicture}[anchorbase,scale=0.4,tinynodes]
\ptru{2.4}{1}{$\mathrm{U}^0_{\pjwm[v{-}1]}$}{0}{0};
\end{tikzpicture}
=
\begin{tikzpicture}[anchorbase,xscale=0.5, yscale = -0.5]
\draw[pJW] (-0.5,-5.5) rectangle (1.6,-6.2);
\node at (0.6,-5.8) {\tiny $\pjwm[v{+}1]$};
\draw[usual] (0.7,-5.5) to[out=90,in=180] (1.05,-5) to[out=0,in=90] (1.4,-5.5); 
\draw[usual] (-0.2,-5.5) --++(0,0.6);
\draw[usual] (0.4,-5.5) --++(0,0.6);
\node[scale = 0.5] at (0.1,-5.2){\tiny $\dots$};
\end{tikzpicture}
\,, \quad \quad
    \mathrm{D}^0_{v-1} = 
    \begin{tikzpicture}[anchorbase,scale=0.4,tinynodes]
\ptrd{2.4}{1}{$\mathrm{D}^0_{\pjwm[v{-}1]}$}{0}{0};
\end{tikzpicture}
=
\begin{tikzpicture}[anchorbase,scale=0.5]
\draw[pJW] (-0.5,-5.5) rectangle (1.6,-6.2);
\node at (0.6,-5.9) {\tiny $\pjwm[v{-}1]$};
\draw[usual] (0.7,-5.5) to[out=90,in=180] (1.05,-5) to[out=0,in=90] (1.4,-5.5); 
\draw[usual] (-0.2,-5.5) --++(0,0.6);
\draw[usual] (0.4,-5.5) --++(0,0.6);
\node[scale = 0.5] at (0.1,-5.2){\tiny $\dots$};
\end{tikzpicture}
\,, \quad \quad
\mathrm{L}_{v-2}^0
=
\begin{tikzpicture}[anchorbase,scale=0.4,tinynodes]
\ptr{2.4}{0.5}{$\mathrm{L}^0_{\pjwm[v{-}1]}$}{0}{0};
\end{tikzpicture}
=
\begin{tikzpicture}[anchorbase,yscale=0.4,xscale = 0.55,tinynodes]
\draw[pJW] (-0.5,-3.5) rectangle (1.6,-4.5);
\node at (0.6,-4.1) {\tiny $\pjwm[v{-}1]$};
\draw[pJW] (-0.5,-5.5) rectangle (1.6,-6.5);
\node at (0.6,-6.1) {\tiny $\pjwm[v{-}1]$};
\draw[usual] (0.6,-5.5) to[out=90,in=180] (0.95,-5.15) to[out=0,in=90] (1.3,-5.5); 
\draw[usual] (0.6,-4.5) to[out=270,in=180] (0.95,-4.85) to[out=0,in=270] (1.3,-4.5); 
\draw[usual] (-0.25,-4.5) to (-0.25,-5.5);
\draw[usual] (0.25,-4.5) to (0.25,-5.5);
\node[scale = 0.5] at (0,-5){\tiny $\dots$};
\end{tikzpicture}\,.
\end{equation*}
It follows from \cite[Theorem 3.25]{STWZ} that that the remaining spaces of morphisms between the indecomposable projective objects of $\Ver^{\zeta^{1/2}}_{p^{(2)}}$ are given, for any $w=[b_1,b_0]_{p,\ell}$, by
$$\Hom(\Tz(v-1), \Tz(w-1)) = \left\{ \begin{array}{ll}
\mathbbm{k} \mathrm{U}^0_{v-1}\,, & \text{if } w= v+2\,,\\
\mathbbm{k} \mathrm{E}_{v-1} \oplus \mathbbm{k} \mathrm{L}^0_{v-1}\,, & \text{if } w= v\,,\\
\mathbbm{k} \mathrm{D}^0_{v-1}\,, & \text{if } w=v-2\,,\\
0\,, & \text{otherwise}\,.
\end{array} \right.$$
The composition rules for these morphisms were also obtained in \cite[Theorem 3.25]{STWZ}, and we record these below
\begin{gather*}
 \mathrm{U}^0_{v-3} \circ \mathrm{D}^0_{v-1}  =     \begin{tikzpicture}[anchorbase,scale=0.4,tinynodes]
\ptrd{2.4}{1}{$\mathrm{D}_{\pjwm[v{-}1]}^0$}{0}{0};
\ptru{2.4}{1}{$\mathrm{U}_{\pjwm[v{-}1]}^0$}{0}{0};
\end{tikzpicture}
=
\begin{tikzpicture}[anchorbase,scale=0.4,tinynodes]
\ptr{2.4}{0.5}{$\mathrm{L}_{\pjwm[v{-}1]}^0$}{0}{0};
\end{tikzpicture}
= \mathrm{L}_{v-1}^0\,,
\\\mathrm{U}_{v+1}^0 \circ \mathrm{U}_{v-1}^0 =0\,,\quad \mathrm{D}_{v-3}^0 \circ \mathrm{D}_{v-1}^0 = 0\,,
\end{gather*}
\begin{equation}\label{eq:partialtrace1strand}
    \mathrm{D}_{v-1}^0 \circ \mathrm{U}_{v-3}^0 =
\begin{tikzpicture}[anchorbase,scale=0.4,tinynodes]
\draw[usual] (2,-6.5) to (2,-5.5);
\draw[usual] (1,-5.5) to[out=90,in=180] (1.5,-5.15) 
to[out=0,in=90] (2,-5.5); 
\draw[usual] (1,-6.5) to[out=270,in=180] (1.5,-6.85) to[out=0,in=270] (2,-6.5); 
\draw[pJW] (-0.5,-5.5) rectangle (2.3,-6.5);
\node at (0.9,-6.1) {\tiny $\pjwm[v{-}1]$};
\end{tikzpicture}
= -\frac{a_1}{a_1-1}\cdot 
\begin{tikzpicture}[anchorbase,scale=0.4,tinynodes]
\ptr{2.4}{0.5}{$\mathrm{L}_{\pjwm[v{-}1]}^0$}{0}{0};
\end{tikzpicture}\,,\quad\textrm{provided that }a_1>1\,.
\end{equation}
In particular, the morphism $\mathrm{L}_{v-1}^0$ is nilpotent of order 2.

\subsection{Categorified Fusion Rules with the Generator}\label{sub:categorifiedfusion}

We now consider the fusion rules of an indecomposable projective object of $\Ver^{\zeta^{1/2}}_{p^{(2)}}$ together with the generator $\mathbb{T}_{\zeta}(1)$. For any $0<v<2p-1$, these are given by
$$\Tz(v-1)\otimes\Tz(1) \cong \left\{\begin{array}{ll}
\Tz(v)\,, & \text{if } v = [a_1,0]_{p,\ell}\,,\\
\Tz(v) \oplus \Tz(v-2) \oplus \Tz(v-2) \oplus \Tz(v-4)\,,   & \text{if } v = [a_1,1]_{p,\ell}\,,
\end{array} \right.$$
where the term $\Tz(v)$ in the second expression is omitted if $v=2p-1$, and the term $\Tz(v-4)$ is omitted if $v=3$.

For our subsequent purposes, it will be crucial to exhibit genuine morphisms witnessing the above decomposition into indecomposable summands. This is more subtle in the second case given that there are two summands isomorphic to $\Tz(v-2)$. The problem of describing such categorified fusion rules was tackled in \cite[Theorem 4.15]{STWZ} of which we record a corrected special case.

\begin{Definition}
Let $v=[a_1,1]_{p,\ell}$. We define endomorphisms of the tensor product $\Tz(v-1)\otimes\Tz(1)$ by
$${\mathrm{A}}^{0}_{v} = \begin{tikzpicture}[anchorbase,scale=0.4,tinynodes]
\draw[pJW] (-0.5,-1.5) rectangle (3.2,-2.5);
\node at (1.35,-2.1) {$\pjwm[v{-}1]$};
\draw[pJW] (-0.5,-3.5) rectangle (2.5,-4.5);
\node at (1,-4.1) {$\pjwm[v{-}2]$};
\draw[pJW] (-0.5,-5.5) rectangle (3.2,-6.5);
\node at (1.35,-6.1) {$\pjwm[v{-}1]$};
\draw[usual] (2,-2.5) to[out=270,in=180] (2.5,-2.85) 
to[out=0,in=270] (3,-2.5);
\draw[usual] (2,-3.5) to[out=90,in=180] (2.5,-3.2) 
to (3,-3.2) to [out=0,in=270] (3.5,-1.5);
\draw[usual] (2,-4.5) to (2,-5.5);
\draw[usual] (2.5,-5.5) to[out=90,in=180] (3,-5.15) 
to[out=0,in=90] (3.5,-5.5) to (3.5,-6.5);
\draw[usual] (0.5,-3.5) to (0.5,-2.5);
\draw[usual] (0.5,-4.5) to (0.5,-5.5);
\end{tikzpicture} \,, \quad \quad
\tilde{\mathrm{A}}^{0}_{v} = \begin{tikzpicture}[anchorbase,scale=0.4,tinynodes, yscale=-1]
\draw[pJW] (-0.5,-1.5) rectangle (3.2,-2.5);
\node at (1.35,-1.9) {$\pjwm[v{-}1]$};
\draw[pJW] (-0.5,-3.5) rectangle (2.5,-4.5);
\node at (1,-3.9) {$\pjwm[v{-}2]$};
\draw[pJW] (-0.5,-5.5) rectangle (3.2,-6.5);
\node at (1.35,-5.9) {$\pjwm[v{-}1]$};
\draw[usual] (2,-2.5) to[out=270,in=180] (2.5,-2.85) 
to[out=0,in=270] (3,-2.5);
\draw[usual] (2,-3.5) to[out=90,in=180] (2.5,-3.2) 
to (3,-3.2) to [out=0,in=270] (3.5,-1.5);
\draw[usual] (2,-4.5) to (2,-5.5);
\draw[usual] (2.5,-5.5) to[out=90,in=180] (3,-5.15) 
to[out=0,in=90] (3.5,-5.5) to (3.5,-6.5);
\draw[usual] (0.5,-3.5) to (0.5,-2.5);
\draw[usual] (0.5,-4.5) to (0.5,-5.5);
\end{tikzpicture}\,,\quad\quad\fusidem{v}{1}=
-\frac{a_{1}-1}{a_{1}}\cdot
\begin{tikzpicture}[anchorbase,scale=0.4,tinynodes]
\draw[pJW] (-0.5,-1.5) rectangle (2.5,-2.5);
\node at (1,-2.1) {$\pjwm[v{-}2]$};
\draw[pJW] (-0.5,-4.5) rectangle (2.5,-5.5);
\node at (1,-5.1) {$\pjwm[v{-}2]$};
\draw[pJW] (-0.5,-3.25) rectangle (0.5,-3.75);
\draw[usual] (2,-2.5) to[out=270,in=180] (2.5,-3) to[out=0,in=270] 
(3,-2.5) to (3,-1.5) node[right,xshift=-2pt,yshift=-2pt]{$2$};
\draw[usual] (2,-4.5) to[out=90,in=180] (2.5,-4) to[out=0,in=90] 
(3,-4.5) to (3,-5.5) node[right,xshift=-2pt,yshift=1pt]{$2$};
\draw[usual] (0,-2.5) to (0,-3.25);
\draw[usual] (0,-3.75) to (0,-4.5);
\end{tikzpicture}
\,.$$
\end{Definition}

\begin{Proposition}\label{prop:categorifiedfusionT1}
Let $v=[a_1,a_0]_{p,\ell}$.
\begin{enumerate}
\item If $a_0 = 0$, we have $$\begin{tikzpicture}[anchorbase,scale=0.25,tinynodes]
\draw[JW] (-2,-1) rectangle (2,1);
\node at (0,-0.2) {$\pjwm[v{-}1]$};
\draw[usual] (2.5,-1) to (2.5,1);
\end{tikzpicture}
=
\begin{tikzpicture}[anchorbase,scale=0.25,tinynodes]
\draw[pJW] (-2,-1) rectangle (2,1);
\node at (0,-0.2) {$\pjwm[v]$};
\end{tikzpicture}\ .$$

\item If $a_0 = 1$, we have  the following decomposition into orthogonal idempotents  $$\begin{tikzpicture}[anchorbase,scale=0.25,tinynodes]
\draw[pJW] (-2,-1) rectangle (2,1);
\node at (0,-0.2) {$\pjwm[v{-}1]$};
\draw[usual] (2.7,-1) to (2.7,1);
\end{tikzpicture}
=
\mathrm{E}_v + \fusidema{v}{0}+\tilde{\mathrm{A}}^{0}_{v} + \fusidem{v}{1}\,,$$
where $\mathrm{E}_v$ is omitted if $v=[p-1,1]_{p,\ell} = 2p-1$, and $\mathrm{B}^1_v$ is omitted if $v=[1,1]_{p,\ell} = 3$.
The morphisms ${\mathrm{A}}^{0}_{v}$ and $\tilde{\mathrm{A}}^{0}_{v}=\big(\fusidema{v}{0}\big)^{\updownarrow}$ are orthogonal idempotents projecting onto the two direct summands isomorphic to $\mathbb{T}_{\zeta}(v-2)$, and the morphism $\mathrm{B}^1_v$ is an idempotent projecting onto the summand isomorphic to $\mathbb{T}_{\zeta}(v-4)$.
\end{enumerate}
\end{Proposition}
\begin{proof}
As already mentioned above, this is essentially \cite[Theorem 4.15]{STWZ}. Although their statement has to be slightly corrected.
Namely, using the morphism $\mathrm{Y}^0_v$ introduced below in Equation \eqref{eq:defYZ}, they claim that ${\mathrm{A}}^{0}_{v} + \mathrm{Y}^0_v$ and $\tilde{\mathrm{A}}^{0}_{v} + \mathrm{Y}^0_v$ are orthogonal, whereas we assert that it is really ${\mathrm{A}}^{0}_{v}$ and $\tilde{\mathrm{A}}^{0}_{v}=\big(\fusidema{v}{0}\big)^{\updownarrow}$ that are orthogonal.
This follows easily from the following formulas, which were derived in the proof of \cite[Lemma 4.21]{STWZ} and that can also be checked directly using Equation \eqref{eq:partialtrace1strand}: $${\mathrm{A}}^{0}_{v}{\mathrm{A}}^{0}_{v} = {\mathrm{A}}^{0}_{v},\ {\mathrm{A}}^{0}_{v}{\mathrm{Y}}^{0}_{v} = {\mathrm{Y}}^{0}_{v},\  {\mathrm{Y}}^{0}_{v}{\mathrm{A}}^{0}_{v} = 0,\  {\mathrm{Y}}^{0}_{v}{\mathrm{Y}}^{0}_{v} = 0, \ {\mathrm{A}}^{0}_{v}({\mathrm{A}}^{0}_{v})^{\updownarrow}=0,\ ({\mathrm{A}}^{0}_{v})^{\updownarrow}{\mathrm{A}}^{0}_{v}=0.$$
\end{proof}

\begin{Notation}\label{not:basis}
We presently fix a basis of $\End(\mathbb{T}_{\zeta}(v-1)\otimes\mathbb{T}_{\zeta}(1))$ for any $v=[a_1,a_0]_{p,\ell}$. \\
If $a_0 =1$, we have $$\End(\mathbb{T}_{\zeta}(v-1)\otimes\mathbb{T}_{\zeta}(1))\cong \End(\mathbb{T}_{\zeta}(v)) = \mathbbm{k} \mathrm{E}_{v} \oplus \mathbbm{k} \mathrm{L}^0_{v}\,.$$ Here $\mathrm{E}_{v}$ is the identity moprhism and $\mathrm{L}^0_{v}$ is a nilpotent endomorphism.

\noindent If $a_0 = 1$, then we have $$\End(\mathbb{T}_{\zeta}(v-1)\otimes\mathbb{T}_{\zeta}(1))\cong \End(\mathbb{T}_{\zeta}(v))\oplus \End(\mathbb{T}_{\zeta}(v-2)^{\oplus 2})\oplus \End(\mathbb{T}_{\zeta}(v-4))\,.$$ 
As $a_1 = 1$, the indecomposable objects $\mathbb{T}_{\zeta}(v)$, $\mathbb{T}_{\zeta}(v-2)$, and $\mathbb{T}_{\zeta}(v-4)$ are simple, so that the first algebra is simply generated by $\mathrm{E}_v$, and the third one by $\mathrm{B}^1_v$. Further, the second algebra can be identified with the algebra of $(2\times 2)$-matrices in $\mathbbm{k}$. The two diagonal terms are $\mathrm{A}^{0}_{v}$ and $\tilde{\mathrm{A}}^{0}_{v}$.
The two off-diagonal terms may be identified with the following two morphisms
\begin{equation}\label{eq:canonicalbasisT1}
\mathrm{Y}^0_v:\tilde{\mathrm{A}}^{0}_{v}\rightarrow \mathrm{A}^{0}_{v}\,,\quad\mathrm{Z}^0_v:\mathrm{A}^{0}_{v}\rightarrow \tilde{\mathrm{A}}^{0}_{v}\,,
\end{equation}
defined graphically by
\begin{equation}\label{eq:defYZ}
\mathrm{Y}^0_v = \begin{tikzpicture}[anchorbase,scale=0.4,tinynodes]
\draw[pJW] (-0.5,-1.5) rectangle (3.2,-2.5);
\node at (1.35,-2.1) {$\pjwm[v{-}1]$};
\draw[pJW] (-0.5,-3.5) rectangle (2.5,-4.5);
\node at (1,-4.1) {$\pjwm[v{-}2]$};
\draw[pJW] (-0.5,-5.5) rectangle (3.2,-6.5);
\node at (1.35,-6.1) {$\pjwm[v{-}1]$};
\draw[usual] (2,-2.5) to[out=270,in=180] (2.5,-2.85) 
to[out=0,in=270] (3,-2.5);
\draw[usual] (2,-3.5) to[out=90,in=180] (2.5,-3.2) 
to (3,-3.2) to [out=0,in=270] (3.5,-1.5);
\draw[usual] (2,-4.5) to[out=270,in=180] (2.5,-4.8) 
to (3,-4.8) to [out=0,in=90] (3.5,-6.5);
\draw[usual] (2,-5.5) to[out=90,in=180] (2.5,-5.15) 
to[out=0,in=90] (3,-5.5);
\draw[usual] (0.5,-3.5) to (0.5,-2.5);
\draw[usual] (0.5,-4.5) to (0.5,-5.5);
\end{tikzpicture}\,,\quad\quad 
\mathrm{Z}^0_v := \begin{tikzpicture}[anchorbase,scale=0.4,tinynodes]
\draw[pJW] (-0.5,-1.5) rectangle (3.2,-2.5);
\node at (1.35,-2.1) {$\pjwm[v{-}1]$};
\draw[pJW] (-0.5,-3.5) rectangle (2.5,-4.5);
\node at (1,-4.1) {$\pjwm[v{-}2]$};
\draw[pJW] (-0.5,-5.5) rectangle (3.2,-6.5);
\node at (1.35,-6.1) {$\pjwm[v{-}1]$};
\draw[usual] (2.5,-2.5) to[out=270,in=180] (3,-2.85) 
to[out=0,in=270] (3.5,-2.5) to (3.5,-1.5);
\draw[usual] (2,-3.5) to (2,-2.5);
\draw[usual] (2,-4.5) to (2,-5.5);
\draw[usual] (2.5,-5.5) to[out=90,in=180] (3,-5.15) 
to[out=0,in=90] (3.5,-5.5) to (3.5,-6.5); 
\draw[usual] (0.5,-3.5) to (0.5,-2.5);
\draw[usual] (0.5,-4.5) to (0.5,-5.5);
\end{tikzpicture}\,.
\end{equation}
These four morphisms satisfy the relations $\mathrm{Y}^0_v\mathrm{Z}^0_v = \mathrm{A}^0_v, \ \mathrm{Z}^0_v\mathrm{Y}^0_v=\tilde{\mathrm{A}}^0_v$, and $\mathrm{Y}^0_v\mathrm{Z}^0_v=\mathrm{Z}^0_v\mathrm{Y}^0_v=0$.
\end{Notation}
\begin{Lemma}\label{lem:YisLtensorId1}
    Let $v=[a_1,1]_{p,\ell}$, then $\mathrm{Y}^0_v = \mathrm{L}^0_{v-1}\otimes \mathrm{E}_1 = $ 
\begin{tikzpicture}[anchorbase,scale=0.25,tinynodes]
\ptr{4}{1}{$\mathrm{L}^0_{\pjwm[v{-}2]}$}{0}{0};
\draw[usual] (1,1) to (1,-1);
\end{tikzpicture} .
\end{Lemma}
\begin{proof}
    The difference between the left hand-side and the right hand-side lies the middle idempotent factor $\mathrm{E}_{v-2}$ that is present in the definition of $\mathrm{Y}^0_v$. The result is immediate for $v=3$ and follows by induction and Proposition \ref{prop:categorifiedfusionT1} by noticing that the terms $\mathrm{A}^0, \tilde{\mathrm{A}}^0$, and $\mathrm{B}^1$ do not contribute.
\end{proof}

\subsection{Double Braiding with the Generator}

We describe the double braiding $\beta^2$ on $\mathbb{T}_{\zeta}(v-1)\otimes\mathbb{T}_{\zeta}(1)$ using the categorified fusion rules of Section \ref{sub:categorifiedfusion}. To avoid unnecessarily cluttered notations, we will simply denote these morphisms using $\beta_{v-1,1}$. Firstly, by definition, we have that $$\beta^2_{1,1} = \zeta\cdot \mathrm{E}_2 + 2\cdot \mathrm{L}_2^0.$$ Note that $\mathbb{T}_{\zeta}(1) = \mathbb{T}_{\zeta}([1,0]_{p,\ell}-1)$, so this case corresponds to $v = [1,0]_{p,\ell}$.

One can also check by hand that $$\beta^2_{2,1} = -\mathrm{E}_3 + \mathrm{A}^0_{3} + \tilde{\mathrm{A}}^0_{3} +2\zeta\cdot\mathrm{Y}^0_{3}\,.$$ 
We have used the basis introduced above in Notation \ref{not:basis}. This is the case $v=[1,1]_{p,\ell}$.

\begin{Proposition}\label{prop:braidingT1}
Let $v=[a_1,a_0]_{p,\ell}$.
\begin{enumerate}
\item If $a_0 = 0$, we have $$\beta^2_{v-1,1} = \zeta^{v-1}\cdot \mathrm{E}_v + 2\zeta^{v-2}\cdot \mathrm{L}_v^0\,.$$
\item If $a_0 = 1$, we have $$\beta^2_{v-1,1}=\zeta^{v-1}\cdot\mathrm{E}_v+ \zeta^{v-3}\cdot\mathrm{A}_{v}^0 + \zeta^{v-3}\cdot\tilde{\mathrm{A}}_{v}^0 + 2\zeta^{v-2}\cdot \mathrm{Y}^0_{v} + \zeta^{v-1}\cdot \mathrm{B}^1_v\,,$$
where the first term is omitted if $v=[p-1,1]_{p,\ell}$, and the last term is omitted if $v=[1,1]_{p,\ell}=3$.
\end{enumerate}
\end{Proposition}
\begin{proof}
We begin by considering the case $a_0 = 0$. By expanding the bottom-most and the top-most crossings in the picture below, we obtain
\begin{equation*}
\begin{tikzpicture}[anchorbase,scale=0.4,tinynodes, yscale=1]
\node[circle, fill=gray!30] at (1.75,1.8) {\ \ };
\node[circle, fill=gray!30] at (1.75,4.1) {\ \ };
\draw[usual] (2.5,0) -- (2.5,1) to[out=90,in=-90] (-0.5,3); 
\draw[usual, crossline] (0.25,1)--++(0,4);
\draw[usual, crossline] (1,1)--++(0,4) node[midway, left=0pt, scale = 0.5]{\tiny $\dots$};
\draw[usual, crossline] (1.75,1)--++(0,4);
\draw[usual, crossline] (-0.5,3) to[out=90,in=-90] (2.5,5)--(2.5,6); 
\draw[JW] (0,0) rectangle (2,1);
\node at (1,0.45) {$\pjwm[v{-}1]$};
\draw[JW] (0,5) rectangle (2,6);
\node at (1,5.45) {$\pjwm[v{-}1]$};
\end{tikzpicture}
= \zeta \cdot
\begin{tikzpicture}[anchorbase,scale=0.4,tinynodes, yscale=1]
\draw[usual] (2.5,0) -- (2.5,6); 
\draw[usual](1.75,1) to[out=90,in=-90] (-0.5,3); 
\draw[usual, crossline] (0.25,1)--++(0,4);
\draw[usual, crossline] (1,1)--++(0,4) node[midway, left=0pt, scale = 0.5]{\tiny $\dots$};
\draw[usual, crossline] (-0.5,3) to[out=90,in=-90] (1.75,5); 
\draw[JW] (0,0) rectangle (2,1);
\node at (1,0.45) {$\pjwm[v{-}1]$};
\draw[JW] (0,5) rectangle (2,6);
\node at (1,5.45) {$\pjwm[v{-}1]$};
\end{tikzpicture}
+
\begin{tikzpicture}[anchorbase,scale=0.4,tinynodes, yscale=1]
\draw[usual] (-0.5,3) to[out=-90,in=-90] (2.5,5) -- (2.5,6); 
\draw[usual](1.75,1) to[out=90,in=90] (2.5,1) -- (2.5,0); 
\draw[usual, crossline] (0.25,1)--++(0,4);
\draw[usual, crossline] (1,1)--++(0,4) node[midway, left=0pt, scale = 0.5]{\tiny $\dots$};
\draw[usual, crossline] (-0.5,3) to[out=90,in=-90] (1.75,5); 
\draw[JW] (0,0) rectangle (2,1);
\node at (1,0.45) {$\pjwm[v{-}1]$};
\draw[JW] (0,5) rectangle (2,6);
\node at (1,5.45) {$\pjwm[v{-}1]$};
\end{tikzpicture}
+
\begin{tikzpicture}[anchorbase,scale=0.4,tinynodes, yscale=1]
\draw[usual](1.75,5) to[out=-90,in=-90] (2.5,5) -- (2.5,6); 
\draw[usual] (-0.5,3) to[out=-90,in=90] (1.75,1); 
\draw[usual, crossline] (0.25,1)--++(0,4);
\draw[usual, crossline] (1,1)--++(0,4) node[midway, left=0pt, scale = 0.5]{\tiny $\dots$};
\draw[usual, crossline] (-0.5,3) to[out=90,in=90] (2.5,1) -- (2.5,0); 
\draw[JW] (0,0) rectangle (2,1);
\node at (1,0.45) {$\pjwm[v{-}1]$};
\draw[JW] (0,5) rectangle (2,6);
\node at (1,5.45) {$\pjwm[v{-}1]$};
\end{tikzpicture}
+ \zeta^{-1}\cdot
\begin{tikzpicture}[anchorbase,scale=0.4,tinynodes, yscale=1]
\draw[usual](1.75,5) to[out=-90,in=-90] (2.5,5) -- (2.5,6); 
\draw[usual](1.75,1) to[out=90,in=90] (2.5,1) -- (2.5,0); 
\draw[usual] (-0.5,3) to[out=-90,in=-90] (2.5,3); 
\draw[usual, crossline] (0.25,1)--++(0,4);
\draw[usual, crossline] (1,1)--++(0,4) node[midway, left=0pt, scale = 0.5]{\tiny $\dots$};
\draw[usual, crossline] (-0.5,3) to[out=90,in=90] (2.5,3); 
\draw[JW] (0,0) rectangle (2,1);
\node at (1,0.45) {$\pjwm[v{-}1]$};
\draw[JW] (0,5) rectangle (2,6);
\node at (1,5.45) {$\pjwm[v{-}1]$};
\end{tikzpicture}\,.
\end{equation*}
Moreover, it follows from \cite[Lemma 5.10]{STWZ}, that the last term vanishes.
Namely, the loop on a single strand around $\mathrm{E}_{v-2}$ evaluates to $-[2]_{\zeta^{v-1}} \cdot \mathrm{E}_{v-2} + \varsigma\cdot \mathrm{L}_{v-2}$ for some scalar $\varsigma$. But we have $[2]_{\zeta^{v-1}} = 0$, and the second term vanishes because of the presence of the idempotent $\mathrm{E}_{v-1}$.
Now, as $v = [a_1,0]_{p,\ell}$, we can appeal to Equation \eqref{eq:classicalbraidingabsorption} so as to compute the remaining terms. 
By counting the number of crossings, we therefore find that 
$$\beta^2_{v-1,1} = \zeta^{v-1}\cdot \mathrm{E}_v + 2\zeta^{v-2}\cdot \mathrm{L}_v^0 + 0\,,$$ as desired.

Let us now consider the case $a_0 = 1$. Writing $\mathrm{E}_{v-1} = \mathrm{E}_{v-2} \otimes \mathrm{E}_{1}$ and expanding the crossings involving $\mathrm{E}_{1}$, we find
\begin{align*}\beta^2_{v-1,1} = &\ \zeta\cdot (\beta^2_{v-2,1}\otimes \mathrm{E}_{1}) \\ &\ + (\mathrm{E}_{v-2}\otimes\mathrm{L}_2^0)\circ (\beta^2_{v-2,1}\otimes \mathrm{E}_{1})\\ &\ + (\beta^2_{v-2,1}\otimes \mathrm{E}_{1})\circ (\mathrm{E}_{v-2}\otimes\mathrm{L}_2^0)\\ &\ + \zeta^{-1}\cdot (\mathrm{E}_{v-2}\otimes\mathrm{L}_2^0)\circ (\beta^2_{v-2,1}\otimes \mathrm{E}_{1})\circ (\mathrm{E}_{v-2}\otimes\mathrm{L}_2^0)\,.\end{align*}
But, using the case $a_0=0$ along with Proposition \ref{prop:categorifiedfusionT1} and Lemma \ref{lem:YisLtensorId1}, we have that 
\begin{align*}\beta^2_{v-2,1}\otimes \mathrm{E}_{1} = &\ \zeta^{v-2}\cdot \big(\mathrm{E}_{v-1} \otimes \mathrm{E}_{1}\big) + 2\zeta^{v-3}\cdot \big(\mathrm{L}_{v-1}^0\otimes \mathrm{E}_{1}\big)\\ =&\ \zeta^{v-2}\cdot (\mathrm{E}_v + \mathrm{A}^0_v + \tilde{\mathrm{A}}^0_v + \fusidem{v}{1}) + 2\zeta^{v-3}\cdot \mathrm{Y}^0_v\,.
\end{align*}
Moreover, we have by definition that
$$\mathrm{E}_{v-2}\otimes\mathrm{L}_2^0 = \mathrm{Z}^0_v\,.$$
Using these two observations, the result follows by a straightforward ---though long--- computation.
\end{proof}

\subsection{Twists}

The twist $\theta_{\mathbb{T}_{\zeta}(v-1)}=\theta_{v-1}$ on the object $\mathbb{T}_{\zeta}(v-1)$ is represented graphically by
$$\theta_{\mathbb{T}_{\zeta}(v-1)}=\theta_{v-1}=\begin{tikzpicture}[anchorbase,scale=0.25,tinynodes]
\draw[usual] (4,2.5) to[out=270,in=0] (3,1.7) to[out=180,in=270] (2,4);
\draw[usual,crossline] (2,1) to[out=90,in=180] (3,3.3) to[out=0,in=90] (4,2.5);
\draw[pJW] (0,0) rectangle (4,1);
\node at (2,0.3) {$\pjwm[{v{-}1}]$};
\end{tikzpicture}\,.$$
We now identify these morphisms.

\begin{Proposition}\label{prop:twists}
Let $v=[a_1,a_0]_{p,\ell}$ with $a_1\geq 1$. 
\begin{enumerate}
    \item If $a_0=0$, then the twist on $\mathbb{T}_{\zeta}(v-1)$ is given by $$\theta_{v-1} = \zeta^{(v-1)(v-3)/2} \cdot\mathrm{E}_{v-1}= \zeta^{(4a_1^2-8a_1+3)/2}\cdot\mathrm{E}_{v-1}=(-1)^{a_1}\zeta^\frac{3}{2}\cdot\mathrm{E}_{v-1}\,.$$
    \item If $a_0=1$, then the twist on $\mathbb{T}_{\zeta}(v-1)$ is given by $$\theta_{v-1} = \mathrm{E}_{v-1} + 2\zeta^{-1}\cdot \mathrm{L}_{v-1}^0\,.$$
\end{enumerate}
\end{Proposition}
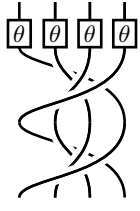
\begin{figure}[h]
    \centering
\begin{tikzpicture}[yscale = 0.5, xscale = -1.4]
    \draw[usual] (0,0) ..controls (0,1.45) and (1,1.45).. (1,2);
    \draw[usual, crossline] (0.33,0) ..controls (0.33,1.15) and (0.66,1.15).. (0.66,2);
    \draw[usual, crossline] (0.66,0) ..controls (0.66,0.85) and (0.33,0.85).. (0.33,2);
    \draw[usual, crossline] (1,0) ..controls (1,0.55) and (0,0.55).. (0,2);
    \begin{scope}[yshift = 2cm]
    \draw[usual] (0,0) ..controls (0,1.45) and (1,1.45).. (1,2) -- ++(0,1.2) node[pos = 0.3, rectangle, draw, fill=white, inner sep = 2pt]{\small $\theta$};
    \draw[usual, crossline] (0.33,0) ..controls (0.33,1.15) and (0.66,1.15).. (0.66,2)-- ++(0,1.2) node[pos = 0.3, rectangle, draw, fill=white, inner sep = 2pt]{\small $\theta$};
    \draw[usual, crossline] (0.66,0) ..controls (0.66,0.85) and (0.33,0.85).. (0.33,2)-- ++(0,1.2) node[pos = 0.3, rectangle, draw, fill=white, inner sep = 2pt]{\small $\theta$};
    \draw[usual, crossline] (1,0) ..controls (1,0.55) and (0,0.55).. (0,2)-- ++(0,1.2) node[pos = 0.3, rectangle, draw, fill=white, inner sep = 2pt]{\small $\theta$};
    \end{scope}
\end{tikzpicture}    
    \caption{The full twist.}
    \label{fig:fulltwist}
\end{figure}
\begin{proof}
The case $a_0=0$ follows from Equation \eqref{eq:classicalbraidingabsorption} and $\theta_{1} = \zeta^{-1/2}\cdot \mathrm{E}_1$.
Namely, the twist on $v-1$ strands can be expressed using $(v-1)(v-2)$ positive crossings between single strands together with a twist $\theta_{1}$ on each individual strand, see Figure \ref{fig:fulltwist} for an example.
We get 
$$\theta_{v-1}=\zeta^{-\frac{v-1}{2}}\cdot \zeta^\frac{(v-1)(v-2)}{2}\cdot \mathrm{E}_{v-1}=\zeta^\frac{(v-1)(v-3)}{2}\cdot \mathrm{E}_{v-1} \ .$$

The case $a_0=1$ may be derived using the fact that $$\theta_{v-1} = (\theta_{v-2}\otimes\theta_{1})\circ\beta^2_{v-2,1}\,.$$
The above formula therefore follows from Proposition \ref{prop:braidingT1} together with the observation that $a_1^2-a_1$ is even for any integer $a_1$, so that $\zeta^{2a_1^2-2a_1}=1$.
\end{proof}

\subsection{More Categorified Fusion Rules}

The fusion rules for Eves in $\Tilt^{\zeta}$ are given in \cite[Proposition 4.11]{STWZ}. Below we record the corresponding fusion rules for the indecomposable projective objects of $\Ver_{p^{(2)}}^{\zeta}$. We emphasize that this corresponds to the fusion rules of the \emph{quotient} $\Tilt^{\zeta}\!/\mathbf{J}_{p^{(2)}}$, so that some terms appearing in \cite{STWZ} do not occur in our case. We have the following definition and formula.

\begin{Definition}
Let $v=[a_1,0]_{p,\ell}$ and $w=[b_1,0]_{p,\ell}$. We say that $u=[c_1,1]_{p,\ell}$ is $p^{(2)}$-admissible for $v$ and $w$ if:
\begin{enumerate}
    \item In case $a_1+b_1-1 < p$, we have $|a_1-b_1|+1\leq c_1\leq a_1+b_1 -1$ and $a_1+b_1-c_1$ is odd.
    \item In case $a_1+b_1-1 \geq p$, we have $|a_1-b_1|+1\leq c_1\leq 2p-a_1-b_1 -1$ and $a_1+b_1-c_1$ is odd.
\end{enumerate}
\end{Definition}

\begin{Lemma}\label{lem:fusionEveEve}
Let $v=[a_1,0]_{p,\ell}$ and $w=[b_1,0]_{p,\ell}$. We have
$$\mathbb{T}_{\zeta}(v-1)\otimes\mathbb{T}_{\zeta}(w-1) = \bigoplus_{u}\mathbb{T}_{\zeta}(u-1),$$ where the sum is over all $u$'s that are $p^{(2)}$-admissible for $v$ and $w$.
\end{Lemma}

The general fusion rules can then be straightforwardly obtained by tensoring with one or two copies of $\mathbb{T}_{\zeta}(1)$. More relevant for our purposes are the categorified fusion rules. In the case of Eves, these were essentially computed in \cite[Theorem 4.32]{STWZ}. As we are in working with $\zeta$ a primitive fourth root of unity, we will in fact be able to simplify their formulas. What makes the case of Eves more manageable is that, for any fixed highest weight, the corresponding indecomposable tilting module appears in the tensor product of two eves at most once. Before recalling this result, we need to introduce some notations. Given $v=[a_1,0]_{p,\ell}$, $w=[b_1,0]_{p,\ell}$, and $u$ that is $p^{(2)}$-admissible, we define:
\begin{equation}\label{eq:XEves}
\mathrm{EX}^{u-1}_{v-1,w-1}:=
\begin{tikzpicture}[anchorbase,scale=0.25,tinynodes]
\draw[JW] (-2.5,3) rectangle (2.5,5);
\node at (0,3.8) {$\pjwm[v{-}1]$};
\draw[JW] (3.5,3) rectangle (8.5,5);
\node at (6,3.8) {$\pjwm[w{-}1]$};
\draw[pJW] (0.5,-1) rectangle (5.5,1);
\node at (3,-0.2) {$\pjwm[u{-}1]$};
\draw[JW] (-2.5,-3) rectangle (2.5,-5);
\node at (0,-4.2) {$\pjwm[v{-}1]$};
\draw[JW] (3.5,-3) rectangle (8.5,-5);
\node at (6,-4.2) {$\pjwm[w{-}1]$};
\draw[usual] (2,-3) to[out=90,in=180] (3,-2)node[above,xshift=0.15cm,yshift=-0.1cm]{$x$} to[out=0,in=90] (4,-3);
\draw[usual] (-1.5,-3)node[left,yshift=0.3cm,xshift=0.3cm]{$m$} to[out=90,in=270] (1.5,-1);
\draw[usual] (7.5,-3)node[left,yshift=0.3cm,xshift=0.1cm]{$n$} to[out=90,in=270] (4.5,-1);
\draw[usual] (2,3) to[out=270,in=180] (3,2)node[below,xshift=0.15cm,yshift=0.05cm]{$x$} to[out=0,in=270] (4,3);
\draw[usual] (-1.5,3) node[left,yshift=-0.4cm,xshift=0.3cm]{$m$}to[out=270,in=90] (1.5,1);
\draw[usual] (7.5,3)node[left,yshift=-0.4cm,xshift=0.1cm]{$n$} to[out=270,in=90] (4.5,1);
\end{tikzpicture}	
,\quad
\Diamond_{u-1}^{v-1,w-1}:=
\begin{tikzpicture}[anchorbase,scale=0.25,tinynodes]
\draw[JW] (-2.5,3) rectangle (2.5,5);
\node at (0,3.8) {$\pjwm[v{-}1]$};
\draw[JW] (3.5,3) rectangle (8.5,5);
\node at (6,3.8) {$\pjwm[w{-}1]$};
\draw[pJW] (0.5,-1) rectangle (5.5,1);
\node at (3,-0.2) {$\pjwm[u{-}1]$};
\draw[pJW] (0.5,7) rectangle (5.5,9);
\node at (3,7.8) {$\pjwm[u{-}1]$};
\draw[usual] (2,5) to[out=90,in=180] (3,6)node[above,xshift=0.15cm,yshift=-0.1cm]{$x$} to[out=0,in=90] (4,5);
\draw[usual] (-1.5,5) to[out=90,in=270] (1.5,7);
\draw[usual] (7.5,5) to[out=90,in=270] (4.5,7);
\draw[usual] (2,3) to[out=270,in=180] (3,2)node[below,xshift=0.15cm,yshift=0.05cm]{$x$} to[out=0,in=270] (4,3);
\draw[usual] (-1.5,3) to[out=270,in=90] (1.5,1);
\draw[usual] (7.5,3) to[out=270,in=90] (4.5,1);
\end{tikzpicture}\,,
\end{equation}
\begin{equation}\label{eq:LXEves}
\mathrm{L}^0\mathrm{X}^{u-1}_{v-1,w-1}:=
\begin{tikzpicture}[anchorbase,scale=0.25,tinynodes]
\draw[JW] (-2.5,3) rectangle (2.5,5);
\node at (0,3.8) {$\pjwm[v{-}1]$};
\draw[JW] (3.5,3) rectangle (8.5,5);
\node at (6,3.8) {$\pjwm[w{-}1]$};
\ptr{5}{1}{$\pjwm[u{-}1]$}{-5.5}{0};
\draw[JW] (-2.5,-3) rectangle (2.5,-5);
\node at (0,-4.2) {$\pjwm[v{-}1]$};
\draw[JW] (3.5,-3) rectangle (8.5,-5);
\node at (6,-4.2) {$\pjwm[w{-}1]$};
\draw[usual] (2,-3) to[out=90,in=180] (3,-2)node[above,xshift=0.15cm,yshift=-0.1cm]{$x$} to[out=0,in=90] (4,-3);
\draw[usual] (-1.5,-3)node[left,yshift=0.3cm,xshift=0.3cm]{$m$} to[out=90,in=270] (1.5,-1);
\draw[usual] (7.5,-3)node[left,yshift=0.3cm,xshift=0.1cm]{$n$}to[out=90,in=270] (4.5,-1);
\draw[usual] (2,3) to[out=270,in=180] (3,2)node[below,xshift=0.15cm,yshift=0.05cm]{$x$} to[out=0,in=270] (4,3);
\draw[usual] (-1.5,3) node[left,yshift=-0.4cm,xshift=0.3cm]{$m$}to[out=270,in=90] (1.5,1);
\draw[usual] (7.5,3)node[left,yshift=-0.4cm,xshift=0.1cm]{$n$} to[out=270,in=90] (4.5,1);
\end{tikzpicture}
=
\begin{tikzpicture}[anchorbase,scale=0.25,tinynodes]
\draw[JW] (-2.5,-7) rectangle (2.5,-5);
\node at (0,-6.2) {$\qjwm[v{-}1]$};
\draw[JW] (3.5,-7) rectangle (8.5,-5);
\node at (6,-6.2) {$\qjwm[w{-}1]$};
\draw[JW] (-2.5,-19) rectangle (2.5,-17);
\node at (0,-18.2) {$\qjwm[v{-}1]$};
\draw[JW] (3.5,-19) rectangle (8.5,-17);
\node at (6,-18.2) {$\qjwm[w{-}1]$};
\draw[JW] (0.5,-11) rectangle (5.5,-9);
\node at (3,-10.2) {$\qjwm[u{-}2]$};
\draw[JW] (0.5,-15) rectangle (5.5,-13);
\node at (3,-14.2) {$\qjwm[u{-}2]$};
\draw[usual] (2,-17) to[out=90,in=180] (3,-16)node[above,xshift=0.15cm,yshift=-0.1cm]{$x$} to[out=0,in=90] (4,-17);
\draw[usual] (2,-7) to[out=270,in=180] (3,-8)node[below,xshift=0.15cm,yshift=0.05cm]{$x$} to[out=0,in=270] (4,-7);
\draw[usual] (-1.5,-17) to[out=90,in=270] (1.5,-15)node[left,xshift=-0.4cm,yshift=-0.2cm]{$m$};
\draw[usual] (7.5,-17) to[out=90,in=270] (4.5,-15);
\draw[usual] (7.9,-17) to[out=90,in=60] (4.5,-13)node[right,xshift=0.5cm,yshift=-0.2cm]{$1$};
\draw[usual] (-1.5,-7) to[out=270,in=90] (1.5,-9)node[left,xshift=-0.4cm,yshift=0.05cm]{$m$};
\draw[usual] (7.9,-7) to[out=270,in=300] (4.5,-11)node[right,xshift=0.5cm,yshift=0.05cm]{$1$};
\draw[usual] (7.5,-7) to[out=270,in=90] (4.5,-9);
\draw[usual] (1.5,-13) to (1.5,-11);
\end{tikzpicture}.
\end{equation}

\begin{Theorem}\label{thm:categorifiedfusion}
Let $v=[a_1,0]_{p,\ell}$ and $w=[b_1,0]_{p,\ell}$, and let $u=[c_1,1]$ be such that $\mathrm{T}_{\zeta}(u-1)$ is an indecomposable tilting module appearing as a direct summand of $\mathrm{T}_{\zeta}(v-1)\otimes\mathrm{T}_{\zeta}(w-1)$. Then, the corresponding idempotent in the Temperley-Lieb category realizing the projection onto $\mathrm{T}_{\zeta}(u-1)$ is given by
$$\mathsf{d}_{v-1,w-1}^{u-1}\cdot \mathrm{EX}_{v-1,w-1}^{u-1}\quad \textrm{with}\quad \mathsf{d}_{v-1,w-1}^{u-1}:=(-1)^{x_1}\frac{\binom{a_1-1}{x_1}\binom{b_1-1}{x_1}}{\binom{c_1+x_1}{x_1}}\in\mathbbm{k}\,,$$ where $x = [x_1,0]_{p,\ell} = a_1+b_1-c_1-1$.
\end{Theorem}
\begin{proof}
The above description is a refinement of \cite[Theorem 4.32]{STWZ}. For our purposes, it is slightly more convenient to rely on the statement of \cite[Theorem 4.27]{STWZ}. In order to recall this result, we define elements of the field $\mathbbm{k}(\mathbbm{v})$, with $\mathbbm{v}$ a formal variable.
Using quantum integers and quantum binomial coefficients, we define following \cite[Section 4C]{STWZ}
$$\tilde{\mathsf{d}}^{v-1,w-1}_{u-1,\emptyset} := (-1)^{x}\frac{\binom{m+n+x+1}{x}_{\mathbbm{v}}}{\binom{m+x}{x}_{\mathbbm{v}}\binom{n+x}{x}_{\mathbbm{v}}}\in \mathbbm{k}(\mathbbm{v})\,,$$
$$\tilde{\mathsf{d}}^{v-1,w-1}_{u-1,\{0\}} = (-1)^{x+1} \Big(-\frac{[m+n-1]_{\mathbbm{v}}}{[m+n]_{\mathbbm{v}}}\frac{\binom{m+n+x+1}{x}_{\mathbbm{v}}}{\binom{m+x}{x}_{\mathbbm{v}}\binom{n+x}{x}_{\mathbbm{v}}} + \big(\frac{[m]_{\mathbbm{v}}}{[m+n]_{\mathbbm{v}}}\big)^2\frac{\binom{m+n+x}{x}_{\mathbbm{v}}}{\binom{m+x}{x+1}_{\mathbbm{v}}\binom{n+x}{x+1}_{\mathbbm{v}}}\Big)\in \mathbbm{k}(\mathbbm{v})\,,$$
where the natural numbers $m$, $n$, and $x$ are as in Equation \eqref{eq:XEves}.
Upon specializing $\mathbbm{v}\mapsto\zeta$, the elements $\tilde{\mathsf{d}}^{v-1,w-1}_{u-1,\emptyset}, \tilde{\mathsf{d}}^{v-1,w-1}_{u-1,\{0\}}\in \mathbbm{k}(\mathbbm{v})$ specialize to $\mathsf{d}^{v-1,w-1}_{u-1,\emptyset}, \mathsf{d}^{v-1,w-1}_{u-1,\{0\}}\in \mathbbm{k}$.
It was established in \cite[Theorem 4.27]{STWZ} that
$$\Diamond_{u-1}^{v-1,w-1} = \mathsf{d}^{v-1,w-1}_{u-1,\emptyset}\cdot \mathrm{E}_{u-1} + \mathsf{d}^{v-1,w-1}_{u-1,\{0\}}\cdot\mathrm{L}^0_{u-1}\,.$$
What we need to show in order to prove the desired result is that $\mathsf{d}^{v-1,w-1}_{u-1,\{0\}} = 0$.
Using the definition of the quantum binomial coefficient, we find
$$\tilde{\mathsf{d}}^{v-1,w-1}_{u-1,\{0\}} = \Big(\frac{[m+n-1]_{\mathbbm{v}}[n]_{\mathbbm{v}}[m+n+x+1]_{\mathbbm{v}}-[m]_{\mathbbm{v}}[m+n+1]_{\mathbbm{v}}[x+1]_{\mathbbm{v}}}{[m+n]_{\mathbbm{v}}[n]_{\mathbbm{v}}[m+n+x+1]_{\mathbbm{v}}}\Big)\,\cdot\,\tilde{\mathsf{d}}^{v-1,w-1}_{u-1,\emptyset}\,.$$
The above expression is advantageous because we know that $\tilde{\mathsf{d}}^{v-1,w-1}_{u-1,\emptyset}$ specialize to $\mathsf{d}^{v-1,w-1}_{u-1,\emptyset}\neq 0$ provided that $u$ is $p^{(2)}$-admissible thanks to \cite[Proposition 4.26]{STWZ}. 
In particular, in order to show that $\tilde{\mathsf{d}}^{v-1,w-1}_{u-1,\{0\}}$ vanishes upon specialization, it will be enough to show that this holds for $$\tilde{\delta}^{v-1,w-1}_{u-1,\{0\}}:=\frac{[m+n-1]_{\mathbbm{v}}[n]_{\mathbbm{v}}[m+n+x+1]_{\mathbbm{v}}-[m]_{\mathbbm{v}}[m+n+1]_{\mathbbm{v}}[x+1]_{\mathbbm{v}}}{[m+n]_{\mathbbm{v}}[n]_{\mathbbm{v}}[m+n+x+1]_{\mathbbm{v}}}.$$

In order to prove this last claim, observe that 
\begin{equation}\label{eq:miracle}
[m+n+1]_{\mathbbm{v}}[x+1]_{\mathbbm{v}} = [m+n+x+1]_{\mathbbm{v}} + [x]_{\mathbbm{v}}[m+n]_{\mathbbm{v}}.
\end{equation}
This equality follows from \cite[Lemma 3]{MV} or can be checked directly. We then find that
\begin{align*}\tilde{\delta}^{v-1,w-1}_{u-1,\{0\}} &=\frac{[m+n-1]_{\mathbbm{v}}[n]_{\mathbbm{v}}[m+n+x+1]_{\mathbbm{v}}-[m]_{\mathbbm{v}}[m+n+x+1]_{\mathbbm{v}}-[m]_{\mathbbm{v}}[m+n]_{\mathbbm{v}}[x]_{\mathbbm{v}}}{[m+n]_{\mathbbm{v}}[n]_{\mathbbm{v}}[m+n+x+1]_{\mathbbm{v}}},\\
&=\frac{[m+n-1]_{\mathbbm{v}}[n]_{\mathbbm{v}}-[m]_{\mathbbm{v}}}{[m+n]_{\mathbbm{v}}[n]_{\mathbbm{v}}} - \frac{[m]_{\mathbbm{v}}[x]_{\mathbbm{v}}}{[n]_{\mathbbm{v}}[m+n+x+1]_{\mathbbm{v}}}\\
&=\frac{[n-1]_{\mathbbm{v}}}{[n]_{\mathbbm{v}}} - \frac{[m]_{\mathbbm{v}}[x]_{\mathbbm{v}}}{[n]_{\mathbbm{v}}[m+n+x+1]_{\mathbbm{v}}}.
\end{align*}
The third equality uses
$$[m+n-1]_{\mathbbm{v}}[n]_{\mathbbm{v}} = [m+n]_{\mathbbm{v}}[n-1]_{\mathbbm{v}} + [m]_{\mathbbm{v}}\,.$$
In particular, as $x$ and $n-1$ are even, whereas $n$ and $m+n+x+1$ are odd, we find that that $\tilde{\delta}^{v-1,w-1}_{u-1,\{0\}}$ specializes to $0$ by setting $\mathbbm{v}$ to the primitive fourth root of unity $\zeta$. In particular, it follows that $$\Diamond_{u-1}^{v-1,w-1} = \mathsf{d}^{v-1,w-1}_{u-1,\emptyset}\cdot \mathrm{E}_{u-1}\,,$$ from which one easily deduces that the endomorphism $(\mathsf{d}^{v-1,w-1}_{u-1,\emptyset})^{-1}\cdot \mathrm{X}_{v-1,w-1}^{u-1}$ of $\mathrm{T}_{\zeta}(v-1)\otimes\mathrm{T}_{\zeta}(w-1)$ implements the projection onto the indecomposable direct summand $\mathbb{T}_{\zeta}(u-1)$. Finally, the equality
$$\mathsf{d}_{v-1,w-1}^{u-1}:=(\mathsf{d}^{v-1,w-1}_{u-1,\emptyset})^{-1} = (-1)^{x_1}\frac{\binom{a_1-1}{x_1}\binom{b_1-1}{x_1}}{\binom{c_1+x_1}{x_1}}$$
follows from the quantum Lucas' Theorem given in \cite[Proposition 2.8]{STWZ}.
\end{proof}

\subsection{More Double Braidings}

Let $v=[a_1,0]_{p,\ell}$ and $w=[b_1,0]_{p,\ell}$. We give a formula for the double and the single braiding of $\mathbb{T}_{\zeta}(v-1)$ and $\mathbb{T}_{\zeta}(w-1)$. In order to give our formula for the single braiding, we introduce for any $p^{(2)}$-admissible $u$ the morphisms
$$\mathrm{E}\check{\mathrm{X}}^{u-1}_{v-1,w-1}:=
\begin{tikzpicture}[anchorbase,scale=0.25,tinynodes]
\draw[JW] (-2.5,3) rectangle (2.5,5);
\node at (0,3.8) {$\pjwm[w{-}1]$};
\draw[JW] (3.5,3) rectangle (8.5,5);
\node at (6,3.8) {$\pjwm[v{-}1]$};
\draw[pJW] (0.5,-1) rectangle (5.5,1);
\node at (3,-0.2) {$\pjwm[u{-}1]$};
\draw[JW] (-2.5,-3) rectangle (2.5,-5);
\node at (0,-4.2) {$\pjwm[v{-}1]$};
\draw[JW] (3.5,-3) rectangle (8.5,-5);
\node at (6,-4.2) {$\pjwm[w{-}1]$};
\draw[usual] (2,-3) to[out=90,in=180] (3,-2) to[out=0,in=90] (4,-3);
\draw[usual] (-1.5,-3) to[out=90,in=270] (1.5,-1);
\draw[usual] (7.5,-3) to[out=90,in=270] (4.5,-1);
\draw[usual] (2,3) to[out=270,in=180] (3,2) to[out=0,in=270] (4,3);
\draw[usual] (-1.5,3) to[out=270,in=90] (1.5,1);
\draw[usual] (7.5,3) to[out=270,in=90] (4.5,1);
\end{tikzpicture}\,,\quad
\mathrm{L}^0\check{\mathrm{X}}^{u-1}_{v-1,w-1}:=
\begin{tikzpicture}[anchorbase,scale=0.25,tinynodes]
\draw[JW] (-2.5,3) rectangle (2.5,5);
\node at (0,3.8) {$\pjwm[w{-}1]$};
\draw[JW] (3.5,3) rectangle (8.5,5);
\node at (6,3.8) {$\pjwm[v{-}1]$};
\ptr{5}{1}{$\pjwm[u{-}1]$}{-5.5}{0};
\draw[JW] (-2.5,-3) rectangle (2.5,-5);
\node at (0,-4.2) {$\pjwm[v{-}1]$};
\draw[JW] (3.5,-3) rectangle (8.5,-5);
\node at (6,-4.2) {$\pjwm[w{-}1]$};
\draw[usual] (2,-3) to[out=90,in=180] (3,-2) to[out=0,in=90] (4,-3);
\draw[usual] (-1.5,-3) to[out=90,in=270] (1.5,-1);
\draw[usual] (7.5,-3)to[out=90,in=270] (4.5,-1);
\draw[usual] (2,3) to[out=270,in=180] (3,2) to[out=0,in=270] (4,3);
\draw[usual] (-1.5,3) to[out=270,in=90] (1.5,1);
\draw[usual] (7.5,3) to[out=270,in=90] (4.5,1);
\end{tikzpicture}\,.$$
These are variants of the morphisms $\mathrm{EX}^{u-1}_{v-1,w-1}$ and $\mathrm{L}^0\mathrm{X}^{u-1}_{v-1,w-1}$ defined in the previous section, for which the source and the target are ordered differently.

\begin{Proposition}\label{prop:braidingEves}
Let $v=[a_1,0]_{p,\ell}$ and $w=[b_1,0]_{p,\ell}$. Then, we have
$$\beta^2_{v-1,w-1}=\begin{tikzpicture}[anchorbase,scale=0.25,tinynodes]
\draw[usual] (7,-2) to[out=90,in=270] (2,0);
\draw[usual,crossline] (2,-2) to[out=90,in=270] (7,0);
\draw[usual] (7,-4) to[out=90,in=270] (2,-2);
\draw[usual,crossline] (2,-4) to[out=90,in=270] (7,-2);
\draw[JW] (0,0) rectangle (4,1);
\node at (2,0.3) {$\pjwm[{v{-}1}]$};
\draw[JW] (5,0) rectangle (9,1);
\node at (7,0.3) {$\pjwm[{w{-}1}]$};
\draw[JW] (0,-5) rectangle (4,-4);
\node at (2,-4.7) {$\pjwm[{v{-}1}]$};
\draw[JW] (5,-5) rectangle (9,-4);
\node at (7,-4.7) {$\pjwm[{w{-}1}]$};
\end{tikzpicture}
=\sum_{u}\zeta^{1-2a_1^2-2b_1^2}\cdot \mathsf{d}^{u-1}_{v-1,w-1}\cdot\big(\mathrm{E}\mathrm{X}_{v-1,w-1}^{u-1} + 2\zeta^{-1}\cdot \mathrm{L}^0\mathrm{X}_{v-1,w-1}^{u-1}\big)\,,$$
\begin{align*}\beta_{v-1,w-1}=\begin{tikzpicture}[anchorbase,scale=0.25,tinynodes]
\draw[usual] (7,-3) to[out=90,in=270] (2,0);
\draw[usual,crossline] (2,-3) to[out=90,in=270] (7,0);
\draw[JW] (0,0) rectangle (4,1);
\node at (2,0.3) {$\pjwm[{w{-}1}]$};
\draw[JW] (5,0) rectangle (9,1);
\node at (7,0.3) {$\pjwm[{v{-}1}]$};
\draw[JW] (0,-4) rectangle (4,-3);
\node at (2,-3.7) {$\pjwm[{v{-}1}]$};
\draw[JW] (5,-4) rectangle (9,-3);
\node at (7,-3.7) {$\pjwm[{w{-}1}]$};
\end{tikzpicture} =  \sum_{u} \lambda^{u-1}_{v-1,w-1} \cdot \mathsf{d}_{v-1,w-1}^{u-1}\cdot\Big(\mathrm{E}\check{\mathrm{X}}^{u-1}_{v-1,w-1} + \zeta^{-1}\cdot \mathrm{L}^0\check{\mathrm{X}}^{u-1}_{v-1,w-1}\Big)\,,\end{align*}
$$\lambda^{u-1}_{v-1,w-1}:=\zeta^{\big((u+v-w-1)(u+w-v-1)-(v+w-u-1)(u+v+w+1)\big)/8} = \zeta^{(u^2-v^2-w^2+1)/4}\,,$$
where the sums run over all $p^{(2)}$-admissible $u$'s for $v$ and $w$.
\end{Proposition}

\begin{proof}
The first equality follows from Proposition \ref{prop:twists} using that $$\beta^2_{v-1,w-1}\circ(\theta_{v-1}\otimes\theta_{w-1})=\sum_u\theta_{u-1}.$$

The second equality can be established by specializing the formula for the braiding in the Temperley-Lieb category at a generic parameter. More precisely, it follows from \cite[Theorem 3]{MV} that at generic parameter  $\mathbbm{v}$, we have
\begin{align*}\begin{tikzpicture}[anchorbase,scale=0.25,tinynodes]
\draw[usual] (7,-3) to[out=90,in=270] (2,0);
\draw[usual,crossline] (2,-3) to[out=90,in=270] (7,0);
\draw[JW] (0,0) rectangle (4,1);
\node at (2,0.3) {$\pjwm[{w{-}1}]$};
\draw[JW] (5,0) rectangle (9,1);
\node at (7,0.3) {$\pjwm[{v{-}1}]$};
\draw[JW] (0,-4) rectangle (4,-3);
\node at (2,-3.7) {$\pjwm[{v{-}1}]$};
\draw[JW] (5,-4) rectangle (9,-3);
\node at (7,-3.7) {$\pjwm[{w{-}1}]$};
\end{tikzpicture} =  \sum_{u} \tilde{\lambda}^{u-1}_{v-1,w-1} \cdot \tilde{\mathsf{d}}_{v-1,w-1}^{u-1}\cdot
\begin{tikzpicture}[anchorbase,scale=0.25,tinynodes]
\draw[JW] (-2.5,3) rectangle (2.5,5);
\node at (0,3.8) {$\pjwm[w{-}1]$};
\draw[JW] (3.5,3) rectangle (8.5,5);
\node at (6,3.8) {$\pjwm[v{-}1]$};
\draw[JW] (0.5,-1) rectangle (5.5,1);
\node at (3,-0.2) {$\pjwm[u{-}1]$};
\draw[JW] (-2.5,-3) rectangle (2.5,-5);
\node at (0,-4.2) {$\pjwm[v{-}1]$};
\draw[JW] (3.5,-3) rectangle (8.5,-5);
\node at (6,-4.2) {$\pjwm[w{-}1]$};
\draw[usual] (2,-3) to[out=90,in=180] (3,-2)node[above,xshift=0.15cm,yshift=-0.1cm]{$x$} to[out=0,in=90] (4,-3);
\draw[usual] (-1.5,-3)node[left,yshift=0.3cm,xshift=0.3cm]{$m$} to[out=90,in=270] (1.5,-1);
\draw[usual] (7.5,-3)node[left,yshift=0.3cm,xshift=0.1cm]{$n$} to[out=90,in=270] (4.5,-1);
\draw[usual] (2,3) to[out=270,in=180] (3,2)node[below,xshift=0.15cm,yshift=0.05cm]{$x$} to[out=0,in=270] (4,3);
\draw[usual] (-1.5,3) node[left,yshift=-0.4cm,xshift=0.3cm]{$n$}to[out=270,in=90] (1.5,1);
\draw[usual] (7.5,3)node[left,yshift=-0.4cm,xshift=0.1cm]{$m$} to[out=270,in=90] (4.5,1);
\end{tikzpicture}\,,\end{align*}
$$\tilde{\mathsf{d}}^{u-1}_{v-1,w-1} = (\tilde{\mathsf{d}}_{u-1,\emptyset}^{v-1,w-1})^{-1}=(-1)^{x}\frac{\binom{m+x}{x}_{\mathbbm{v}}\binom{n+x}{x}_{\mathbbm{v}}}{\binom{m+n+x+1}{x}_{\mathbbm{v}}}\,,\quad\tilde{\lambda}^{u-1}_{v-1,w-1}=(-1)^{x}\mathbbm{v}^{\big(mn-x(m+n+x+2)\big)/2}\,,$$ 
and the sum runs over all $u$'s that are admissible for $v$ and $w$ in the generic sense, meaning that $0\leq x\leq \min(v-1,w-1)$ can be arbitrary.

On the one hand, if $u$ is $p^{(2)}$-admissible for $v$ and $w$, we have at generic parameter
$$\begin{tikzpicture}[anchorbase,scale=0.189,tinynodes]
\draw[JW] (-2.5,3) rectangle (2.5,5);
\node at (0,3.8) {$\pjwm[w{-}1]$};
\draw[JW] (3.5,3) rectangle (8.5,5);
\node at (6,3.8) {$\pjwm[v{-}1]$};
\draw[JW] (0.5,-1) rectangle (5.5,1);
\node at (3,-0.2) {$\pjwm[u{-}1]$};
\draw[JW] (-2.5,-3) rectangle (2.5,-5);
\node at (0,-4.2) {$\pjwm[v{-}1]$};
\draw[JW] (3.5,-3) rectangle (8.5,-5);
\node at (6,-4.2) {$\pjwm[w{-}1]$};
\draw[usual] (2,-3) to[out=90,in=180] (3,-2)node[above,xshift=0.15cm,yshift=-0.1cm]{$x$} to[out=0,in=90] (4,-3);
\draw[usual] (-1.5,-3)node[left,yshift=0.3cm,xshift=0.3cm]{$m$} to[out=90,in=270] (1.5,-1);
\draw[usual] (7.5,-3)node[left,yshift=0.3cm,xshift=0.1cm]{$n$} to[out=90,in=270] (4.5,-1);
\draw[usual] (2,3) to[out=270,in=180] (3,2)node[below,xshift=0.15cm,yshift=0.05cm]{$x$} to[out=0,in=270] (4,3);
\draw[usual] (-1.5,3) node[left,yshift=-0.4cm,xshift=0.3cm]{$n$}to[out=270,in=90] (1.5,1);
\draw[usual] (7.5,3)node[left,yshift=-0.4cm,xshift=0.1cm]{$m$} to[out=270,in=90] (4.5,1);
\end{tikzpicture}
= \left[
\begin{tikzpicture}[anchorbase,scale=0.189,tinynodes]
\draw[JW] (-2.5,3) rectangle (2.5,5);
\node at (0,3.8) {$\pjwm[w{-}1]$};
\draw[JW] (3.5,3) rectangle (8.5,5);
\node at (6,3.8) {$\pjwm[v{-}1]$};
\draw[JW] (0.5,-1) rectangle (5.5,1);
\node at (3,-0.2) {$\pjwm[u{-}1]$};
\draw[JW] (-2.5,-3) rectangle (2.5,-5);
\node at (0,-4.2) {$\pjwm[v{-}1]$};
\draw[JW] (3.5,-3) rectangle (8.5,-5);
\node at (6,-4.2) {$\pjwm[w{-}1]$};
\draw[usual] (2,-3) to[out=90,in=180] (3,-2)node[above,xshift=0.15cm,yshift=-0.1cm]{$x$} to[out=0,in=90] (4,-3);
\draw[usual] (-1.5,-3)node[left,yshift=0.3cm,xshift=0.3cm]{$m$} to[out=90,in=270] (1.5,-1);
\draw[usual] (7.5,-3)node[left,yshift=0.3cm,xshift=0.1cm]{$n$} to[out=90,in=270] (4.5,-1);
\draw[usual] (2,3) to[out=270,in=180] (3,2)node[below,xshift=0.15cm,yshift=0.05cm]{$x$} to[out=0,in=270] (4,3);
\draw[usual] (-1.5,3) node[left,yshift=-0.4cm,xshift=0.3cm]{$n$}to[out=270,in=90] (1.5,1);
\draw[usual] (7.5,3)node[left,yshift=-0.4cm,xshift=0.1cm]{$m$} to[out=270,in=90] (4.5,1);
\end{tikzpicture} - \frac{[m+n-1]_{\mathbbm{v}}}{[m+n]_{\mathbbm{v}}}\cdot \begin{tikzpicture}[anchorbase,scale=0.189,tinynodes]
\draw[JW] (-2.5,-7) rectangle (2.5,-5);
\node at (0,-6.2) {$\qjwm[w{-}1]$};
\draw[JW] (3.5,-7) rectangle (8.5,-5);
\node at (6,-6.2) {$\qjwm[v{-}1]$};
\draw[JW] (-2.5,-19) rectangle (2.5,-17);
\node at (0,-18.2) {$\qjwm[v{-}1]$};
\draw[JW] (3.5,-19) rectangle (8.5,-17);
\node at (6,-18.2) {$\qjwm[w{-}1]$};
\draw[JW] (0.5,-11) rectangle (5.5,-9);
\node at (3,-10.2) {$\qjwm[u{-}2]$};
\draw[JW] (0.5,-15) rectangle (5.5,-13);
\node at (3,-14.2) {$\qjwm[u{-}2]$};
\draw[usual] (2,-17) to[out=90,in=180] (3,-16)node[above,xshift=0.15cm,yshift=-0.1cm]{$x$} to[out=0,in=90] (4,-17);
\draw[usual] (2,-7) to[out=270,in=180] (3,-8)node[below,xshift=0.15cm,yshift=0.05cm]{$x$} to[out=0,in=270] (4,-7);
\draw[usual] (-1.5,-17) to[out=90,in=270] (1.5,-15)node[left,xshift=-0.4cm,yshift=-0.2cm]{$m$};
\draw[usual] (7.5,-17) to[out=90,in=270] (4.5,-15);
\draw[usual] (7.9,-17) to[out=90,in=60] (4.5,-13)node[right,xshift=0.5cm,yshift=-0.2cm]{$1$};
\draw[usual] (-1.5,-7) to[out=270,in=90] (1.5,-9)node[left,xshift=-0.4cm,yshift=0.05cm]{$n$};
\draw[usual] (7.9,-7) to[out=270,in=300] (4.5,-11)node[right,xshift=0.5cm,yshift=0.05cm]{$1$};
\draw[usual] (7.5,-7) to[out=270,in=90] (4.5,-9);
\draw[usual] (1.5,-13) to (1.5,-11);
\end{tikzpicture}\right] + \frac{[m+n-1]_{\mathbbm{v}}}{[m+n]_{\mathbbm{v}}}\cdot \begin{tikzpicture}[anchorbase,scale=0.189,tinynodes]
\draw[JW] (-2.5,-7) rectangle (2.5,-5);
\node at (0,-6.2) {$\qjwm[w{-}1]$};
\draw[JW] (3.5,-7) rectangle (8.5,-5);
\node at (6,-6.2) {$\qjwm[v{-}1]$};
\draw[JW] (-2.5,-19) rectangle (2.5,-17);
\node at (0,-18.2) {$\qjwm[v{-}1]$};
\draw[JW] (3.5,-19) rectangle (8.5,-17);
\node at (6,-18.2) {$\qjwm[w{-}1]$};
\draw[JW] (0.5,-11) rectangle (5.5,-9);
\node at (3,-10.2) {$\qjwm[u{-}2]$};
\draw[JW] (0.5,-15) rectangle (5.5,-13);
\node at (3,-14.2) {$\qjwm[u{-}2]$};
\draw[usual] (2,-17) to[out=90,in=180] (3,-16)node[above,xshift=0.15cm,yshift=-0.1cm]{$x$} to[out=0,in=90] (4,-17);
\draw[usual] (2,-7) to[out=270,in=180] (3,-8)node[below,xshift=0.15cm,yshift=0.05cm]{$x$} to[out=0,in=270] (4,-7);
\draw[usual] (-1.5,-17) to[out=90,in=270] (1.5,-15)node[left,xshift=-0.4cm,yshift=-0.2cm]{$m$};
\draw[usual] (7.5,-17) to[out=90,in=270] (4.5,-15);
\draw[usual] (7.9,-17) to[out=90,in=60] (4.5,-13)node[right,xshift=0.5cm,yshift=-0.2cm]{$1$};
\draw[usual] (-1.5,-7) to[out=270,in=90] (1.5,-9)node[left,xshift=-0.4cm,yshift=0.05cm]{$n$};
\draw[usual] (7.9,-7) to[out=270,in=300] (4.5,-11)node[right,xshift=0.5cm,yshift=0.05cm]{$1$};
\draw[usual] (7.5,-7) to[out=270,in=90] (4.5,-9);
\draw[usual] (1.5,-13) to (1.5,-11);
\end{tikzpicture}\,.$$
Moreover, by construction \cite[Section 3B]{STWZ}, we have that the term in bracket specializes under $\mathbbm{v}\mapsto \zeta$ to $\mathrm{E}\check{\mathrm{X}}^{u-1}_{v-1,w-1}$.
On the other hand, if $u\leq p^{(2)}-1$ is not $p^{(2)}$-admissible for $v$ and $w$, then we find 
$$\begin{tikzpicture}[anchorbase,scale=0.189,tinynodes]
\draw[JW] (-2.5,3) rectangle (2.5,5);
\node at (0,3.8) {$\pjwm[w{-}1]$};
\draw[JW] (3.5,3) rectangle (8.5,5);
\node at (6,3.8) {$\pjwm[v{-}1]$};
\draw[JW] (0.5,-1) rectangle (5.5,1);
\node at (3,-0.2) {$\pjwm[u{-}1]$};
\draw[JW] (-2.5,-3) rectangle (2.5,-5);
\node at (0,-4.2) {$\pjwm[v{-}1]$};
\draw[JW] (3.5,-3) rectangle (8.5,-5);
\node at (6,-4.2) {$\pjwm[w{-}1]$};
\draw[usual] (2,-3) to[out=90,in=180] (3,-2)node[above,xshift=0.15cm,yshift=-0.1cm]{$x$} to[out=0,in=90] (4,-3);
\draw[usual] (-1.5,-3)node[left,yshift=0.3cm,xshift=0.3cm]{$m$} to[out=90,in=270] (1.5,-1);
\draw[usual] (7.5,-3)node[left,yshift=0.3cm,xshift=0.1cm]{$n$} to[out=90,in=270] (4.5,-1);
\draw[usual] (2,3) to[out=270,in=180] (3,2)node[below,xshift=0.15cm,yshift=0.05cm]{$x$} to[out=0,in=270] (4,3);
\draw[usual] (-1.5,3) node[left,yshift=-0.4cm,xshift=0.3cm]{$n$}to[out=270,in=90] (1.5,1);
\draw[usual] (7.5,3)node[left,yshift=-0.4cm,xshift=0.1cm]{$m$} to[out=270,in=90] (4.5,1);
\end{tikzpicture} =
\frac{[m+n]_{\mathbbm{v}}^2}{[m]_{\mathbbm{v}}[n]_{\mathbbm{v}}}\cdot \begin{tikzpicture}[anchorbase,scale=0.189,tinynodes]
\draw[JW] (-2.5,-7) rectangle (2.5,-5);
\node at (0,-6.2) {$\qjwm[w{-}1]$};
\draw[JW] (3.5,-7) rectangle (8.5,-5);
\node at (6,-6.2) {$\qjwm[v{-}1]$};
\draw[JW] (-2.5,-19) rectangle (2.5,-17);
\node at (0,-18.2) {$\qjwm[v{-}1]$};
\draw[JW] (3.5,-19) rectangle (8.5,-17);
\node at (6,-18.2) {$\qjwm[w{-}1]$};
\draw[JW] (0.5,-11) rectangle (5.5,-9);
\node at (3,-10.2) {$\qjwm[u]$};
\draw[JW] (0.5,-15) rectangle (5.5,-13);
\node at (3,-14.2) {$\qjwm[u]$};
\draw[usual] (2,-17) to[out=90,in=180] (3,-16)node[above,xshift=0.15cm,yshift=-0.1cm]{$x$} to[out=0,in=90] (4,-17);
\draw[usual] (2,-7) to[out=270,in=180] (3,-8)node[below,xshift=0.15cm,yshift=0.05cm]{$x$} to[out=0,in=270] (4,-7);
\draw[usual] (-1.5,-17) to[out=90,in=270] (1.5,-15)node[left,xshift=-0.4cm,yshift=-0.2cm]{$m$};
\draw[usual] (7.5,-17) to[out=90,in=270] (4.5,-15);
\draw[usual] (7.9,-17) to[out=90,in=60] (4.5,-13)node[right,xshift=0.5cm,yshift=-0.2cm]{$1$};
\draw[usual] (-1.5,-7) to[out=270,in=90] (1.5,-9)node[left,xshift=-0.4cm,yshift=0.05cm]{$n$};
\draw[usual] (7.9,-7) to[out=270,in=300] (4.5,-11)node[right,xshift=0.5cm,yshift=0.05cm]{$1$};
\draw[usual] (7.5,-7) to[out=270,in=90] (4.5,-9);
\draw[usual] (1.5,-13) to (1.5,-11);
\end{tikzpicture}$$
by applying \cite[Lemma 4.30]{STWZ} to the right hand-side twice.

By grouping together identical terms, it only remains to analyze the behavior under specialization of 
$$\Big(\frac{[m+n-1]_{\mathbbm{v}}}{[m+n]_{\mathbbm{v}}}\cdot\tilde{\lambda}^{u-1}_{v-1,w-1}\cdot\tilde{\mathsf{d}}^{u-1}_{v-1,w-1} + \frac{[m+n]_{\mathbbm{v}}^2}{[m]_{\mathbbm{v}}[n]_{\mathbbm{v}}}\cdot\tilde{\lambda}^{u-3}_{v-1,w-1}\cdot\tilde{\mathsf{d}}^{u-3}_{v-1,w-1}\Big)\cdot\begin{tikzpicture}[anchorbase,scale=0.189,tinynodes]
\draw[JW] (-2.5,-7) rectangle (2.5,-5);
\node at (0,-6.2) {$\qjwm[w{-}1]$};
\draw[JW] (3.5,-7) rectangle (8.5,-5);
\node at (6,-6.2) {$\qjwm[v{-}1]$};
\draw[JW] (-2.5,-19) rectangle (2.5,-17);
\node at (0,-18.2) {$\qjwm[v{-}1]$};
\draw[JW] (3.5,-19) rectangle (8.5,-17);
\node at (6,-18.2) {$\qjwm[w{-}1]$};
\draw[JW] (0.5,-11) rectangle (5.5,-9);
\node at (3,-10.2) {$\qjwm[u-2]$};
\draw[JW] (0.5,-15) rectangle (5.5,-13);
\node at (3,-14.2) {$\qjwm[u-2]$};
\draw[usual] (2,-17) to[out=90,in=180] (3,-16)node[above,xshift=0.15cm,yshift=-0.1cm]{$x$} to[out=0,in=90] (4,-17);
\draw[usual] (2,-7) to[out=270,in=180] (3,-8)node[below,xshift=0.15cm,yshift=0.05cm]{$x$} to[out=0,in=270] (4,-7);
\draw[usual] (-1.5,-17) to[out=90,in=270] (1.5,-15)node[left,xshift=-0.4cm,yshift=-0.2cm]{$m$};
\draw[usual] (7.5,-17) to[out=90,in=270] (4.5,-15);
\draw[usual] (7.9,-17) to[out=90,in=60] (4.5,-13)node[right,xshift=0.5cm,yshift=-0.2cm]{$1$};
\draw[usual] (-1.5,-7) to[out=270,in=90] (1.5,-9)node[left,xshift=-0.4cm,yshift=0.05cm]{$n$};
\draw[usual] (7.9,-7) to[out=270,in=300] (4.5,-11)node[right,xshift=0.5cm,yshift=0.05cm]{$1$};
\draw[usual] (7.5,-7) to[out=270,in=90] (4.5,-9);
\draw[usual] (1.5,-13) to (1.5,-11);
\end{tikzpicture}$$
for any $p^{(2)}$-admissible $u$. The last generic Temperley-Lieb diagram above specializes to $\mathrm{L}^0\check{\mathrm{X}}^{u-1}_{v-1,w-1}$, so we only have to argue that the coefficient specializes appropriately.
This coefficient may be rewritten as $$\Big(\frac{[m+n-1]_{\mathbbm{v}}}{[m+n]_{\mathbbm{v}}} + \mathbbm{v}^{-(m+n)}\cdot\frac{[m+n+1]_{\mathbbm{v}}[x+1]_{\mathbbm{v}}}{[m+n]_{\mathbbm{v}}[m+n+x+1]_{\mathbbm{v}}}\Big)\cdot \tilde{\lambda}^{u-1}_{v-1,w-1}\cdot\tilde{\mathsf{d}}^{u-1}_{v-1,w-1}\,.$$
But, we have 
\begin{align*}
\frac{[m+n-1]_{\mathbbm{v}}}{[m+n]_{\mathbbm{v}}} + \mathbbm{v}^{-(m+n)}\cdot & \frac{[m+n+1]_{\mathbbm{v}}[x+1]_{\mathbbm{v}}}{[m+n]_{\mathbbm{v}}[m+n+x+1]_{\mathbbm{v}}} = \\ & = \frac{[m+n-1]_{\mathbbm{v}} + \mathbbm{v}^{-(m+n)}}{[m+n]_{\mathbbm{v}}} + \mathbbm{v}^{-(m+n)}\cdot\frac{[x]_{\mathbbm{v}}}{[m+n+x+1]_{\mathbbm{v}}}\\ & = \mathbbm{v}^{-1} + \mathbbm{v}^{-(m+n)}\cdot\frac{[x]_{\mathbbm{v}}}{[m+n+x+1]_{\mathbbm{v}}}.
\end{align*} As $x$ is even and $m+n+x+1$ is odd because $u$ is $p^{(2)}$-admissible for $v$ and $w$, the second term of the last equality vanishes upon specialization, which yields the desired formula.
\end{proof}

\subsection{Encirclements}

For any indecomposable projective object $\mathbb{T}_{\zeta}(w-1)$, we consider the endomorphism given graphically by encircling the identity morphism on $\mathbb{T}_{\zeta}(w-1)$ with a circle labeled by $\mathbb{T}_{\zeta}(v-1)$. Said differently, this is the endomorphism given by taking the right partial trace $\mathsf{pTr}^R_{v-1}$ of the double braiding $\beta_{w-1,v-1}^2$.

\begin{Lemma}\label{lem:encirclingbyasimple}
Let $v=[a_1,0]_{p,\ell}$ and $w=[b_1,b_0]_{p,\ell}$ with $a_1,b_1\geq 1$.
\begin{enumerate}
\item If $b_0 = 0$, we have $$\begin{tikzpicture}[anchorbase,scale=0.5,tinynodes]
\draw[usual] (0,0.2) arc (-180:0:1 and 0.5); 
\draw[usual, crossline] (1,-1) -- ++(0,2);
\draw[usual, crossline] (0,0.2) arc (180:0:1 and 0.5);
\draw[pJW] (-0.6,-0.1) rectangle (0.6,0.5);
\node at (0,0.15) {$\pjwm[v{-}1]$};
\draw[pJW] (0.4,-1.3) rectangle (1.6,-0.7);
\node at (1,-1.05) {$\pjwm[w{-}1]$};
\end{tikzpicture} =
\mathsf{pTr}^R_{v-1}(\beta_{w-1,v-1}^2)=v(-1)^{b_1-1}\cdot \mathrm{E}_{w-1}.$$
\item If $b_0 = 1$, we have $$\begin{tikzpicture}[anchorbase,scale=0.5,tinynodes]
\draw[usual] (0,0.2) arc (-180:0:1 and 0.5); 
\draw[usual, crossline] (1,-1) -- ++(0,2);
\draw[usual, crossline] (0,0.2) arc (180:0:1 and 0.5);
\draw[pJW] (-0.6,-0.1) rectangle (0.6,0.5);
\node at (0,0.15) {$\pjwm[v{-}1]$};
\draw[pJW] (0.4,-1.3) rectangle (1.6,-0.7);
\node at (1,-1.05) {$\pjwm[w{-}1]$};
\end{tikzpicture} = \mathsf{pTr}^R_{v-1}(\beta_{w-1,v-1}^2)= 2v(-1)^{(a_1-1)+(b_1-1)}\cdot \mathrm{L}^0_{w-1}.$$
\end{enumerate}
\end{Lemma}
\begin{proof}
If $b_0=0$, we proceed by induction on $a_1$. For $a_1=1$, i.e. $v-1=1$, we are encircling with the generator $\mathbb{T}_{\zeta}(1)$ and can therefore use Proposition \ref{prop:braidingT1}, which yields:
\begin{equation*}
\begin{tikzpicture}[anchorbase,scale=0.5,tinynodes]
\draw[usual] (0,0.2) arc (-180:0:1 and 0.5); 
\draw[usual, crossline] (1,-1) -- ++(0,2);
\draw[usual, crossline] (0,0.2) arc (180:0:1 and 0.5);
\draw[pJW] (0.4,-1.3) rectangle (1.6,-0.7);
\node at (1,-1.05) {$\pjwm[w{-}1]$};
\end{tikzpicture} = \mathsf{pTr}^R_{1}(\zeta^{w-1}\cdot \mathrm{E}_w + 2\zeta^{w-2}\cdot \mathrm{L}_w^0) = 2\zeta^{w-2} \mathrm{E}_{w-1} =v(-1)^{b_1-1}\cdot \mathrm{E}_{w-1}\ .
\end{equation*}
Note that this also follows directly from \cite[Lemma 5.10]{STWZ}.
Now, using that $$\Tz(v-1)\otimes \Tz(1) \otimes \Tz(1) \simeq \Tz(v+1)\oplus \Tz(v-1)\oplus\Tz(v-1)\oplus\Tz(v-3)\,,$$ where the last term is omitted if $v=2$ and the first if $v=2p-2$, we obtain the following equalities:
\begin{equation*}
\begin{tikzpicture}[anchorbase,scale=0.5,tinynodes]
\draw[usual] (0,0.2) arc (-180:0:1 and 0.5); 
\draw[usual, crossline] (1,-1) -- ++(0,2);
\draw[usual, crossline] (0,0.2) arc (180:0:1 and 0.5);
\draw[pJW] (-0.6,-0.1) rectangle (0.6,0.5);
\node at (0,0.15) {$\pjwm[v{+}1]$};
\draw[pJW] (0.4,-1.3) rectangle (1.6,-0.7);
\node at (1,-1.05) {$\pjwm[w{-}1]$};
\end{tikzpicture} = 
\begin{tikzpicture}[anchorbase,scale=0.5,tinynodes]
\draw[usual] (0,1.2) arc (-180:0:1 and 0.5); 
\draw[usual, double] (0,0.2) arc (-180:0:1 and 0.3); 
\draw[usual, crossline] (1,-1) -- ++(0,3);
\draw[usual, crossline, double] (0,0.2) arc (180:0:1 and 0.3);
\draw[usual, crossline] (0,1.2) arc (180:0:1 and 0.5);
\draw[pJW] (-0.6,0.9) rectangle (0.6,1.5);
\node at (0,1.15) {$\pjwm[v{-}1]$};
\draw[pJW] (0.4,-1.3) rectangle (1.6,-0.7);
\node at (1,-1.05) {$\pjwm[w{-}1]$};
\end{tikzpicture} 
- 2 \cdot
\begin{tikzpicture}[anchorbase,scale=0.5,tinynodes]
\draw[usual] (0,0.2) arc (-180:0:1 and 0.5); 
\draw[usual, crossline] (1,-1) -- ++(0,2);
\draw[usual, crossline] (0,0.2) arc (180:0:1 and 0.5);
\draw[pJW] (-0.6,-0.1) rectangle (0.6,0.5);
\node at (0,0.15) {$\pjwm[v{-}1]$};
\draw[pJW] (0.4,-1.3) rectangle (1.6,-0.7);
\node at (1,-1.05) {$\pjwm[w{-}1]$};
\end{tikzpicture}
-
\begin{tikzpicture}[anchorbase,scale=0.5,tinynodes]
\draw[usual] (0,0.2) arc (-180:0:1 and 0.5); 
\draw[usual, crossline] (1,-1) -- ++(0,2);
\draw[usual, crossline] (0,0.2) arc (180:0:1 and 0.5);
\draw[pJW] (-0.6,-0.1) rectangle (0.6,0.5);
\node at (0,0.15) {$\pjwm[v{-}3]$};
\draw[pJW] (0.4,-1.3) rectangle (1.6,-0.7);
\node at (1,-1.05) {$\pjwm[w{-}1]$};
\end{tikzpicture} \ .
\end{equation*}
By induction, the right hand-side equals
$$(-1)^{b_1-1}(2\cdot 2 \cdot v-2\cdot v-(v-2))\cdot \mathrm{E}_{w-1} = (-1)^{b_1-1}(v+2)\cdot \mathrm{E}_{w-1}\ .$$
Note that the last term $(-1)^{b_1-1}(v-2)$ already vanishes for $v=2$, so we do not need to treat it separately.

If $b_0=1$, let us write
$$\begin{tikzpicture}[anchorbase,scale=0.5,tinynodes]
\draw[usual] (0,0.2) arc (-180:0:1 and 0.5); 
\draw[usual, crossline] (1,-1) -- ++(0,2);
\draw[usual, crossline] (0,0.2) arc (180:0:1 and 0.5);
\draw[pJW] (-0.6,-0.1) rectangle (0.6,0.5);
\node at (0,0.15) {$\pjwm[v{-}1]$};
\draw[pJW] (0.4,-1.3) rectangle (1.6,-0.7);
\node at (1,-1.05) {$\pjwm[w{-}1]$};
\end{tikzpicture}
= \lambda_{v,w}\cdot \mathrm{E}_{w-1} + \mu_{v,w}\cdot \mathrm{L}^0_{w-1}\,,$$
for some scalars $\lambda_{v,w}$ and $\mu_{v,w}$.
We can use that $\mathbb{T}_{\zeta}(w-1) = \mathbb{T}_{\zeta}(w-2) \otimes \mathbb{T}_{\zeta}(1)$ in combination with Proposition \ref{prop:braidingT1}, so as to find that 
\begin{equation*}
\begin{tikzpicture}[anchorbase,scale=0.5,tinynodes]
\draw[usual] (0,0.2) arc (180:0:1 and 0.5); 
\draw[usual, crossline] (1,-1) -- ++(0,2);
\draw[usual, crossline] (1.6,-1.5) -- ++(0,2.5);
\draw[usual, crossline] (0,0.2) arc (-180:0:1 and 0.5);
\draw[pJW] (-0.6,-0.1) rectangle (0.6,0.5);
\node at (0,0.15) {$\pjwm[v{-}1]$};
\draw[pJW] (0.1,-1.3) rectangle (1.3,-0.7);
\node at (0.7,-1.05) {$\pjwm[w{-}2]$};
\end{tikzpicture} 
= \zeta^{v-1}\cdot 
\begin{tikzpicture}[anchorbase,scale=0.5,tinynodes]
\draw[usual] (0,0.2) arc (180:0:1 and 0.5); 
\draw[usual, crossline] (1,-1) -- ++(0,2);
\draw[usual, crossline] (2.2,-1.5) -- ++(0,2.5);
\draw[usual, crossline] (0,0.2) arc (-180:0:1 and 0.5);
\draw[pJW] (-0.6,-0.1) rectangle (0.6,0.5);
\node at (0,0.15) {$\pjwm[v{-}1]$};
\draw[pJW] (0.4,-1.3) rectangle (1.6,-0.7);
\node at (1,-1.05) {$\pjwm[w{-}2]$};
\end{tikzpicture}
+ 2\zeta^{v-2}\cdot 
\begin{tikzpicture}[anchorbase,scale=0.5,tinynodes]
\draw[usual] (0,0.2) arc (180:0:1 and 0.5); 
\draw[usual, crossline] (1,-1) -- ++(0,2);
\draw[usual, crossline] (2.2,-1.5) -- ++(0,2.5);
\draw[usual, crossline] (0,0.2) arc (-180:0:1 and 0.5);
\draw[pJW] (-0.6,-0.1) rectangle (0.6,0.5);
\node at (0,0.15) {$\pjwm[v{-}1]$};
\draw[pJW] (0.4,-1.3) rectangle (1.6,-0.7);
\node at (1,-1.05) {$\pjwm[w{-}2]$};
\ptr{1.5}{0.3}{$\mathrm{L}^0_v$}{-2.8}{0.2}
\end{tikzpicture}\,.
\end{equation*}
In particular, in order to identify $\mu_{v,w}$, it is enough to compute
$$\mu_{v,w}\cdot\mathrm{E}_{w-2} = \mathsf{pTr}^R_{1}\left( 
\begin{tikzpicture}[anchorbase,scale=0.5,tinynodes]
\draw[usual] (0,0.2) arc (180:0:1 and 0.5); 
\draw[usual, crossline] (1,-1) -- ++(0,2);
\draw[usual, crossline] (0,0.2) arc (-180:0:1 and 0.5);
\draw[pJW] (-0.6,-0.1) rectangle (0.6,0.5);
\node at (0,0.15) {$\pjwm[v{-}1]$};
\draw[pJW] (0.4,-1.3) rectangle (1.6,-0.7);
\node at (1,-1.05) {$\pjwm[w{-}1]$};
\end{tikzpicture}\right) = 0 + 2\zeta^{v-2}\cdot 
\begin{tikzpicture}[anchorbase,scale=0.5,tinynodes]
\draw[usual] (0,0.2) arc (180:0:1 and 0.5); 
\draw[usual, crossline] (1,-1) -- ++(0,2);
\draw[usual, crossline] (0,0.2) arc (-180:0:1 and 0.5);
\draw[pJW] (-0.6,-0.1) rectangle (0.6,0.5);
\node at (0,0.15) {$\pjwm[v{-}1]$};
\draw[pJW] (0.4,-1.3) rectangle (1.6,-0.7);
\node at (1,-1.05) {$\pjwm[w{-}2]$};
\end{tikzpicture} =  2\zeta^{v-2}v(-1)^{b_1-1}\cdot \mathrm{E}_{w-2}\,,$$
using the previously treated case $b_0=0$ in the last equality. It follows that $\mu_{v,w} = 2v(-1)^{(a_1-1)+(b_1-1)}$.
Finally, in order to compute $\lambda_{v,w}$, we note that $$\lambda_{v,w}\cdot \mathrm{L}^0_{w-1}=
\begin{tikzpicture}[anchorbase,scale=0.5,tinynodes]
\draw[usual] (0,0.2) arc (180:0:1 and 0.5); 
\draw[usual, crossline] (1,-1) -- ++(0,2);
\draw[usual, crossline] (0,0.2) arc (-180:0:1 and 0.5);
\draw[pJW] (-0.6,-0.1) rectangle (0.6,0.5);
\node at (0,0.15) {$\pjwm[v{-}1]$};
\draw[pJW] (0.4,-1.3) rectangle (1.6,-0.7);
\node at (1,-1.05) {$\pjwm[w{-}1]$};
\ptr{1.5}{0.3}{$\mathrm{L}^0$}{-1.8}{1.3}
\end{tikzpicture}
=
\begin{tikzpicture}[anchorbase,scale=0.5,tinynodes]
\draw[usual] (0,0.2) arc (180:0:1 and 0.5); 
\draw[usual, crossline] (1,-1) -- ++(0,2);
\draw[usual, crossline] (0,0.2) arc (-180:0:1 and 0.5);
\draw[pJW] (-0.6,-0.1) rectangle (0.6,0.5);
\node at (0,0.15) {$\pjwm[v{-}1]$};
\ptru{1.5}{0.6}{$\mathrm{U}^0$}{-1.8}{1}
\ptrd{1.5}{0.6}{$\mathrm{D}^0$}{-1.8}{-0.7}
\end{tikzpicture} = \left\{ \begin{array}{ll} \lambda_{v,w-2}\cdot\mathrm{L}^0_{w-1}\,, & \text{if }w>3\,, \\
0\,, & \text{if }w=3\,,
\end{array} \right.$$ and, by induction, we find $\lambda_{v,w} = 0$ as claimed.
\end{proof}

\subsection{Hopf Pairings}

In order to compute the value of the invariant in practice, it is useful to have an explicit formula for certain morphisms which we dub Hopf pairings.

\begin{Definition} Let $v=[a_1,a_0]_{p,\ell}$ and $w=[b_1,b_0]_{p,\ell}$, the Hopf pairing on $\Tz(v-1)$ and $\Tz(w-1)$ is the morphism 
    $$\mathrm{H}_{v-1}^{w-1} := \begin{tikzpicture}[anchorbase,scale=0.25,tinynodes]
\draw[usual] (7,0) to[out=270,in=0] (4.5,-3);
\draw[usual,crossline] (7,-4) to[out=90,in=0] (4.5,-1);
\draw[usual] (4.5,-1) to[out=180,in=90] (2,-4);
\draw[usual,crossline] (4.5,-3) to[out=180,in=270] (2,0);
\draw[pJWl] (0,0) rectangle (4,1);
\node at (2,0.3) {$\pjwm[{w{-}1}]$};
\draw[pJW] (5,0) rectangle (9,1);
\node at (7,0.3) {$\pjwm[{w{-}1}]$};
\draw[pJWl] (0,-5) rectangle (4,-4);
\node at (2,-4.7) {$\pjwm[{v{-}1}]$};
\draw[pJW] (5,-5) rectangle (9,-4);
\node at (7,-4.7) {$\pjwm[{v{-}1}]$};
\end{tikzpicture}\in \Hom(\Tz(v-1)^*\otimes\Tz(v-1),\Tz(w-1)^*\otimes\Tz(w-1)) \ .$$
\end{Definition}

\begin{Theorem}\label{thm:HopfPairings}
 Let $v=[a_1,a_0]_{p,\ell}$ and $w=[b_1,b_0]_{p,\ell}$, then
 the Hopf pairing of $\Tz(v-1)$ and $\Tz(w-1)$ decomposes as follows:
 \begin{enumerate}
     \item If $a_0=0$ and $b_0=0$, the Hopf pairing of $\Tz(v-1)$ and $\Tz(w-1)$ is given by
     \begin{equation*}
         \mathrm{H}_{v-1}^{w-1} = (-1)^{a_1-1}(-1)^{b_1-1}\left( \zeta \cdot 
         \begin{tikzpicture}[anchorbase,scale=0.25,tinynodes]
\draw[usual] (7,0) to[out=270,in=0] (4.5,-1.5) node[above = -2pt]{$\pjwm[{w{-}1}]$}to[out=180,in=270] (2,0);
\draw[usual] (7,-4) to[out=90,in=0] (4.5,-2.5) node[below = 0pt]{$\pjwm[{v{-}1}]$}to[out=180,in=90] (2,-4);
\draw[JW] (0,0) rectangle (4,1);
\node at (2,0.3) {$\pjwm[{w{-}1}]$};
\draw[JW] (5,0) rectangle (9,1);
\node at (7,0.3) {$\pjwm[{w{-}1}]$};
\draw[JW] (0,-5) rectangle (4,-4);
\node at (2,-4.7) {$\pjwm[{v{-}1}]$};
\draw[JW] (5,-5) rectangle (9,-4);
\node at (7,-4.7) {$\pjwm[{v{-}1}]$};
\end{tikzpicture}
+2\cdot
\begin{tikzpicture}[anchorbase,scale=0.25,tinynodes]
\draw[usual] (7,0) to[out=270,in=0] (4.5,-1.5) node[above = -2pt]{$\pjwm[{w{-}2}]$}to[out=180,in=270] (2,0);
\draw[usual] (7,-4) to[out=90,in=0] (4.5,-2.5) node[below = 0pt]{$\pjwm[{v{-}2}]$}to[out=180,in=90] (2,-4);
\draw[usual] (1,0) --++(0,-4);
\draw[usual] (8,0) --++(0,-4);
\draw[JW] (0,0) rectangle (4,1);
\node at (2,0.3) {$\pjwm[{w{-}1}]$};
\draw[JW] (5,0) rectangle (9,1);
\node at (7,0.3) {$\pjwm[{w{-}1}]$};
\draw[JW] (0,-5) rectangle (4,-4);
\node at (2,-4.7) {$\pjwm[{v{-}1}]$};
\draw[JW] (5,-5) rectangle (9,-4);
\node at (7,-4.7) {$\pjwm[{v{-}1}]$};
\end{tikzpicture}
\right)\,.
     \end{equation*}

     \item If $a_0=0$ and $b_0=1$, the Hopf pairing of $\Tz(v-1)$ and $\Tz(w-1)$ is given by
     \begin{equation*} \ \hskip-30pt
         \mathrm{H}_{v-1}^{w-1} = (-1)^{b_1-1}\left( - 
         \begin{tikzpicture}[anchorbase,scale=0.23,tinynodes]
\draw[usual] (7,0) to[out=270,in=0] (4.5,-1) node[above = -2pt]{$\pjwm[{w{-}2}]$}to[out=180,in=270] (2,0);
\draw[usual] (8.5,1)--(8.5,0) to[out=270,in=0] (4.5,-1.9) to[out=180,in=270] (0.5,0)--(0.5,1);
\draw[usual] (7,-4) to[out=90,in=0] (4.5,-3) node[below = 0pt]{$\pjwm[{v{-}1}]$}to[out=180,in=90] (2,-4);
\draw[JW] (1,0) rectangle (4,1);
\node at (2.5,0.3) {$\pjwm[{w{-}2}]$};
\draw[JW] (5,0) rectangle (8,1);
\node at (6.5,0.3) {$\pjwm[{w{-}2}]$};
\draw[JW] (0,-5) rectangle (4,-4);
\node at (2,-4.7) {$\pjwm[{v{-}1}]$};
\draw[JW] (5,-5) rectangle (9,-4);
\node at (7,-4.7) {$\pjwm[{v{-}1}]$};
\end{tikzpicture}
+2\cdot
\begin{tikzpicture}[anchorbase,scale=0.23,tinynodes]
\draw[usual] (7,0) to[out=270,in=0] (4.5,-1) node[above = -2pt]{$\pjwm[{w{-}3}]$}to[out=180,in=270] (2,0);
\draw[usual] (7,-4) to[out=90,in=0] (4.5,-3) node[below = 0pt]{$\pjwm[{v{-}2}]$}to[out=180,in=90] (2,-4);
\draw[usual] (8.5,1)--(8.5,0) to[out=270,in=270] (7.5,0); 
\draw[usual] (1.5,0) to[out=270,in=90] (8.5,-4); 
\draw[usual] (0.5,1) -- (0.5,-4); 
\draw[JW] (1,0) rectangle (4,1);
\node at (2.5,0.3) {$\pjwm[{w{-}2}]$};
\draw[JW] (5,0) rectangle (8,1);
\node at (6.5,0.3) {$\pjwm[{w{-}2}]$};
\draw[JW] (0,-5) rectangle (4,-4);
\node at (2,-4.7) {$\pjwm[{v{-}1}]$};
\draw[JW] (5,-5) rectangle (9,-4);
\node at (7,-4.7) {$\pjwm[{v{-}1}]$};
\end{tikzpicture}
+2\cdot
\begin{tikzpicture}[anchorbase,xscale=-0.23,yscale=0.23,tinynodes]
\draw[usual] (7,0) to[out=270,in=0] (4.5,-1) node[above = -2pt]{$\pjwm[{w{-}3}]$}to[out=180,in=270] (2,0);
\draw[usual] (7,-4) to[out=90,in=0] (4.5,-3) node[below = 0pt]{$\pjwm[{v{-}2}]$}to[out=180,in=90] (2,-4);
\draw[usual] (8.5,1)--(8.5,0) to[out=270,in=270] (7.5,0); 
\draw[usual] (1.5,0) to[out=270,in=90] (8.5,-4); 
\draw[usual] (0.5,1) -- (0.5,-4); 
\draw[JW] (1,0) rectangle (4,1);
\node at (2.5,0.3) {$\pjwm[{w{-}2}]$};
\draw[JW] (5,0) rectangle (8,1);
\node at (6.5,0.3) {$\pjwm[{w{-}2}]$};
\draw[JW] (0,-5) rectangle (4,-4);
\node at (2,-4.7) {$\pjwm[{v{-}1}]$};
\draw[JW] (5,-5) rectangle (9,-4);
\node at (7,-4.7) {$\pjwm[{v{-}1}]$};
\end{tikzpicture}
+2 \zeta \cdot 
\begin{tikzpicture}[anchorbase,xscale=0.23,yscale=0.23,tinynodes]
\draw[usual] (7,0) to[out=270,in=0] (4.5,-1) node[above = -2pt]{$\pjwm[{w{-}3}]$}to[out=180,in=270] (2,0);
\draw[usual] (7,-4) to[out=90,in=0] (4.5,-3) node[below = 0pt]{$\pjwm[{v{-}1}]$}to[out=180,in=90] (2,-4);
\draw[usual] (8.5,1)--(8.5,0) to[out=270,in=270] (7.5,0); 
\draw[usual] (0.5,1)--(0.5,0) to[out=270,in=270] (1.5,0); 
\draw[JW] (1,0) rectangle (4,1);
\node at (2.5,0.3) {$\pjwm[{w{-}2}]$};
\draw[JW] (5,0) rectangle (8,1);
\node at (6.5,0.3) {$\pjwm[{w{-}2}]$};
\draw[JW] (0,-5) rectangle (4,-4);
\node at (2,-4.7) {$\pjwm[{v{-}1}]$};
\draw[JW] (5,-5) rectangle (9,-4);
\node at (7,-4.7) {$\pjwm[{v{-}1}]$};
\end{tikzpicture}
\right)\,.
\end{equation*}
     
    \item If $a_0=1$ and $b_0=0$, the Hopf pairing of $\Tz(v-1)$ and $\Tz(w-1)$ is obtained by rotating the expression above by 180 degrees, and changing the sign to $(-1)^{a_1-1}$.

     \item If $a_0=1$ and $b_0=1$, the Hopf pairing of $\Tz(v-1)$ and $\Tz(w-1)$ is given by 
     \begin{equation*} \ \hskip-30pt
         \mathrm{H}_{v-1}^{w-1} = \left( \begin{array}{ccc}
\begin{tikzpicture}[anchorbase,scale=0.25,tinynodes]
\draw[usual] (7,0) to[out=270,in=0] (4.5,-1) node[above = -2pt]{$\pjwm[{w{-}2}]$}to[out=180,in=270] (2,0);
\draw[usual] (7,-4) to[out=90,in=0] (4.5,-3) node[below = 0pt]{$\pjwm[{v{-}2}]$}to[out=180,in=90] (2,-4);
\draw[usual] (8.5,1)--(8.5,0) to[out=270,in=0] (4.5,-1.8) to[out=180,in=270] (0.5,0)--(0.5,1);
\draw[usual] (8.5,-5)--(8.5,-4) to[out=90,in=0] (4.5,-2.2) to[out=180,in=90] (0.5,-4)--(0.5,-5);
\draw[JW] (1,0) rectangle (4,1);
\node at (2.5,0.3) {$\pjwm[{w{-}2}]$};
\draw[JW] (5,0) rectangle (8,1);
\node at (6.5,0.3) {$\pjwm[{w{-}2}]$};
\draw[JW] (1,-5) rectangle (4,-4);
\node at (2.5,-4.7) {$\pjwm[{v{-}2}]$};
\draw[JW] (5,-5) rectangle (8,-4);
\node at (6.5,-4.7) {$\pjwm[{v{-}2}]$};
\end{tikzpicture} 
&
+2 \zeta\cdot
\begin{tikzpicture}[anchorbase,scale=0.25,tinynodes]
\draw[usual] (7,0) to[out=270,in=0] (4.5,-1) node[above = -2pt]{$\pjwm[{w{-}3}]$}to[out=180,in=270] (2,0);
\draw[usual] (7,-4) to[out=90,in=0] (4.5,-3) node[below = 0pt]{$\pjwm[{v{-}3}]$}to[out=180,in=90] (2,-4);
\draw[usual] (0.5,1)--(0.5,0) to[out=270,in=270] (1.5,0);
\draw[usual] (8.5,-5)--(8.5,-4) to[out=90,in=90] (7.5,-4);
\draw[usual] (0.5,-5)--(0.5,-3.5) to[out=90,in=270] (7.5,0);
\draw[usual] (1.5,-4) to[out=90,in=270] (8.5,-0.5)--(8.5,1);
\draw[JW] (1,0) rectangle (4,1);
\node at (2.5,0.3) {$\pjwm[{w{-}2}]$};
\draw[JW] (5,0) rectangle (8,1);
\node at (6.5,0.3) {$\pjwm[{w{-}2}]$};
\draw[JW] (1,-5) rectangle (4,-4);
\node at (2.5,-4.7) {$\pjwm[{v{-}2}]$};
\draw[JW] (5,-5) rectangle (8,-4);
\node at (6.5,-4.7) {$\pjwm[{v{-}2}]$};
\end{tikzpicture} 
&
+2 \zeta\cdot
\begin{tikzpicture}[anchorbase,xscale=-0.25,yscale=0.25,tinynodes]
\draw[usual] (7,0) to[out=270,in=0] (4.5,-1) node[above = -2pt]{$\pjwm[{w{-}3}]$}to[out=180,in=270] (2,0);
\draw[usual] (7,-4) to[out=90,in=0] (4.5,-3) node[below = 0pt]{$\pjwm[{v{-}3}]$}to[out=180,in=90] (2,-4);
\draw[usual] (0.5,1)--(0.5,0) to[out=270,in=270] (1.5,0);
\draw[usual] (8.5,-5)--(8.5,-4) to[out=90,in=90] (7.5,-4);
\draw[usual] (0.5,-5)--(0.5,-3.5) to[out=90,in=270] (7.5,0);
\draw[usual] (1.5,-4) to[out=90,in=270] (8.5,-0.5)--(8.5,1);
\draw[JW] (1,0) rectangle (4,1);
\node at (2.5,0.3) {$\pjwm[{w{-}2}]$};
\draw[JW] (5,0) rectangle (8,1);
\node at (6.5,0.3) {$\pjwm[{w{-}2}]$};
\draw[JW] (1,-5) rectangle (4,-4);
\node at (2.5,-4.7) {$\pjwm[{v{-}2}]$};
\draw[JW] (5,-5) rectangle (8,-4);
\node at (6.5,-4.7) {$\pjwm[{v{-}2}]$};
\end{tikzpicture} 
\\[30pt]
-2 \zeta\cdot
\begin{tikzpicture}[anchorbase,xscale=0.25,yscale=0.25,tinynodes]
\draw[usual] (7,0) to[out=270,in=0] (4.5,-1) node[above = -2pt]{$\pjwm[{w{-}3}]$}to[out=180,in=270] (2,0);
\draw[usual] (7,-4) to[out=90,in=0] (4.5,-3) node[below = 0pt]{$\pjwm[{v{-}3}]$}to[out=180,in=90] (2,-4);
\draw[usual] (0.5,1)--(0.5,0) to[out=270,in=270] (1.5,0);
\draw[usual] (0.5,-5)--(0.5,-4) to[out=90,in=90] (1.5,-4);
\draw[usual] (8.5,-5)-- (8.5,1);
\draw[usual] (7.5,-4)-- (7.5,0);
\draw[JW] (1,0) rectangle (4,1);
\node at (2.5,0.3) {$\pjwm[{w{-}2}]$};
\draw[JW] (5,0) rectangle (8,1);
\node at (6.5,0.3) {$\pjwm[{w{-}2}]$};
\draw[JW] (1,-5) rectangle (4,-4);
\node at (2.5,-4.7) {$\pjwm[{v{-}2}]$};
\draw[JW] (5,-5) rectangle (8,-4);
\node at (6.5,-4.7) {$\pjwm[{v{-}2}]$};
\end{tikzpicture} 
&
-2 \zeta\cdot
\begin{tikzpicture}[anchorbase,xscale=-0.25,yscale=0.25,tinynodes]
\draw[usual] (7,0) to[out=270,in=0] (4.5,-1) node[above = -2pt]{$\pjwm[{w{-}3}]$}to[out=180,in=270] (2,0);
\draw[usual] (7,-4) to[out=90,in=0] (4.5,-3) node[below = 0pt]{$\pjwm[{v{-}3}]$}to[out=180,in=90] (2,-4);
\draw[usual] (0.5,1)--(0.5,0) to[out=270,in=270] (1.5,0);
\draw[usual] (0.5,-5)--(0.5,-4) to[out=90,in=90] (1.5,-4);
\draw[usual] (8.5,-5)-- (8.5,1);
\draw[usual] (7.5,-4)-- (7.5,0);
\draw[JW] (1,0) rectangle (4,1);
\node at (2.5,0.3) {$\pjwm[{w{-}2}]$};
\draw[JW] (5,0) rectangle (8,1);
\node at (6.5,0.3) {$\pjwm[{w{-}2}]$};
\draw[JW] (1,-5) rectangle (4,-4);
\node at (2.5,-4.7) {$\pjwm[{v{-}2}]$};
\draw[JW] (5,-5) rectangle (8,-4);
\node at (6.5,-4.7) {$\pjwm[{v{-}2}]$};
\end{tikzpicture} 
&
+4\cdot
\begin{tikzpicture}[anchorbase,xscale=0.25,yscale=0.25,tinynodes]
\draw[usual] (7,0) to[out=270,in=0] (4.5,-1) node[above = -2pt]{$\pjwm[{w{-}3}]$}to[out=180,in=270] (2,0);
\draw[usual] (7,-4) to[out=90,in=0] (4.5,-3) node[below = 0pt]{$\pjwm[{v{-}3}]$}to[out=180,in=90] (2,-4);
\draw[usual] (0.5,1)--(0.5,0) to[out=270,in=270] (1.5,0);
\draw[usual] (0.5,-5)--(0.5,-4) to[out=90,in=90] (1.5,-4);
\draw[usual] (8.5,1)--(8.5,0) to[out=270,in=270] (7.5,0);
\draw[usual] (8.5,-5)--(8.5,-4) to[out=90,in=90] (7.5,-4);
\draw[JW] (1,0) rectangle (4,1);
\node at (2.5,0.3) {$\pjwm[{w{-}2}]$};
\draw[JW] (5,0) rectangle (8,1);
\node at (6.5,0.3) {$\pjwm[{w{-}2}]$};
\draw[JW] (1,-5) rectangle (4,-4);
\node at (2.5,-4.7) {$\pjwm[{v{-}2}]$};
\draw[JW] (5,-5) rectangle (8,-4);
\node at (6.5,-4.7) {$\pjwm[{v{-}2}]$};
\end{tikzpicture} 
\end{array} \right)\,.
\end{equation*}
    
 \end{enumerate}
In particular, Hopf pairing always factors through zero or two strands.
 \end{Theorem}

The proof will proceed by induction, and we first need to compute the first few cases. 
\begin{Lemma}\label{lem:Hopf111222}
For $v,w\leq3$, the Hopf pairings are given by
    \begin{gather*}
  \mathrm{H}_{1}^{1}  = 
\begin{tikzpicture}[anchorbase,xscale=0.2, yscale = 0.13]
\draw[usual] (1,0) to[out=90, in=180] (2,3); 
\draw[usual] (2,2) to[out=0, in=-90] (3,5); 
\draw[usual, crossline] (2,3) to[out=0, in=90] (3,0); 
\draw[usual, crossline] (1,5) to[out=-90, in=180] (2,2); 
\end{tikzpicture}
= \zeta \cdot \
\begin{tikzpicture}[anchorbase,xscale=0.2, yscale = 0.13]
\draw[usual] (1,0) arc(180:0:1 and 2); 
\draw[usual] (1,5) arc(-180:0:1 and 2); 
\end{tikzpicture}
+2 \cdot \
\begin{tikzpicture}[anchorbase,xscale=0.2, yscale = 0.13]
\draw[usual] (1,0) --++(0,5); 
\draw[usual] (3,0) --++(0,5); 
\end{tikzpicture}\,,
\\
\mathrm{H}_{2}^{1} = 
\begin{tikzpicture}[anchorbase,xscale=0.2, yscale = 0.2]
\draw[usual] (0,0) to[out=90, in=180] (2,4); 
\draw[usual] (1,0) to[out=90, in=180] (2,3); 
\draw[usual] (2,2) to[out=0, in=-90] (3,5); 
\draw[usual, crossline] (2,4) to[out=0, in=90] (4,0); 
\draw[usual, crossline] (2,3) to[out=0, in=90] (3,0); 
\draw[usual, crossline] (1,5) to[out=-90, in=180] (2,2); 
\end{tikzpicture}
= -
\begin{tikzpicture}[anchorbase,xscale=0.2, yscale = 0.2]
\draw[usual] (0,0) arc(180:0:2 and 2.5); 
\draw[usual] (1,0) arc(180:0:1 and 1.5); 
\draw[usual] (1,5) arc(-180:0:1 and 1.5); 
\end{tikzpicture}
+2  \cdot \
\begin{tikzpicture}[anchorbase,xscale=0.2, yscale = 0.2]
\draw[usual] (0,0) to[out=90, in=-90] (1,5); 
\draw[usual] (1,0) to[out=90, in=-90] (3,5); 
\draw[usual] (3,0) arc(180:0:0.5 and 1.5); 
\end{tikzpicture}
+2 \cdot \
\begin{tikzpicture}[anchorbase,xscale=-0.2, yscale = 0.2]
\draw[usual] (0,0) to[out=90, in=-90] (1,5); 
\draw[usual] (1,0) to[out=90, in=-90] (3,5); 
\draw[usual] (3,0) arc(180:0:0.5 and 1.5); 
\end{tikzpicture}
+ 2 \zeta \cdot
\begin{tikzpicture}[anchorbase,xscale=0.2, yscale = 0.2]
\draw[usual] (0,0) arc(180:0:0.5 and 1.5); 
\draw[usual] (1,5) arc(-180:0:1 and 1.5); 
\draw[usual] (3,0) arc(180:0:0.5 and 1.5); 
\end{tikzpicture}\,,
\\
\mathrm{H}_{2}^{2} =
\begin{tikzpicture}[anchorbase,xscale=0.2, yscale = 0.2]
\draw[usual] (0,0) to[out=90, in=180] (2,4); 
\draw[usual] (1,0) to[out=90, in=180] (2,3); 
\draw[usual] (2,2) to[out=0, in=-90] (3,5); 
\draw[usual] (2,1) to[out=0, in=-90] (4,5); 
\draw[usual, crossline] (2,4) to[out=0, in=90] (4,0); 
\draw[usual, crossline] (2,3) to[out=0, in=90] (3,0); 
\draw[usual, crossline] (0,5) to[out=-90, in=180] (2,1); 
\draw[usual, crossline] (1,5) to[out=-90, in=180] (2,2); 
\end{tikzpicture}
=
\begin{tikzpicture}[anchorbase,xscale=0.2, yscale = 0.2]
\draw[usual] (0,0) arc(180:0:2); 
\draw[usual] (1,0) arc(180:0:1); 
\draw[usual] (0,5) arc(-180:0:2); 
\draw[usual] (1,5) arc(-180:0:1); 
\end{tikzpicture}
+2 \zeta \cdot \
\begin{tikzpicture}[anchorbase,xscale=0.2, yscale = 0.2]
\draw[usual] (0,0) to[out=90, in=-90] (3,5); 
\draw[usual] (1,0) to[out=90, in=-90] (4,5); 
\draw[usual] (3,0) arc(180:0:0.5 and 1.5); 
\draw[usual] (0,5) arc(-180:0:0.5 and 1.5); 
\end{tikzpicture}
+2 \zeta\cdot \
\begin{tikzpicture}[anchorbase,xscale=-0.2, yscale = 0.2]
\draw[usual] (0,0) to[out=90, in=-90] (3,5); 
\draw[usual] (1,0) to[out=90, in=-90] (4,5); 
\draw[usual] (3,0) arc(180:0:0.5 and 1.5); 
\draw[usual] (0,5) arc(-180:0:0.5 and 1.5); 
\end{tikzpicture}
-2 \zeta\cdot \
\begin{tikzpicture}[anchorbase,xscale=0.2, yscale = 0.2]
\draw[usual] (2.5,0) --++(0,5); 
\draw[usual] (3.5,0) --++(0,5); 
\draw[usual] (0,0) arc(180:0:0.5 and 1.5); 
\draw[usual] (0,5) arc(-180:0:0.5 and 1.5); 
\end{tikzpicture}
-2 \zeta\cdot \
\begin{tikzpicture}[anchorbase,xscale=-0.2, yscale = 0.2]
\draw[usual] (2.5,0) --++(0,5); 
\draw[usual] (3.5,0) --++(0,5); 
\draw[usual] (0,0) arc(180:0:0.5 and 1.5); 
\draw[usual] (0,5) arc(-180:0:0.5 and 1.5); 
\end{tikzpicture}
+4\cdot \
\begin{tikzpicture}[anchorbase,xscale=0.2, yscale = 0.2]
\draw[usual] (0,0) arc(180:0:0.5 and 1.5); 
\draw[usual] (0,5) arc(-180:0:0.5 and 1.5); 
\draw[usual] (2.5,0) arc(180:0:0.5 and 1.5); 
\draw[usual] (2.5,5) arc(-180:0:0.5 and 1.5); 
\end{tikzpicture}\,.
    \end{gather*}
\end{Lemma}
\begin{proof}
    The first equality is a straightforward computation by resolving both crossings. The second follows by using the first one twice and resolving two crossings. Finally, the last equality cam also be derived by using the second one twice and resolving resulting crossings.
\end{proof}
The main take-away of these preliminary computations is that the Hopf pairing of $\Tz(2)$ with itself is equal to the sum of the unlinked evaluation and coevaluation together with terms involving a small cup or cap in $\Tz(2)$.

\begin{proof}[Proof of Theorem \ref{thm:HopfPairings}:]
We begin with the case $v=3$, so that we are pairing with $\Tz(2)$. For $w=[1,0]_{p,\ell}$ and $w=[1,1]_{p,\ell}$, the result is Lemma \ref{lem:Hopf111222}. Suppose now that $w=[b_1,b_0]_{p,\ell}$ with $b_1\geq 2$, so that $\Tz(w-1)$ has at least three strands. We can decompose the Hopf pairing of $\Tz(w-1)$ with $\Tz(2)$ as successively pairing $\Tz(2)$ and $\Tz(w-3)$ with $\Tz(2)$. Graphically, this corresponds to:
\begin{equation*}
    \mathrm{H}_{2}^{w-1} = \begin{tikzpicture}[anchorbase,scale=0.25,tinynodes]
\draw[usual] (6.5,0) to[out=270,in=0] (4.5,-2);
\draw[usual] (7,0) to[out=270,in=0] (4.5,-2.5);
\draw[usual] (8,0) to[out=270,in=0] (4.5,-3.5) node[below]{$\pjwm[{w{-}3}]$};
\draw[usual,crossline] (7,-4) to[out=90,in=0] (4.5,-1)to[out=180,in=90] (2,-4);
\draw[usual,crossline] (7.5,-4) to[out=90,in=0] (4.5,-0.5)to[out=180,in=90] (1.5,-4);
\draw[usual,crossline] (4.5,-2) to[out=180,in=270] (2.5,0);
\draw[usual,crossline] (4.5,-2.5) to[out=180,in=270] (2,0);
\draw[usual,crossline] (4.5,-3.5) to[out=180,in=270] (1,0);
\draw[pJWl] (0,0) rectangle (4,1);
\node at (2,0.3) {$\pjwm[{w{-}1}]$};
\draw[pJW] (5,0) rectangle (9,1);
\node at (7,0.3) {$\pjwm[{w{-}1}]$};
\end{tikzpicture} 
=
\begin{tikzpicture}[anchorbase,scale=0.25,tinynodes]
\draw[usual] (6.5,0) to[out=270,in=0] (4.5,-1);
\draw[usual] (7,0) to[out=270,in=0] (4.5,-1.5);
\draw[usual] (8,0) to[out=270,in=0] (4.5,-3.5) node[below]{$\pjwm[{w{-}3}]$};
\draw[usual,crossline] (7,-4) to[out=90,in=0] (4.5,-2.5)to[out=180,in=90] (2,-4);
\draw[usual,crossline] (7.5,-4) to[out=90,in=0] (4.5,-2)to[out=180,in=90] (1.5,-4);
\draw[usual,crossline] (4.5,-1) to[out=180,in=270] (2.5,0);
\draw[usual,crossline] (4.5,-1.5) to[out=180,in=270] (2,0);
\draw[usual,crossline] (4.5,-3.5) to[out=180,in=270] (1,0);
\draw[pJWl] (0,0) rectangle (4,1);
\node at (2,0.3) {$\pjwm[{w{-}1}]$};
\draw[pJW] (5,0) rectangle (9,1);
\node at (7,0.3) {$\pjwm[{w{-}1}]$};
\end{tikzpicture}\,.
\end{equation*}
The second equality is a consequence of Lemma \ref{lem:Hopf111222} and Equation \eqref{eq:classicalbraidingabsorption}.
The claimed result for $\mathrm{H}_{2}^{w-1}$ therefore follows by induction.

The general case follows using an analogous argument. Suppose $v=[a_1,a_0]_{p,\ell}$ with $a_1\geq 2$. We may write the Hopf pairing of $\Tz(w-1)$ with $\Tz(v-1)$ as successively pairing $\Tz(w-1)$ with $\Tz(2)$ and $\Tz(v-3)$. Graphically, this is given by
\begin{equation*}
    \mathrm{H}_{v-1}^{w-1} = \begin{tikzpicture}[anchorbase,scale=-0.25,tinynodes]
\draw[usual] (6.5,0) to[out=270,in=0] (4.5,-2);
\draw[usual] (7,0) to[out=270,in=0] (4.5,-2.5);
\draw[usual] (8,0)to[out=270,in=0] (4.5,-3.5) ;
\node[rotate=65] at (8,-1.5) {$\pjwm[v{-}3]$};
\draw[usual,crossline] (7,-4) to[out=90,in=0] (4.5,-1)to[out=180,in=90] (2,-4);
\draw[usual,crossline] (4.5,-2) to[out=180,in=270] (2.5,0);
\draw[usual,crossline] (4.5,-2.5) to[out=180,in=270] (2,0);
\draw[usual,crossline] (4.5,-3.5) to[out=180,in=270] (1,0);
\draw[pJWl] (0,0) rectangle (4,1);
\node at (2,0.6) {$\pjwm[{v{-}1}]$};
\draw[pJW] (5,0) rectangle (9,1);
\node at (7,0.6) {$\pjwm[{v{-}1}]$};
\draw[pJWl] (0,-5) rectangle (4,-4);
\node at (2,-4.4) {$\pjwm[{w{-}1}]$};
\draw[pJW] (5,-5) rectangle (9,-4);
\node at (7,-4.4) {$\pjwm[{w{-}1}]$};
\end{tikzpicture} 
= (-1)^{b_0-1}
\begin{tikzpicture}[anchorbase,scale=-0.25,tinynodes]
\draw[usual] (6.5,0) to[out=270,in=0] (4.5,-1);
\draw[usual] (7,0) to[out=270,in=0] (4.5,-1.5);
\draw[usual] (8,0) to[out=270,in=0] (4.5,-3.5);
\node[rotate=65] at (8,-1.5) {$\pjwm[v{-}3]$};
\draw[usual,crossline] (7,-4) to[out=90,in=0] (4.5,-2.5)to[out=180,in=90] (2,-4);
\draw[usual,crossline] (4.5,-1) to[out=180,in=270] (2.5,0);
\draw[usual,crossline] (4.5,-1.5) to[out=180,in=270] (2,0);
\draw[usual,crossline] (4.5,-3.5) to[out=180,in=270] (1,0);
\draw[pJWl] (0,0) rectangle (4,1);
\node at (2,0.6) {$\pjwm[{v{-}1}]$};
\draw[pJW] (5,0) rectangle (9,1);
\node at (7,0.6) {$\pjwm[{v{-}1}]$};
\draw[pJWl] (0,-5) rectangle (4,-4);
\node at (2,-4.4) {$\pjwm[{w{-}1}]$};
\draw[pJW] (5,-5) rectangle (9,-4);
\node at (7,-4.4) {$\pjwm[{w{-}1}]$};
\end{tikzpicture}\,. 
\end{equation*}
The second equality above follows by induction and Equation \eqref{eq:classicalbraidingabsorption}. This concludes the proof of the result.
\end{proof}

\section{The Associated Topological Invariants}\label{sec:TopoHandles}

We now bring together the previous two sections. More precisely, we explain how to compute the invariant of 4-dimensional 2-handlebodies of \cite{CGHPM} associated to the mixed Verlinde categories $\Ver_{p^{(2)}}^{\zeta^{1/2}}$ with $\zeta$ a primitive fourth root of unity.

\subsection{The Modified Trace, or the 4-Handle}

It was established in \cite[Theorem 4.7]{STWZ} that there is a non-degenerate modified trace $\mathsf{t}$ on the ideal of projective objects $\mathrm{Proj} \subseteq \Ver_{p^{(2)}}^{\zeta^{1/2}}$, whose indecomposable objects are $\mathbb{T}_{\zeta}(1),...,\mathbb{T}_{\zeta}(p^{(2)}-2)$. The partial trace $\mathsf{t}$ is uniquely determined by its value on $\mathrm{E}_1$, the identity morphism on $\mathbb{T}_{\zeta}(1)$, which we set to be $$\mathsf{t}(\mathrm{E}_1) = +1\,.$$

For use below, we now determine the value of the modified trace on all the endomorphisms of projective objects. In particular, recall that if $v=[a_1,0]_{p,\ell}$, then $\mathbb{T}_{\zeta}(v-1)$ is simple and its endomorphism algebra is spanned by $\mathrm{E}_{v-1}$. On the other hand, if $v=[a_1,1]_{p,\ell}$, the endomorphism algebra of $\mathbb{T}_{\zeta}(v-1)$ is spanned by $\mathrm{E}_{v-1}$ along with the nilpotent endomorphism $\mathrm{L}_{v-1}$, see Section \ref{sub:linearstructure}.

\begin{Proposition}
Let $v=[a_1,a_0]_{p,\ell}$.
\begin{enumerate}
    \item If $a_0=0$, then we have
\begin{equation}\label{eq:modifiedtracesimpleprojective}
\mathsf{t}(\mathrm{E}_{v-1}) = (-1)^v[a_1]_{\zeta^2} = (-1)^{a_1-1}a_1\ .
\end{equation}
\item If $a_0 = 1$, then we have \begin{equation}\label{eq:modifiedtraceNONsimpleprojective}
\mathsf{t}(\mathrm{E}_{v-1}) = 0 \ , \quad \quad \mathsf{t}(\mathrm{L}^0_{v-1}) = \mathsf{t}(\mathrm{E}_{v-2}) = (-1)^{a_1-1}a_1 \ .
\end{equation}
\end{enumerate}
\end{Proposition}
\begin{proof}
    The first equality follows from the formula in \cite[Theorem 5.7]{STWZ} by recalling the convention of \cite[Example 2.10]{STWZ}, and noting that $\zeta^2 = -1$. Then, the equality $\mathsf{t}(\mathrm{E}_{v-1}) = 0$ whenever $a_0=1$ follows from the first part of Proposition \ref{prop:categorifiedfusionT1} together with the fact that the modified trace is compatible with the right partial trace as in Equation \eqref{eq:modifiedpartial}. The last equality follows analogously.
\end{proof}

We have reviewed above in Section \ref{sub:skeintheory} that the modified trace $\mathsf{t}$ yields an invariant $\tilde{\mathsf t}$ of closed admissible skeins in $S^3$ by cutting one projective strand open and taking the modified trace of the induced endomorphism. In particular, we have that 
$$\tilde{\mathsf t}\bigg(\!\!\begin{tikzpicture}[scale = 0.5, baseline = -3pt]
\draw[usual, blue] (0,0) circle (0.5);
\node[blue] at (-1.3, 0) {\scriptsize $\mathbb{T}_{\zeta}(1)$};
\end{tikzpicture} \bigg) = \mathsf t \left(\mathrm E_1\right) = 1\in\mathbbm{k}\,,$$
even though
$$\begin{tikzpicture}[scale = 0.5, baseline = -3pt]
\draw[usual] (0,0) circle (0.5);
\node at (-1.3, 0) {\scriptsize $\mathbb{T}_{\zeta}(1)$};
\end{tikzpicture} = -[2]_{\zeta} = 0\in \mathrm{End}(\mathbbm{1})\cong\mathbbm{k}\,.$$

\subsection{The Cutting Morphism, or the 3-Handle}

Following the procedure recalled in Section \ref{sub:skeintheory}, we now need to compute the cutting morphism (see Definition \ref{def:cutting}). This is needed for the evaluation of 3-handles.
In order to do so, write $G:=\oplus_{v=2}^{2p-1} \Tz(v-1)$ for the direct sum of the indecomposable projective objects of $\Ver_{p^{(2)}}^{\zeta^{1/2}}$. For every indecomposable projective object $V$, the modified trace $\mathsf{t}$ induces a non-degenerate pairing between $\mathrm{Hom}(\mathbbm{1},V\otimes G^*)$ and $\mathrm{Hom}(V\otimes G^*,\mathbbm{1})$.
We identify the cutting morphism $$\Lambda_{V\otimes G^*}=\sum_i x_i\circ x^i:V\otimes G^*\rightarrow V\otimes G^*,$$ where $\{x_i\}$ and $\{x^i\}$ are dual bases of $\mathrm{Hom}(\mathbbm{1},V\otimes G^*)$ and $\mathrm{Hom}(V\otimes G^*,\mathbbm{1})$.

By convention, the dual of $\mathrm{E}_{v-1}$ is given by its left-right mirror $\mathrm{E}_{v-1}^* = (\mathrm{E}_{v-1})^{\leftrightarrow}$. On the one hand, this operation does not affect $\mathrm{E}_{v-1}$ for $v=[a_1,0]_{p,\ell}$. On the other hand, it switches the free strand from the right to the left if $v=[a_1,1]_{p,\ell}$.

\begin{Lemma}\label{lem:Lambda}
Let $v=[a_1,a_0]_{p,\ell}$. 
\begin{enumerate}
\item If $a_0 = 0$, we have
$$\Lambda_{\mathbb{T}_{\zeta}(v-1)\otimes G^*} = \frac{(-1)^{a_1-1}}{a_1}\cdot \begin{tikzpicture}[anchorbase,scale=0.7,tinynodes]
\draw[usual] (0.5,0.5) arc (180:0:0.75 and 0.5); 
\draw[pJW] (0,0) rectangle ++(1,0.5);
\node at (0.5,0.2) {$\pjwm[v{-}1]$};
\draw[pJWl] (1.5,0) rectangle ++(1,0.5);
\node at (2,0.2) {$\pjwm[v{-}1]$};
\draw[usual] (0.5,2) arc (-180:0:0.75 and 0.5); 
\draw[pJW] (0,2) rectangle ++(1,0.5);
\node at (0.5,2.2) {$\pjwm[v{-}1]$};
\draw[pJWl] (1.5,2) rectangle ++(1,0.5);
\node at (2,2.2) {$\pjwm[v{-}1]$};
\end{tikzpicture}\ \ .$$
\item If $a_0 = 1$, we have
\begin{multline*}
\Lambda_{\mathbb{T}_{\zeta}(v -1)\otimes G^*}
=
\frac{(-1)^{a_1-1}}{a_1}\cdot\left(\ 
\begin{tikzpicture}[anchorbase,xscale=0.6, yscale = 0.7,tinynodes]
\draw[usual] (0.5,0.5) arc (180:0:0.75 and 0.5); 
\draw[pJWl] (1.5,0) rectangle ++(1,0.5);
\node at (1.95,0.2) {$\pjwm[v{-}3]$};
\draw[usual] (0.5,2) arc (-180:0:0.75 and 0.5); 
\draw[pJWl] (1.5,2) rectangle ++(1,0.5);
\node at (1.95,2.2) {$\pjwm[v{-}3]$};
\ptru{1.3}{0.5}{$\mathrm{U}^0$}{-1}{2}
\ptrd{1.3}{0.5}{$\mathrm{D}^0$}{-1}{0.5}
\end{tikzpicture}
+
\begin{tikzpicture}[anchorbase,xscale=0.6, yscale = 0.7,tinynodes]
\draw[usual] (0.5,0.5) arc (180:0:0.75 and 0.5); 
\draw[pJWl] (1.5,0) rectangle ++(1,0.5);
\node at (1.95,0.2) {$\pjwm[v{-}1]$};
\draw[usual] (0.5,2) arc (-180:0:0.75 and 0.5); 
\draw[pJW] (0,2) rectangle ++(1,0.5);
\node at (0.55,2.2) {$\pjwm[v{-}1]$};
\draw[pJWl] (1.5,2) rectangle ++(1,0.5);
\node at (1.95,2.2) {$\pjwm[v{-}1]$};
\ptr{1.3}{0.25}{$\mathrm{L}^0$}{-1}{0.25}
\end{tikzpicture}
+
\begin{tikzpicture}[anchorbase,xscale=0.6, yscale = 0.7,tinynodes]
\draw[usual] (0.5,0.5) arc (180:0:0.75 and 0.5); 
\draw[pJW] (0,0) rectangle ++(1,0.5);
\node at (0.55,0.2) {$\pjwm[v{-}1]$};
\draw[pJWl] (1.5,0) rectangle ++(1,0.5);
\node at (1.95,0.2) {$\pjwm[v{-}1]$};
\draw[usual] (0.5,2) arc (-180:0:0.75 and 0.5); 
\draw[pJWl] (1.5,2) rectangle ++(1,0.5);
\node at (1.95,2.2) {$\pjwm[v{-}1]$};
\ptr{1.3}{0.25}{$\mathrm{L}^0$}{-1}{2.25}
\end{tikzpicture}
-
\frac{a_1}{a_1+1}\cdot
\begin{tikzpicture}[anchorbase,xscale=0.6, yscale = 0.7,tinynodes]
\draw[usual] (0.5,0.5) arc (180:0:0.75 and 0.5); 
\draw[pJWl] (1.5,0) rectangle ++(1,0.5);
\node at (1.95,0.2) {$\pjwm[v{+}1]$};
\draw[usual] (0.5,2) arc (-180:0:0.75 and 0.5); 
\draw[pJWl] (1.5,2) rectangle ++(1,0.5);
\node at (1.95,2.2) {$\pjwm[v{+}1]$};
\ptrd{1.3}{0.5}{$\mathrm{D}^0$}{-1}{2.5}
\ptru{1.3}{0.5}{$\mathrm{U}^0$}{-1}{0}
\end{tikzpicture} \ \right)\,,
\end{multline*}
    where the morphism labeled $v-3$ is omitted if $v=[1,1]_{p,\ell} = 3$, and the morphism labeled $v+1$ is omitted if $v=[p-2,1]_{p,\ell} = 2p-3$, so that there are no division by zero.
\end{enumerate}
\end{Lemma}
\begin{proof}
Provided that $a_0=0$, for an indecomposable projective object $W$, we have that $\mathrm{Hom}(\mathbb{T}_{\zeta}(v-1)\otimes W^*, \mathbbm{1})\neq 0$ if and only if $W=\mathbb{T}_{\zeta}(v-1)$. Provided that $a_0=1$, for an indecomposable projective object $W$, we have that $\mathrm{Hom}(\mathbb{T}_{\zeta}(v-1)\otimes W^*, \mathbbm{1})\neq 0$ if and only if $W=\mathbb{T}_{\zeta}(v-3)$, $\mathbb{T}_{\zeta}(v-1)$, or $\mathbb{T}_{\zeta}(v+1)$. It will therefore be enough to compute the corresponding values of $\Lambda_{\mathbb{T}_{\zeta}(v-1)\otimes W}$.

We begin by considering the case $a_0 = 0$. The space of morphisms $\mathrm{Hom}(\mathbb{T}_{\zeta}(v-1)\otimes \mathbb{T}_{\zeta}(v-1)^*, \mathbbm{1})$ is spanned by the evaluation morphism $\mathrm{ev}_{v-1}$. The space of morphisms $\mathrm{Hom}( \mathbbm{1},\mathbb{T}_{\zeta}(v-1)\otimes \mathbb{T}_{\zeta}(v-1)^*)$ is spanned by the coevaluation morphism $\mathrm{coev}_{v-1}$. Using the compatibility of the modified trace with partial traces reviewed in Equation \eqref{eq:modifiedpartial} and Equation \eqref{eq:modifiedtracesimpleprojective}, we find 
$$\mathsf{t}(\mathrm{coev}_{v-1}\circ \mathrm{ev}_{v-1}) = \mathsf{t}(\mathrm{E}_{v-1}) = (-1)^{a_1-1}a_1\,.$$
This yields the claimed description of $\Lambda_{\mathbb{T}_{\zeta}(v-1)\otimes G^*}$ when $a_0=0$.

We now move on to the case $a_0 = 1$. The indecomposable projective object $W$ may therefore be either one of $\mathbb{T}_{\zeta}(v-3)$, $\mathbb{T}_{\zeta}(v-1)$, or $\mathbb{T}_{\zeta}(v + 1)$. We begin by considering $W=\mathbb{T}_{\zeta}(v-1)$. The space of morphisms $\mathrm{Hom}(\mathbb{T}_{\zeta}(v-1)\otimes \mathbb{T}_{\zeta}(v-1)^*, \mathbbm{1})$ is spanned by $$\mathrm{ev}_{v-1}\,,\quad \mathrm{ev}_{v-1}\circ (\mathrm{L}^0_{v-1}\otimes \mathrm{E}_{v-1})\,.$$ Similarly, the space of morphisms $\mathrm{Hom}( \mathbbm{1},\mathbb{T}_{\zeta}(v-1)\otimes \mathbb{T}_{\zeta}(v-1)^*)$ is spanned by 
$$\mathrm{coev}_{v-1}\,,\quad (\mathrm{L}^0_{v-1}\otimes \mathrm{E}_{v-1})\circ \mathrm{coev}_{v-1}\,.$$
Using the compatibility of the modified trace with partial traces, we compute
$$\mathsf{t}(\mathrm{coev}_{v-1}\circ \mathrm{ev}_{v-1}) = 0\,,$$
$$\mathsf{t}(\mathrm{coev}_{v-1}\circ \mathrm{ev}_{v-1}\circ (\mathrm{L}^0_{v-1}\otimes \mathrm{E}_{v-1})) = \mathsf{t}(\mathrm{L}^0_{v-1}) = \mathsf{t}(\mathrm{E}_{v-2}) = (-1)^{a_1-1}a_1\,,$$
$$\mathsf{t}((\mathrm{L}^0_{v-1}\otimes \mathrm{E}_{v-1})\circ \mathrm{coev}_{v-1}\circ \mathrm{ev}_{v-1}) = (-1)^{a_1-1}a_1\,,$$
$$\mathsf{t}((\mathrm{L}^0_{v-1}\otimes \mathrm{E}_{v-1})\circ\mathrm{coev}_{v-1}\circ \mathrm{ev}_{v-1}\circ (\mathrm{L}^0_{v-1}\otimes \mathrm{E}_{v-1})) = 0\,.$$
The second and third equalities make use of Equation \eqref{eq:modifiedtraceNONsimpleprojective}.
This yields \begin{align*}\Lambda_{\mathbb{T}_{\zeta}(v -1)\otimes \mathbb{T}_{\zeta}(v -1)^*} =&\ \frac{(-1)^{a_1-1}}{a_1}\cdot (\mathrm{L}^0_{v-1}\otimes\mathrm{E}_{v-1})\circ \mathrm{coev}_{v-1}\circ \mathrm{ev}_{v-1}\\ & + \frac{(-1)^{a_1-1}}{a_1}\cdot \mathrm{coev}_{v-1}\circ \mathrm{ev}_{v-1} \circ(\mathrm{L}^0_{v-1}\otimes\mathrm{E}_{v-1}).\end{align*}
We now consider the case $W=\mathbb{T}_{\zeta}(v-3)$. The spaces of morphisms $\mathrm{Hom}(\mathbb{T}_{\zeta}(v-1)\otimes \mathbb{T}_{\zeta}(v-3)^*, \mathbbm{1})$ and $\mathrm{Hom}( \mathbbm{1},\mathbb{T}_{\zeta}(v-1)\otimes \mathbb{T}_{\zeta}(v-3)^*)$ are spanned by $$\mathrm{ev}_{v-3}\circ (\mathrm{D}^0_{v-1}\otimes \mathrm{E}_{v-3})\,,\quad  (\mathrm{U}^0_{v-3}\otimes \mathrm{E}_{v-3})\circ\mathrm{coev}_{v-3}\,.$$ We compute $$\mathsf{t}((\mathrm{U}^0_{v-3}\otimes \mathrm{E}_{v-3})\circ\mathrm{coev}_{v-3}\circ \mathrm{ev}_{v-3}\circ (\mathrm{D}^0_{v-1}\otimes \mathrm{E}_{v-3})) = \mathsf{t}(\mathrm{L}^0_{v-1}) = (-1)^{a_1-1}a_1.$$
Finally, the case $W=\mathbb{T}_{\zeta}(v+1)$ follows using an analogous argument. In particular, we have
$$\mathsf{t}((\mathrm{D}^0\otimes \mathrm{E}_{v+1})\circ\mathrm{coev}_{v+1}\circ \mathrm{ev}_{v+1}\circ (\mathrm{U}^0\otimes \mathrm{E}_{v+1})) = \mathsf{t}(\mathrm{D}^0\circ\mathrm{U}^0) = \mathsf{t}\left(\begin{tikzpicture}[anchorbase,scale=0.4,tinynodes]
\draw[pJW] (-0.5,-5.5) rectangle (1.6,-6.5);
\node at (0.6,-6.1) {$\pjwm[v]$};
\draw[usual] (2,-6.5) to (2,-5.5);
\draw[usual] (1,-5.5) to[out=90,in=180] (1.5,-5.15) 
to[out=0,in=90] (2,-5.5); 
\draw[usual] (1,-6.5) to[out=270,in=180] (1.5,-6.85) to[out=0,in=270] (2,-6.5); 
\end{tikzpicture}\right) = \mathsf{t}(\mathrm{E}_{v}) = (-1)^{a_1}(a_1+1)\,,$$
where the first and third equalities use the compatibility of the modified trace with partial traces. The claimed descriptions of both $\Lambda_{\mathbb{T}_{\zeta}(v -1)\otimes \mathbb{T}_{\zeta}(v-3)^*}$ and $\Lambda_{\mathbb{T}_{\zeta}(v -1)\otimes \mathbb{T}_{\zeta}(v+1)^*}$ follow form the computations above, so that the proof is complete.
\end{proof}

\begin{Example}
The skein $\Gamma_0$ introduced in Equation \eqref{eq:Gamma0} may be simplified:
\begin{equation}\label{eq:Gamma0explicit}
\Gamma_0 =
\begin{tikzpicture}[anchorbase,scale=0.2,tinynodes]
\draw[lJW] (-2,3) rectangle (2,5);
\node at (0,4) {$\Lambda_{P_{\mathbbm{1}}}$};
\draw[usual, blue] (0,5) to[out=90,in=90] (4,5);
\draw[usual, blue] (0,3) to[out=270,in=270] (4,3) to (4,5) node[right, blue]{$P_{\mathbbm{1}}$};
\end{tikzpicture} =\, \begin{tikzpicture}[anchorbase,scale=0.3,tinynodes]
\draw[usual, blue] (0,0) circle(1);
\node[blue] at (2.5,0) {\tiny $\mathbb{T}_{\zeta}(1)$};
\end{tikzpicture}\,.
\end{equation}
\end{Example}

\subsection{The Chromatic Morphism, or the 2-Handle}\label{sub:chromatic}

The next step in Section \ref{sub:skeintheory} requires us to exhibit a chromatic morphism. This will allow us to evaluate 2-handles.
More precisely, we construct a chromatic morphism based at $\mathbb{T}_{\zeta}(1)$ as in Definition \ref{def:chromatic} above, that is a morphism
$$\mathsf{c}:G\otimes \Tz(1)\rightarrow G\otimes \Tz(1)\,,\quad G:=\oplus_{v=2}^{2p-1} \mathbb{T}_{\zeta}(v-1)\,,$$ satisfying Equation \eqref{eq:chromatic}. In order to do so, we will use the morphisms $\mathrm{Z}^{0}_{v}$ defined in Equation \eqref{eq:canonicalbasisT1}.

\begin{Proposition}\label{prop:chromatic}
For $v=[a_1,a_0]_{p,\ell}$, let
$$\mathsf{c}_{v-1} :\mathbb{T}_{\zeta}(v-1)\otimes \Tz(1)\rightarrow \mathbb{T}_{\zeta}(v-1)\otimes\Tz(1)$$ be defined as
$$\mathsf{c}_{v-1} = \begin{cases}
(-1)^{a_1-1}a_1\cdot  \mathrm{E}_{v-1} \otimes \mathrm{E}_{1}\,, & \text{if } a_0=0\,,\\
(-1)^{a_1-1}a_1\cdot \mathrm{Z}^{0}_{v}\,, & \text{if } a_0=1\,.
\end{cases}$$
Then $\mathsf{c} := \oplus_{v=2}^{2p-1} \mathsf{c}_{v-1} : G\otimes \Tz(1)\rightarrow G\otimes \Tz(1)$ is a chromatic morphism based at $\Tz(1)$. 
\end{Proposition}
\begin{proof}
We need to check that Equation \eqref{eq:chromatic} holds for $V=\Tz(v-1)$ with $v\in \{2,\dots,2p-1\}.$

If $a_0 = 0$, thanks to Lemma \ref{lem:Lambda} the cutting morphism only has one term, which involves the summand $\Tz(v-1)^*$ of $G^*$. We therefore find that
$$\begin{tikzpicture}[anchorbase,scale=0.3,tinynodes]
\draw[lJW] (-2.5,3) rectangle (2.5,5);
\node at (0,4) {$\Lambda_{V\otimes G^*}$};
\draw[cJW] (3.5,3) rectangle (8.5,5);
\node at (6,3.9) {\normalsize $\mathsf{c}$};
\draw[usual] (2,5) to[out=90,in=180] (3,6)node[above,xshift=0.2cm,yshift=-0.15cm]{$G$} to[out=0,in=90] (4,5);
\draw[usual] (-1.5,5) to (-1.5,7)node[left,yshift=-0.15cm]{$V$};
\draw[usual] (7.5,5) to (7.5,7)node[right,yshift=-0.15cm]{$\Tz(1)$};
\draw[usual] (2,3) to[out=270,in=180] (3,2)node[below,xshift=0.15cm,yshift=0cm]{$G$} to[out=0,in=270] (4,3);
\draw[usual] (-1.5,3) to (-1.5,1)node[left,xshift=-0.00cm,yshift=-0.00cm]{$V$};
\draw[usual] (7.5,3) to (7.5,1)node[right,xshift=0.0cm,yshift=0.00cm]{$\Tz(1)$};
\end{tikzpicture} = \frac{(-1)^{a_1-1}}{a_1}\cdot  \begin{tikzpicture}[anchorbase,scale=0.3,tinynodes]
\draw[cJW] (3.5,3) rectangle (8.5,5);
\node at (6,3.9) {\normalsize $\mathsf{c}_{\pjwm[v{-}1]}$};
\end{tikzpicture} = \mathrm{E}_{v-1}\otimes\mathrm{E}_1\,,$$
so that Equation \eqref{eq:chromatic} holds as desired.

If $a_0=1$, by Lemma \ref{lem:Lambda} the cutting morphism has four terms, involving the summands $\Tz(v-1)$, $\Tz(v-3)$ and $\Tz(v+1)$ of $G$. One finds
\begin{equation*}
\begin{tikzpicture}[anchorbase,scale=0.3,tinynodes]
\draw[lJW] (-2.5,3) rectangle (2.5,5);
\node at (0,4) {$\Lambda_{V\otimes G^*}$};
\draw[cJW] (3.5,3) rectangle (8.5,5);
\node at (6,3.9) {\normalsize $\mathsf{c}$};
\draw[usual] (2,5) to[out=90,in=180] (3,6)node[above,xshift=0.2cm,yshift=-0.15cm]{$G$} to[out=0,in=90] (4,5);
\draw[usual] (-1.5,5) to (-1.5,7)node[left,yshift=-0.15cm]{$V$};
\draw[usual] (7.5,5) to (7.5,7)node[right,yshift=-0.15cm]{$\Tz(1)$};
\draw[usual] (2,3) to[out=270,in=180] (3,2)node[below,xshift=0.15cm,yshift=0cm]{$G$} to[out=0,in=270] (4,3);
\draw[usual] (-1.5,3) to (-1.5,1)node[left,xshift=-0.00cm,yshift=-0.00cm]{$V$};
\draw[usual] (7.5,3) to (7.5,1)node[right,xshift=0.0cm,yshift=0.00cm]{$\Tz(1)$};
\end{tikzpicture} 
\!\!\!\!\!\!\!=\frac{(-1)^{a_1}}{a_1+1}\cdot\, \begin{tikzpicture}[anchorbase,scale=0.8,tinynodes]
\draw[usual] (0.5,0) -- ++(0,2.5);
\draw[usual] (1.3,0) -- ++(0,2.5);
\ptru{1}{0.5}{$\mathrm{U}^0$}{-1}{0.2}
\ptrd{1}{0.5}{$\mathrm{D}^0$}{-1}{2.3}
\draw[cJW] (0,1) rectangle (1.6,1.5);
\node at (0.8,1.2) {\normalsize $\mathsf{c}_{\pjwm[v{+}1]}$};
\end{tikzpicture}
\,+\,\frac{(-1)^{a_1-1}}{a_1}\cdot\,
\begin{tikzpicture}[anchorbase,scale=0.8,tinynodes]
\draw[usual] (0.5,0) -- ++(0,2.5);
\draw[usual] (1.3,0) -- ++(0,2.5);
\ptr{1}{0.25}{$\mathrm{L}^0$}{-1}{2.1}
\draw[cJW] (0,1) rectangle (1.6,1.5);
\node at (0.8,1.2) {\normalsize $\mathsf{c}_{\pjwm[v{-}1]}$};
\end{tikzpicture}
\,+\,
\frac{(-1)^{a_1-1}}{a_1}\cdot\,
\begin{tikzpicture}[anchorbase,scale=0.8,tinynodes]
\draw[usual] (0.5,0) -- ++(0,2.5);
\draw[usual] (1.3,0) -- ++(0,2.5);
\ptr{1}{0.25}{$\mathrm{L}^0$}{-1}{0.4}
\draw[cJW] (0,1) rectangle (1.6,1.5);
\node at (0.8,1.2) {\normalsize $\mathsf{c}_{\pjwm[v{-}1]}$};
\end{tikzpicture}
\,+\,
\frac{(-1)^{a_1-1}}{a_1}\cdot\,
\begin{tikzpicture}[anchorbase,scale=0.8,tinynodes]
\draw[usual] (0.5,0) -- ++(0,2.5);
\draw[usual] (1.3,0) -- ++(0,2.5);
\ptrd{1}{0.5}{$\mathrm{D}^0$}{-1}{0.7}
\ptru{1}{0.5}{$\mathrm{U}^0$}{-1}{1.8}
\draw[cJW] (0,1) rectangle (1.6,1.5);
\node at (0.8,1.2) {\normalsize $\mathsf{c}_{\pjwm[v{-}3]}$};
\end{tikzpicture}\,.
\end{equation*}
Using the graphical definition of $\mathrm{Z}^0_{v}$ given in Equation \eqref{eq:defYZ}, we can verify the following equalities:
\begin{gather*}
(\mathrm{D}^0\otimes \mathrm{E}_1)\circ\mathrm{Z}^0_{v+2}\circ (\mathrm{U}^0\otimes \mathrm{E}_1) = \mathrm{E}_{v}\,,
\\
(\mathrm{L}^0_{v-1}\otimes \mathrm{E}_1) \circ \mathrm{Z}^0_{v} = \mathrm{Y}^0_{v}\circ \mathrm{Z}^0_{v} = \mathrm{A}^{0}_{v}\,,
\\
\mathrm{Z}^0_{v}\circ (\mathrm{L}^0_{v-1}\otimes \mathrm{E}_1) = \mathrm{Z}^0_{v}\circ \mathrm{Y}^0_{v} = \tilde{\mathrm{A}}^{0}_{v}\,, 
\\ 
(\mathrm{U}^0\otimes \mathrm{E}_1)\circ \mathrm{Z}^0_{v-2}\circ (\mathrm{D}^0\otimes \mathrm{E}_1) = -\frac{a_{1}}{a_{1}-1}\cdot \mathrm{B}^{1}_{v} \,.
\end{gather*}
Using Proposition \ref{prop:categorifiedfusionT1}, we therefore find that
$$\begin{tikzpicture}[anchorbase,scale=0.3,tinynodes]
\draw[lJW] (-2.5,3) rectangle (2.5,5);
\node at (0,4) {$\Lambda_{V\otimes G^*}$};
\draw[cJW] (3.5,3) rectangle (8.5,5);
\node at (6,3.9) {\normalsize $\mathsf{c}$};
\draw[usual] (2,5) to[out=90,in=180] (3,6)node[above,xshift=0.2cm,yshift=-0.15cm]{$G$} to[out=0,in=90] (4,5);
\draw[usual] (-1.5,5) to (-1.5,7)node[left,yshift=-0.15cm]{$V$};
\draw[usual] (7.5,5) to (7.5,7)node[right,yshift=-0.15cm]{$\Tz(1)$};
\draw[usual] (2,3) to[out=270,in=180] (3,2)node[below,xshift=0.15cm,yshift=0cm]{$G$} to[out=0,in=270] (4,3);
\draw[usual] (-1.5,3) to (-1.5,1)node[left,xshift=-0.00cm,yshift=-0.00cm]{$V$};
\draw[usual] (7.5,3) to (7.5,1)node[right,xshift=0.0cm,yshift=0.00cm]{$\Tz(1)$};
\end{tikzpicture} =  \mathrm{E}_{v}+\mathrm{A}^{0}_{v}+ \tilde{\mathrm{A}}^{0}_{v} +\mathrm{B}^{1}_{v} = \mathrm{E}_{v-1} \otimes \mathrm{E}_{1}\,,$$
so that Equation \eqref{eq:chromatic} also holds in this case, thereby concluding the proof.
\end{proof}

\subsection{No 1-Handle}

Thanks to the above computations, we can compute the partially defined 4-dimensional TQFT $\SS_{p^{(2)}}^{\zeta^{1/2}}:\HD_4^{2,3,4} \to \Vect$ afforded by Theorem \ref{thm:partialCGHPM}. In \cite{CGHPM}, it was established that this TQFT can be extended to $\HD_4^{1,2,3,4}$ as soon as the input finite ribbon tensor category is chromatic non-degenerate in the sense of Definition \ref{def:chrnondegen}. We will now examine when the finite ribbon tensor categories $\Ver_{p^{(2)}}^{\zeta^{1/2}}$ are chromatic non-degenerate. In order to do so, we have to compute the morphism $\Delta_0^{P_{\mathbbm{1}}}$ recalled in Definition \ref{def:chrnondegen}. More generally, we will identify the endomorphisms $\Delta_0^{v-1}$ defined by $$\Delta_0^{v-1}= \Delta_0^{\mathbb{T}_{\zeta}(v-1)}=
\begin{tikzpicture}[anchorbase,scale=0.25,tinynodes]
\draw[usual,crossline] (0,6) to (0,7);
\draw[usual,crossline, red] (-1.5,5) to[out=90,in=90] (1.5,5) node[right=-2pt]{red} to[out=270,in=270] (-1.5,5);
\draw[usual, crossline] (0,3) to (0,5);
\draw[pJW] (-2,1) rectangle (2,3);
\node at (0,1.9) {$\pjwm[{v{-}1}]$};
\end{tikzpicture}\in\End(\Tz(v-1))\,,$$ which can be computed by turning the red circle blue using Equation \eqref{eq:redtoblue}.

\begin{Theorem}
    The finite ribbon tensor category $\Ver^{\zeta^{1/2}}_{p^{(2)}}$ is chromatic non-degenerate if and only if $p=3$.
\end{Theorem}
\begin{proof}
    This is the case $v=3$ of Proposition \ref{prop:Delta0} below.
\end{proof}

\begin{Proposition}\label{prop:Delta0}
Let $v=[a_1,a_0]_{p,\ell}$.
\begin{enumerate}
\item If $a_0=0$, we have $\Delta_0^{v-1}=0$.
\item If $a_0 = 1$, we have $$\Delta_0^{v-1}= 2(1+(-1)^{a_1-1})\frac{(p-1)p(2p+1)}{3}\cdot \mathrm{L}_{v-1}.$$
\end{enumerate}
In particular, the morphisms $\Delta_0^{v-1}$ vanish for all $v$ when $p>3$.
\end{Proposition}
\begin{proof}
By the sliding property for red circles given in \cite[Lemma 2.4]{CGHPM}, the morphism $\Delta_0^{v-1}$ satisfies
\begin{equation*}
\begin{tikzpicture}[anchorbase,scale=0.25,tinynodes]
\draw[usual,crossline] (0,5) to (0,8);
\draw[usual,crossline, red] (-1.5,5) to[out=90,in=90] (1.5,5) node[right=-2pt]{red} to[out=270,in=270] (-1.5,5);
\draw[usual,crossline] (4,0) to (4,1.2) to[out=90,in=270] (-3,5) to[out=90,in=270] (4,8);
\draw[usual, crossline] (0,2) to (0,5);
\draw[pJW] (-2,0) rectangle (2,2);
\node at (0,1) {$\pjwm[{v{-}1}]$};
\end{tikzpicture}
= 
\begin{tikzpicture}[anchorbase,scale=0.25,tinynodes]
\draw[usual,crossline] (0,5) to (0,8);
\draw[usual,crossline, red] (-1.5,5) to[out=90,in=90] (1.5,5) node[right=-2pt]{red} to[out=270,in=270] (-1.5,5);
\draw[usual,crossline] (4,0) to (4,8);
\draw[usual, crossline] (0,2) to (0,5);
\draw[pJW] (-2,0) rectangle (2,2);
\node at (0,1) {$\pjwm[{v{-}1}]$};
\end{tikzpicture}
\end{equation*}
\noindent or, algebraically, $$(\Delta_0^{v-1} \otimes \mathrm{E}_{1}) \circ \beta^2_{v-1,1} =  \Delta_0^{v-1} \otimes \mathrm{E}_{1}\ .$$
If $a_0=0$, we can write $\Delta_0^{v-1} = \lambda\cdot \mathrm{E}_{v-1}$ for some scalar $\lambda$. Then, using the first case of Proposition \ref{prop:braidingT1}, we get $$\lambda(\zeta^{v-1}\cdot \mathrm{E}_v + 2\zeta^{v-2}\cdot \mathrm{L}_v^0) = \lambda\cdot \mathrm{E}_v\,,$$ which implies $\lambda=0$.

If $a_0=1$, we can write $\Delta_0^{v-1} = \lambda\cdot \mathrm{E}_{v-1} + \mu\cdot \mathrm{L}_{v-1}^0$ for some scalars $\lambda$ and $\mu$. Then, using the second case of Proposition \ref{prop:braidingT1}, we similarly get
$$\lambda=0\,, \quad \quad \text{and} \quad\quad \zeta^{v-3}\mu=\mu\ .$$
If $a_1$ even, this implies that $\mu=0$, so that $\Delta_0^{v-1}=0$.

It therefore only remains to compute $\Delta_0^{v-1} = \mu\cdot \mathrm{L}_{v-1}^0$ for $a_0=1$ and $a_1$ odd. In order to do so, it will suffice to compute the modified trace $\mathsf{t}(\Delta_0^{v-1}) = \mu (-1)^{a_1-1}a_1$. Now, recall that, by Equation \eqref{eq:redtoblue}, we have:
\begin{equation*}
\Delta_0^{v-1}=
\begin{tikzpicture}[anchorbase,scale=0.25,tinynodes]
\draw[usual,crossline] (0,5) to (0,8);
\draw[usual,crossline] (2.5,5) to (2.5,8);
\draw[usual,crossline, red] (-1.5,5) to[out=90,in=90] (3.5,5) node[right=-2pt]{red} to[out=270,in=270] (-1.5,5);
\draw[usual, crossline] (0,2) to (0,5);
\draw[usual, crossline] (2.5,0) to (2.5,5);
\draw[pJW] (-2,0) rectangle (2,2);
\node at (0,0.9) {$\pjwm[{v{-}2}]$};
\end{tikzpicture}
= \sum_{w=2}^{2p-1}\,
\begin{tikzpicture}[anchorbase,scale=0.3,tinynodes]
\draw[usual] (-5,0.5) to[out=90,in=180] (-1.5,2) 
to[out=0,in=90] (2,0.5);
\draw[usual,crossline] (-2.5,-4) to (-2.5,3.5);
\draw[usual,crossline] (0,0) to (0,3.5);
\draw[usual] (0,-5.5) to[out=90,in=200] (1.5,-4) to[out=20,in=270] (4,-1);
\draw[usual] (0,0) to[out=270,in=180] (1.5,-3) to[out=0,in=270] (3,-1);
\draw[usual,crossline] (-5,-1) to[out=270,in=180] (-1.5,-2.5) 
to[out=0,in=270] (2,-1);
\draw[pJW] (-6.5,-1) rectangle (-3.5,0.5);
\node at (-4.9,-0.3) {$\pjwm[{w{-}1}]$};
\draw[cJW] (1,-1) rectangle (3.5,0.5);
\node at (2.25,-0.25) {$\mathsf{c}_{\pjwm[{w{-}1}]}$};
\draw[usual] (3,0.5) to[out=90,in=180] (3.5,1) 
to[out=0,in=90] (4,0.5)to[out=-90,in=90] (4,-1); 
\draw[pJW] (-4,-5.5) rectangle (-1,-4);
\node at (-2.4,-4.85) {$\pjwm[{v{-}2}]$};
\end{tikzpicture}\,.
\end{equation*}
Using the compatibility of the modified trace with left partial traces together with its cyclicity property, we get
\begin{equation*}\mathsf{t}(\Delta_0^{v-1}) 
=
\sum_{w=2}^{2p-1}\,
\mathsf{t}\left(\,
\begin{tikzpicture}[anchorbase,scale=0.3,tinynodes]
\draw[usual] (-5,0.5) to[out=90,in=180] (-1.5,2.5) 
to[out=0,in=90] (2,0.5);
\draw[usual,crossline] (-9.5,0.5) to[out=90,in=180] (-6,2.5) 
to[out=0,in=90] (-2.5,0.5);
\draw[usual] (-2.5,-1) to (-2.5,0.5);
\draw[usual] (0.5,-2) to[out=310,in=180] (1.5,-2.5) to[out=0,in=270] (3,-1);
\draw[usual] (-9.5,-1) to[out=270,in=180] (-6,-2.5) 
to[out=0,in=270] (-2.5,-1);
\draw[usual,crossline] (-5,-1) to[out=270,in=180] (-1.5,-2.5) 
to[out=0,in=270] (2,-1);
\draw[usual,crossline] (0.5,-2) to[out=130,in=90] (-1.5,-3) to (-1.5,-3.5);
\draw[pJW] (-6.5,-1) rectangle (-3.5,0.5);
\node at (-4.9,-0.3) {$\pjwm[{w{-}1}]$};
\draw[cJW] (1,-1) rectangle (4,0.5);
\node at (2.5,-0.25) {$\mathsf{c}_{\pjwm[{w{-}1}]}$};
\draw[usual] (3,0.5) to[out=90,in=-90] (3,3); 
\draw[pJW] (-11,-1) rectangle (-8,0.5);
\node at (-9.4,-0.3) {$\pjwm[{v{-}2}]$};
\end{tikzpicture}
\,\right)\,.
\end{equation*}
Let us write $w=[b_1,b_0]_{p,\ell}$. 
On the one hand, if $b_0 = 0$, then the chromatic morphism $\mathsf{c}_{\pjwm[{w{-}1}]}$ is $(-1)^{b_1-1}b_1$ times the identity. The corresponding term in the last sum above can then be simplified using the first part of Lemma \ref{lem:encirclingbyasimple} twice, so as to yield
$$(v-1)(-1)^{b_1-1}\cdot w(-1)^{1-1}\cdot (-1)^{b_1-1}b_1 = 2(v-1)b_1^2 \ .$$
On the other hand, if $b_0 = 1$, then using the second parts of Lemma \ref{lem:encirclingbyasimple} and Proposition \ref{prop:braidingT1}, the corresponding term can be simplified to
\begin{align*}2(v-1)b_1(-1)^{-(a_1-1)}\cdot \mathsf{t}(\mathsf{pTr}^L_{w-1}(\mathrm{Z}^0_{w}\circ &\beta^{-2}_{w-1,1}\circ \mathrm{Y}^0_{w})) = \\ &= 2(v-1)b_1(-1)^{-(a_1-1)}\zeta^{-(w-3)}\cdot \mathsf{t}(\mathsf{pTr}^L_{w-1}(\tilde{\mathrm{A}}^0_{w}))\\
& = 2(v-1)b_1^2(-1)^{(b_1-1)-(a_1-1)}\zeta^{-(w-3)}\\
& = 2(v-1)b_1^2(-1)^{-(a_1-1)},\end{align*} as $\mathrm{L}^0_{w-1}\otimes \mathrm{E}_1 = \mathrm{Y}^0_{w}$ by Lemma \ref{lem:YisLtensorId1}, and $\mathsf{t}(\mathsf{pTr}^L_{w-1}(\tilde{\mathrm{A}}^0_{w})) = \mathsf{t}(\mathrm{E}_{w-2}) = (-1)^{b_1-1}b_1$. We therefore find that
$$\Delta_0^{v-1}= (-1)^{a_1-1} 2(v-1)\sum_{b_1=1}^{p-1}(b_1^2 + (-1)^{a_1-1}b_1^2) = 2(1+(-1)^{a_1-1})\frac{(p-1)p(2p+1)}{3}\,.$$
\end{proof}

\begin{Corollary}
    Provided that $p>3$, the functor $\SS_{p^{(2)}}^{\zeta^{1/2}}: \HD_4^{2,3,4} \to \Vect$ does not extend to $\HD_4^{1,2,3,4}$, i.e.\ there is no consistent way of assigning a value to the 1-handle.
\end{Corollary}
\begin{proof}
    Suppose $\tilde{\SS}:\HD_4^{1,2,3,4} \to \Vect$ extends $\SS_{p^{(2)}}^{\zeta^{1/2}}: \HD_4^{2,3,4} \to \Vect$. Consider the cobordism $W_1: S^3 \to S^2\times S^1$ given by attaching a 1-handle on two disks in $S^3$, whose belt sphere is $S^2 \times \{\ast\}$, and the cobordism $W_2: S^2\times S^1 \to S^3$ given by attaching a 2-handle along $\{\ast\}\times S^1$. These two handle attachments cancel, as the belt sphere and the attaching sphere intersect transversely once, whence $\tilde\SS(W_2)\circ\tilde\SS( W_1)$ is an isomorphism. This is absurd as $\tilde\SS(W_2)=\SS_{p^{(2)}}^{\zeta^{1/2}}(W_2) =0$ by Proposition \ref{prop:Delta0}.
\end{proof}

\section{Computations of 4-Dimensional 2-Handlebody Invariants}

The results of the previous section allow us to compute explicitly the invariant $\dot{\mathcal{S}}_{p^{(2)}}^{\zeta^{1/2}}$  of 4-dimensional 2-handlebodies up to 2-equivalence.
We have recalled in Section \ref{sub:invariant} above how this invariant is computed from a handle decomposition. The key steps are summarized in Figure \ref{fig:2hbInv}.
In particular, we have by convention that $\dot{\mathcal{S}}_{p^{(2)}}^{\zeta^{1/2}}(D^4)=1$ where $D^4$ refers to the 2-handlebody with a single 4-handle.
There may exist other presentations of the 4-manifold $D^4$ as a 4-dimensional 2-handlebody \cite{Gom:AkbulutKirby}, so that we are somewhat abusing notations by using $D^4$ to refer to a 2-handlebody.
Likewise, we will use $\mathbb{C}P^2$ and $S^2\times S^2$ below to refer to specific handle decompositions of these 4-manifolds. In fact, $\mathbb CP^2$ and $S^2\times S^2$ are not even 2-handlebodies as their handle decompositions involve both a $0$-handle and a $4$-handle. This can be easily be rectified by omitting the $0$-handle. Additionally, we also evaluate our invariant on the Mazur manifold, potentially decorated with skeins.

\subsection{Complex Projective Planes}

We presently evaluate the invariant $\dot{\mathcal{S}}_{p^{(2)}}^{\zeta^{1/2}}$ on the two 2-handlebodies whose Kirby diagrams with blackboard framing are given by
$$
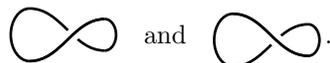
\begin{tikzpicture}[anchorbase,xscale=0.5, yscale = -0.7]
\draw[usual] (0,0) to[out=0, in=180] (2,0.8); 
\draw[usual] (2,0.8) arc(90:-90:0.3); 
\draw[usual, crossline] (2,0.2) to[out=180, in=0] (0,1); 
\draw[usual] (0,1) arc(90:270:0.5); 
\end{tikzpicture}
\quad\mathrm{and}\quad 
\begin{tikzpicture}[anchorbase,xscale=0.5, yscale = 0.7]
\draw[usual] (0,0) to[out=0, in=180] (2,0.8); 
\draw[usual] (2,0.8) arc(90:-90:0.3); 
\draw[usual, crossline] (2,0.2) to[out=180, in=0] (0,1); 
\draw[usual] (0,1) arc(90:270:0.5); 
\end{tikzpicture}\,.$$
Abusing notation, we will denote these two 2-handlebodies by $\mathbb{CP}^2$ and $\overline{\mathbb{CP}}^2$.
Following \cite{CGHPM}, we use
$\Delta_+:=\dot{\mathcal{S}}_{p^{(2)}}^{\zeta^{1/2}}(\mathbb{CP}^2)$ and $\Delta_-:=\dot{\mathcal{S}}_{p^{(2)}}^{\zeta^{1/2}}(\overline{\mathbb{CP}}^2)$ to denote the value of the invariant $\dot{\mathcal{S}}_{p^{(2)}}^{\zeta^{1/2}}$ on $\mathbb{CP}^2$ and $\overline{\mathbb{CP}}^2$.
Following the prescriptions of Section \ref{sub:invariant} and recalling our identification of the skein $\Gamma_0$ in Equation \eqref{eq:Gamma0explicit}, the scalars $\Delta_{\pm}$ are explicitly given by
$$\Delta_{\pm} = \mathsf{t}\left(\begin{tikzpicture}[anchorbase,xscale=0.5]
\draw[usual, red] (0,0) arc (180:520:0.5);
\draw[usual] (1.5,-0.5)--++(0,1)node[pos = 0.8, right]{\tiny $\Tz(1)$};
\node[rectangle, draw=black, fill=white] at (0,0) {\small $\theta^{\pm1}$};
\end{tikzpicture} \right) 
= \sum_{v=2}^{p\ell-1}\mathsf{t}\big((\theta_{v-1}^{\pm}\otimes \mathrm{E}_{1})\circ\mathsf{c}_{v-1}\big)\,.$$

\begin{Theorem}\label{thm:invariantCP2}
We have $$\Delta_{\pm} = \frac{(p-1)p(2p-1)}{3}\cdot \zeta^{\mp 1}\,.$$
In particular, $\mathrm{Ver}_{p^{(2)}}^{\zeta^{1/2}}$ is twist non-degenerate if and only if $p=3$.
\end{Theorem}
\begin{proof}
If $v=[a_1,0]_{p,\ell}$, then we have $$\mathsf{t}\big((\theta_{v-1}^{\pm}\otimes \mathrm{E}_{1})\circ\mathsf{c}_{v-1}\big)=(-1)^{a_1-1}a_1\cdot\mathsf{t}\big((\theta_{v-1}^{\pm}\otimes \mathrm{E}_{1})\big) = 0\,,$$ by Proposition \ref{prop:chromatic} and the compatibility of the modifed trace with partial traces as in Equation \eqref{eq:modifiedpartial}.

Let now $v=[a_1,1]_{p,\ell}$. By combining Proposition \ref{prop:twists}, and Lemma \ref{lem:YisLtensorId1} together with the identity $\mathrm{Y}^0_v\, \mathrm{Z}^0_v=\mathrm{A}^0_v$, we get
$$(\theta_{v-1}\otimes \mathrm{E}_{1})\circ\mathsf{c}_{v-1} = a_1(-1)^{a_1-1}\cdot \mathrm{Z}^0_{v} + 2a_1(-1)^{a_1-1}\zeta^{-1}\cdot \mathrm{A}^0_{v}\,.$$
Using the cyclicity property of the modified trace, we then find 
$$\mathsf{t}\big((\theta_{v-1}\otimes \mathrm{E}_{1})\circ\mathsf{c}_{v-1}\big)
=
a_1(-1)^{a_1-1}\cdot 0 + 2a_1(-1)^{a_1-1}\zeta^{-1}\cdot \mathsf{t}(\mathrm{E}_{v-1}) = 2a_1^2\zeta^{-1}\,.$$
Altogether, we have $$\Delta_+ = \sum_{a_1=1}^{p-1}2a_1^2\zeta^{-1} = 2\zeta^{-1}\sum_{a_1=1}^{p-1}a_1^2 = \frac{(p-1)p(2p-1)}{3}\cdot \zeta^{-1}\,.$$

It remains to identify $\Delta_-$. For $v=[a_1,1]_{p,\ell}$, it is easy to check that $$\theta^{-1}_{v-1} =  \mathrm{E}_{v-1} + 2 \zeta\cdot \mathrm{L}^0_{v-1}\,.$$ By adapting the above argument slightly, it follows that $\Delta_{-} = (p-1)p(2p-1)/3\cdot \zeta$.
\end{proof}

\begin{Remark}
We can compute the invariant associated to an unknot with any framing $k\in \mathbb{Z}$. On the one hand, if $v=[a_1,0]_{p,\ell}$, we have $$\mathsf{t}\big((\theta_{v-1}^{k}\otimes \mathrm{E}_{1})\circ\mathsf{c}_{v-1}\big) = 0\,.$$ 
On the other hand, if $v=[a_1,1]_{p,\ell}$, we find by adapting the argument used in the proof of Theorem \ref{thm:invariantCP2} above that
$$\mathsf{t}\big((\theta_{v-1}^{k}\otimes \mathrm{E}_{1})\circ\mathsf{c}_{v-1}\big) = 2ka_1^2\zeta^{-1}\,.$$
Altogether, we find that $$\dot{\mathcal{S}}_{p^{(2)}}^{\zeta^{1/2}} \left( \begin{tikzpicture}[anchorbase,scale=0.5]
\draw[usual] (0,0) arc (360:0:1) node[above right]{$k$}; 
\end{tikzpicture}\right) = k\frac{(p-1)p(2p-1)}{3}\cdot \zeta^{-1}\ .$$
\end{Remark}

\subsection{Product of two 2-Spheres}

It is important to compute the value of our invariant on $S^2\times S^2$. Namely, a classical result of Gompf \cite{Gom:stable} asserts that any two 4-manifolds that are homemorphic become diffeomorphic upon taking their connected sum with sufficiently many copies of $S^2\times S^2$. Said differently, $S^2\times S^2$-stabilization kills exotic phenomena. It follows that an invariant that is multiplicative under connected sum, such as $\dot{\mathcal{S}}_{p^{(2)}}^{\zeta^{1/2}}$, cannot detect exotic smooth structures if it is nonzero on $S^2\times S^2$.\footnote{The situation here is slightly more ambivalent as we do not actually define an invariant of $S^2\times S^2$, but rather of a presentation of $S^2\times S^2$ minus a ball as a 2-handlebody, considered up to 2-equivalence.} It is therefore rather reassuring that our invariant vanishes on $S^2\times S^2$.

\begin{Theorem}
If $p\neq 3$, we have that
$$\dot{\mathcal{S}}_{p^{(2)}}^{\zeta^{1/2}}\bigg( \begin{tikzpicture}[anchorbase,scale=0.5]
\draw[usual] (1,0) arc (180:0:1); 
\draw[usual] (0,0) arc (-180:0:1); 
\draw[usual, crossline] (0,0) arc (180:0:1);
\draw[usual, crossline] (1,0) arc (-180:0:1);
\end{tikzpicture} \bigg)=0\,.$$
If $p=3$, we have that $\dot{\mathcal{S}}_{p^{(2)}}^{\zeta^{1/2}}(S^2\times S^2) = 1$.
\end{Theorem}
\begin{proof}
    By definition, this invariant is computed by coloring both components of the Hopf link in red, and then turning both of them blue using the chromatic morphism, i.e.\ using Equation \eqref{eq:redtoblue}. For the purposes of this computation, it is useful to turn them blue one at a time. In doing so, we find
\begin{multline*}
    \Inv\left( \begin{tikzpicture}[anchorbase,scale=0.5]
\draw[usual] (1,0) arc (180:0:1); 
\draw[usual] (0,0) arc (-180:0:1); 
\draw[usual, crossline] (0,0) arc (180:0:1);
\draw[usual, crossline] (1,0) arc (-180:0:1);
\end{tikzpicture} \right)
=
\mathsf{t}\left(\begin{tikzpicture}[anchorbase,scale=0.5]
\draw[usual, red] (1,0) arc (180:0:1)node[pos = 0.6, above]{\small red}; 
\draw[usual, red] (0,0) arc (-180:0:1); 
\draw[usual, crossline, red] (0,0) arc (180:0:1)node[pos = 0.4, above]{\small red};
\draw[usual, crossline, red] (1,0) arc (-180:0:1);
\draw[usual] (3.5,-1)--++(0,2.3) node[midway, right]{\tiny $\Tz(1)$};
\end{tikzpicture} \right)
= \sum_{v=2}^{2p-1}
\mathsf{t}\left(\begin{tikzpicture}[anchorbase,scale=0.5]
\draw[usual] (1,0) arc (180:0:1)node[pos = 0.6, above]{\tiny $\pjwm[v{-}1]$}; 
\draw[usual, red] (0,0) arc (-180:0:1); 
\draw[usual, crossline, red] (0,0) arc (180:0:1)node[pos = 0.4, above]{\small red};
\draw[usual, crossline] (1,0) arc (-180:0:1);
\draw[usual] (3.5,-1)--++(0,2.3);
\draw[cJW] (2.6,-0.5) rectangle (3.9, 0.5);
\node at (3.25,0) {\normalsize $\mathsf{c}$};
\end{tikzpicture} \right) 
=  \sum_{v=2}^{2p-1}
\mathsf{t}\left(\begin{tikzpicture}[anchorbase,scale=0.5]
\draw[usual] (1,0) arc (180:0:1)node[pos = 0.5, above]{\tiny $\pjwm[v{-}1]$}; 
\draw[usual, crossline] (1,0) arc (-180:0:1);
\draw[usual] (3.5,-1)--++(0,2.3);
\draw[cJW] (2.6,-0.5) rectangle (3.9, 0.5);
\node at (3.25,0) {\normalsize $\mathsf{c}$};
\node[rectangle, fill=white, draw=black] at (1,0) {\normalsize $\Delta^0_{v-1}$};
\end{tikzpicture} \right) 
=0\,.
\end{multline*}
The last sum evaluates to zero because every morphism $\Delta^0_{v-1}$ vanishes when $p>3$ by Proposition \ref{prop:Delta0}.
\end{proof}

\subsection{The Mazur Manifold}

We finally turn our attention to the 4-dimensional 2-handlebody presented by 
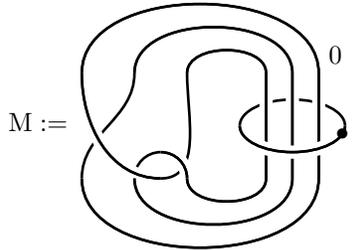
\begin{equation*}    \mathrm{M} :=
\begin{tikzpicture}[anchorbase,scale=0.7]
\draw[usual] (3,0) arc (180:0:1 and 0.5); 
\draw[usual] (0,-1) to[out=90,in=270] (1,1);
\draw[usual] (1.5,-1) to[out = 0, in=270] (2,1);
\draw[usual] (1.5,-0.5) to[out = 180, in=90] (1,-1);
\draw[usual, crossline] (1.5,-0.5) to[out = 0, in=90] (2,-1);
\draw[usual, crossline](1.5,-1) to[out = 180, in=270] (0,1);
\draw[usual, crossline] (0,1) to[out = 90, in=90] (4.5,1) node[above right]{0} -- (4.5,-1) to[out=270, in=270] (0,-1);
\draw[usual, crossline] (1,1) to[out = 90, in=90] (4,1) -- (4,-1) to[out=270, in=270] (1,-1);
\draw[usual, crossline] (2,1) to[out = 90, in=90] (3.5,1) -- (3.5,-1) to[out=270, in=270] (2,-1);
\draw[usual, crossline] (3,0) arc (-180:0:1 and 0.5) node[pos = 0.9]{$\bullet$}; 
\end{tikzpicture}\,.
\end{equation*}
In the diagram above, the integer $0$ denotes the framing of the 2-handle, so that we have to add three positive twists to the blackboard framing.

The corresponding 4-manifold is a Mazur manifold in the sense that it is obtained from a single 1-handle and a single 2-handle that cancel each other out algebraically. In particular, it is homotopy equivalent to the 4-disk. Notably, Akbulut showed in \cite{Akbulut} that $\mathrm{M}$ is a relative exotic manifold:\ there exists a diffeomorphism $f$ of $\partial\mathrm{M}$ that extends to a homeomorphism of $\mathrm{M}$, but not to a diffeomorphism. 

The most subtle part in this argument is to find an obstruction to the extension of the diffeomorphism $f$ of the boundary $\partial\mathrm{M}$ to a diffeomorphism of the whole 4-manifold $\mathrm{M}$.
Skein-theoretic 4-dimensional TQFTs are particularly well-suited to this purpose.
Indeed, if we find any skein $T\subset \partial \mathrm{M}$ such that the invariant $\SS$ distinguishes $T$ from its image under $f$, i.e.
$$\SS(-\mathrm{M})(T)\neq \SS(-\mathrm{M})(f(T))\,,$$
then it would immediately follow that $f$ cannot extend to a diffeomorphism of $\mathrm{M}$. Note that there is a subtlety in the above argument in the case when $\SS$ is only an invariant of 4-dimensional 2-handlebodies up to 2-equivalence.

In the case of the Mazur manifold $\mathrm{M}$, the diffeomorphism of the boundary $f$ is obtained by exchanging the dotted and undotted components. In particular, it sends meridians of one to meridians of the other. Hence, if we can find colours for the meridians of the dotted and undotted component that give different invariants when they are interchanged, we would have shown that $f$ cannot extend to $\mathrm{M}$. We now examine whether our invariant $\SS^{\zeta^{1/2}}_{p^{(2)}}$ can provide such an obstruction.

\begin{Definition}
For $v=[a_1,a_0]_{p,\ell}$ and $w=[b_1,b_0]_{p,\ell}$, we let $\mathrm{M}_{v-1,w-1}\in \mathbbm{k}$ be the invariant of the Mazur manifold with a meridian of the undotted component coloured by $\Tz(v-1)$ along with a meridian of the dotted circle coloured by $\Tz(w-1)$.\footnote{We need not prescribe an orientation of the meridians because all the objects $\Tz(v-1)$ are self-dual.} Pictorially, we have
\begin{equation*}
\mathrm{M}_{v-1,w-1} := \SS^{\zeta^{1/2}}_{p^{(2)}}\left(
\begin{tikzpicture}[anchorbase,yscale=0.6, xscale = 0.7]
\draw[usual, blue] (-0.5,1.3) arc (180:0:0.5 and 0.35) node[pos = 0.3, above = 0pt]{\tiny $\mathbb{T}_{\zeta}(v-1)$}; 
\draw[usual, blue] (4.75,0) arc (-180:0:0.5 and 0.35); 
\draw[usual] (3,0) arc (180:0:1 and 0.5); 
\draw[usual] (0,-1) to[out=90,in=270] (1,1);
\draw[usual] (1.5,-1) to[out = 0, in=270] (2,1);
\draw[usual] (1.5,-0.5) to[out = 180, in=90] (1,-1);
\draw[usual, crossline] (1.5,-0.5) to[out = 0, in=90] (2,-1);
\draw[usual, crossline](1.5,-1) to[out = 180, in=270] (0,1);
\draw[usual, crossline] (0,1) to[out = 90, in=90] (4.5,1) node[above right]{0} -- (4.5,-1) to[out=270, in=270] (0,-1);
\draw[usual, crossline] (1,1) to[out = 90, in=90] (4,1) -- (4,-1) to[out=270, in=270] (1,-1);
\draw[usual, crossline] (2,1) to[out = 90, in=90] (3.5,1) -- (3.5,-1) to[out=270, in=270] (2,-1);
\draw[usual, crossline] (3,0) arc (-180:0:1 and 0.5) node[pos = 0.8]{$\bullet$}; 
\draw[usual, blue, crossline] (4.75,0) arc (180:0:0.5 and 0.35) node[pos = 0.5, above right = 0pt]{\tiny $\mathbb{T}_{\zeta}(w-1)$}; 
\draw[usual, crossline, blue] (-0.5,1.3) arc (-180:0:0.5 and 0.35); 
\end{tikzpicture}
\right)\ \ .
\end{equation*}
Let us also denote $\mathrm{M}_{0,w-1}$, respectively $\mathrm{M}_{v-1,0}$, the invariant corresponding decorating $\mathrm{M}$ with a single meridian around the undotted, respectively dotted, component. We also write $\mathrm{M}_{0,0} = \Inv^{\zeta^{1/2}}_{p^{(2)}}(\mathrm{M})$, which corresponds by definition to the invariant of $\mathrm{M}$ with no meridians, but with the skein $\Gamma_0$.
\end{Definition}

At height $n=2$ and with $\zeta$ a primitive fourth root of unity, the mixed Verlinde category $\Ver_{p^{(2)}}^{\zeta^{1/2}}$ does not yield an invariant $\mathcal{S}_{p^{(2)}}^{\zeta^{1/2}}$ able to detect the obstruction the extension of $f$ to all of $\mathrm{M}$.

\begin{Proposition}
Let $v=[a_1,a_0]_{p,\ell}$ and $w=[b_1,b_0]_{p,\ell}$, then we have

$$\mathrm{M}_{v-1,w-1} = \left\{
\begin{array}{ll}
-14\, a_1b_1\,, & \text{if } a_0 = 0 \text{ and } b_0=0\,, \\
4(-1)^{b_1-1}a_1b_1\,, & \text{if } a_0 = 1 \text{ and } b_0=0\,, \\
4(-1)^{a_1-1}a_1b_1\,, & \text{if } a_0 = 0 \text{ and } b_0=1\,, \\
0\,, & \text{if } a_0 = 1 \text{ and } b_0=1\,, \\
\end{array}   \right.$$

\begin{gather*}
    \mathrm{M}_{v-1,0} = \left\{\begin{array}{ll}
(-1)^{a_1-1}a_1\,, & \text{if } a_0 = 0\,, \\
0\,, & \text{if } a_0 = 1\,, \\
    \end{array}   \right. \quad 
    \mathrm{M}_{0, w-1} = \left\{\begin{array}{ll}
(-1)^{b_1-1}b_1\,, & \text{if } b_0 = 0\,, \\
0\,, & \text{if } b_0 = 1\,, \\
    \end{array}   \right.
\quad
    \mathrm{M}_{0,0} = 1\,.
\end{gather*}
In particular, we have $\mathrm{M}_{v-1,w-1} = \mathrm{M}_{w-1,v-1}$ for all $1\leq v,w\leq 2p-1$. 
\end{Proposition}
\begin{proof}
It is challenging to give a uniform computation that covers all the cases above. This stems from the fact that we have to invoke Equation \eqref{eq:redtoblue} in different ways depending on which skeins are present.

\paragraph{Case 1: $w-1$ is odd.} In this case, we must have $b_0 = 0$, so that we know that $w-1 \neq 0$. We are therefore entitled to use Equation \eqref{eq:redtoblue} on the strand coloured by $\Tz(w-1)$.
We find
\begin{equation*}    \mathrm{M}_{v-1,w-1} = \sum_{u = [c_1,c_0]_{p,\ell}} \tilde{\mathsf t}\left(
\begin{tikzpicture}[anchorbase,tinynodes, yscale=0.6, xscale = 0.9]
\draw[usual, blue] (-0.5,1.3) arc (180:0:0.5 and 0.5); 
\draw[usual, blue] (0,-1) to[out=90,in=270] (1,1);
\draw[usual, blue] (1.5,-1) to[out = 0, in=270] (2,1);
\draw[usual, blue] (1.5,-0.5) to[out = 180, in=90] (1,-1);
\draw[usual, crossline, blue] (1.5,-0.5) to[out = 0, in=90] (2,-1);
\draw[usual, crossline, blue](1.5,-1) to[out = 180, in=270] (0,1);
\draw[usual, crossline, blue] (0,1) --++(0,1) to[out = 90, in=90] (5.5,2) -- ++(0,-3) to[out=270, in=270] (0,-1) ;
\draw[usual, crossline, blue] (1,1) --++(0,1) to[out = 90, in=90] (5,2) -- ++(0,-3) to[out=270, in=270] (1,-1) ;
\draw[usual, crossline, blue] (2,1) --++(0,1) node[pos = 0.8, above, rotate =90, blue]{\scriptsize $\pjwm[u{-}1]$} to[out = 90, in=90] (4.5,2) -- ++(0,-3) to[out=270, in=270] (2,-1) ;
\draw[usual, blue] (2.5,1) --++(0,1) to[out = 90, in=90] (4,2) -- ++(0,-3) to[out=270, in=270] (2.5,-1) -- ++(0,2) ;
\draw[usual, blue] (2.8,1) --++(0,1)node[pos = 0.3, below, xshift = 0.4ex, rotate =90, blue]{\scriptsize $\pjwm[w{-}2]$} to[out = 90, in=90] (3.8,2) -- ++(0,-3) to[out=270, in=270] (2.8,-1) -- ++(0,2) ;
\draw[pJW] (2.25,-1) rectangle (3.2,-0.3);
\node at (2.75,-0.72) {\scriptsize $\pjwm[w{-}1]$};
\draw[cJW] (1.75,0.2) rectangle (2.6, 0.8);
\node at (2.17,0.46) {\small $\mathsf{c}_{\pjwm[u{-}1]}$};
\draw[lJW] (3.5,-0.5) rectangle (6,0.5);
\node at (4.75,-0.12) {\normalsize $\Lambda$};
\draw[usual, crossline, blue] (-0.5,1.3) arc (-180:0:0.5 and 0.5); 
\draw[pJW] (-0.9,1) rectangle (-0.1, 1.6);
\node at (-0.48,1.23) {\scriptsize $\pjwm[v{-}1]$};
\node[draw, rectangle, fill=white, minimum height=0.5cm] at (0,-1.2){\normalsize \phantom{$\theta_{u-1}^3$}};
\node at (0,-1.32){\normalsize $\theta_{u-1}^3$};
\end{tikzpicture}
\right)\ \ .
\end{equation*} 
Now, recall that the cutting morphism $\Lambda$ factors through the unit, i.e.\ it is a linear combination of composites of a morphism from the incoming strands to the unit with one from the unit to the outgoing strands. So as to simplify the graphical representation, we will make this splitting explicit by writing
$$\begin{tikzpicture}[anchorbase, yscale=0.6, xscale = 0.7]
\draw[lJW] (0,0.3) rectangle (2,1.2);
\node at (1, 0.75) {$\Lambda$};
\end{tikzpicture} = \begin{tikzpicture}[anchorbase, scale=0.6]
\draw[lJW] (0,-0.1) rectangle (2,0.5);
\node at (1, 0.22) {$\Lambda'$};
\draw[lJW] (0,1) rectangle (2,1.6);
\node at (1, 1.28) {$\Lambda''$};
\end{tikzpicture}$$
\noindent and leaving the sum implicit.

Note that we are cutting $3(u-1)+(w-1)$ strands, so the cutting morphism $\Lambda$ vanishes if $u-1$ and $w-1$ have different parity. Given that we have assumed that $b_0=0$, the sum reduces to the terms of the form $u=[c_1,0]_{p,\ell}$, so that $\Tz(u-1)$ is simple. The above sum therefore becomes
\begin{equation*}    \mathrm{M}_{v-1,w-1} = \sum_{u = [c_1,0]_{p,\ell}} \tilde{\mathsf{t}}\left(
\begin{tikzpicture}[anchorbase,tinynodes, yscale=0.6, xscale = 0.7]
\draw[usual, blue] (-0.5,1.3) arc (180:0:0.5 and 0.5); 
\draw[usual, blue] (0,-1.5) to[out=90,in=270] (1.1,1);
\draw[usual, blue] (1.5,-1) to[out = 0, in=270] (2,1);
\draw[usual, blue] (1.5,-0.5) to[out = 180, in=90] (1,-1.5);
\draw[usual, crossline, blue] (1.5,-0.5) to[out = 0, in=90] (2,-1.5);
\draw[usual, crossline, blue](1.5,-1) to[out = 180, in=270] (0,1);
\draw[usual, crossline, blue] (0,1) --++(0,1) ;
\draw[usual, crossline, blue] (1.1,1) --++(0,1);
\draw[usual, crossline, blue] (2,1) --++(0,1);
\draw[usual, blue] (2.5,-1.5) --++(0,3.5);
\draw[usual, blue] (2.8,-1.5) --++(0,3.5);
\draw[pJW] (2.3,-1) rectangle (3.2,-0.4);
\node at (2.78,-0.74) {$\pjwm[w{-}1]$};
\draw[cJW] (1.75,0.2) rectangle (2.6, 0.8);
\node at (2.17,0.5) {$\mathsf{c}_{\pjwm[u{-}1]}$};
\draw[lJW] (-0.5,2) rectangle (3.3,2.8);
\node at (1.5,2.2) {\normalsize$\Lambda''$};
\draw[lJW] (-0.5,-1.5) rectangle (3.3,-2.3);
\node at (1.5,-2.1) {\normalsize$\Lambda'$};
\draw[usual, crossline, blue] (-0.5,1.3) arc (-180:0:0.5 and 0.5); 
\draw[pJW] (-1,1.05) rectangle (-0.15, 1.55);
\node at (-0.53,1.25) {$\pjwm[v{-}1]$};
\node[draw, rectangle, fill=white, minimum height=0.5cm, inner sep=1pt] at (1.1,1.4){$\theta_{\pjwm[u{-}1]}^3$};
\end{tikzpicture}
\right)\ \ .
\end{equation*}

\noindent We have already identified most of the parts of the above diagrams. Namely, the chromatic morphism is $\mathsf{c}_{u-1}=(-1)^{c_1-1}c_1 \cdot (\mathrm{E}_{u-1}\otimes \mathrm{E}_1)$ by Proposition \ref{prop:chromatic}. The twist is $\theta_{u-1}^3 = (-1)^{c_1}\zeta^{\frac{9}{2}}\cdot \mathrm{E}_{u-1}$ by Proposition \ref{prop:twists}. Further, encircling with $\Tz(v-1)$ yields $e_{u,v}\cdot \mathrm{E}_{u-1}$, where $e_{u,v}$ is the scalar given by
\begin{equation}\label{eq:encircleCoeffSimples}
e_{u,v} = \left\{ \begin{array}{ll}
    v(-1)^{c_1-1}\,, & \text{  if  }\ a_0=0\,,  \\
    2(v-1)\,, & \text{  if  }\ a_0=1\,, \\
    1\,, &  \text{  if  }\ v-1=0\,.
\end{array} \right. 
\end{equation}
This follows from Lemma \ref{lem:encirclingbyasimple}, which needs to be applied twice in the case $a_0=1$, using that $\Tz(v-1)=\Tz(v-2)\otimes\Tz(1)$ in this case.
If $v-1=0$, note that we are encircling with no strands.

Finally, we can invoke the first part of Theorem \ref{thm:HopfPairings} to identify the Hopf pairing. We obtain
\begin{equation*}    
\mathrm{M}_{v-1,w-1} = \sum_{u = [c_1,0]_{p,\ell}} 
-c_1\zeta^{\frac{1}{2}}e_{u,v}\cdot
\tilde{\mathsf t}\left(\zeta \cdot
\begin{tikzpicture}[anchorbase,tinynodes, yscale=0.6, xscale = 0.7]
\draw[usual, blue] (0,-1.5)--(0,0) to[out=90,in=270] (1,2);
\draw[usual, crossline, blue] (0,2) to[out=270,in=90] (1,1);
\draw[usual, blue] (2,2) -- (2,1);
\draw[usual, blue] (1,1) arc(-180:0:0.5 and 1);
\draw[usual, blue] (1,-1.5) arc(180:0:0.5 and 1);
\draw[usual, blue] (2.9,-1.5) to (2.9,1.3);
\draw[usual, blue] (2.9,1.65) to (2.9,2);
\draw[JW] (2.5,1.3) rectangle (3.3, 1.65);
\node at (2.9, 1.42) {$\pjwm[w{-}1]$};
\draw[lJW] (-0.5,2) rectangle (3.3,2.8);
\node at (1.5,2.2) {\normalsize$\Lambda''$};
\draw[lJW] (-0.5,-1.5) rectangle (3.3,-2.3);
\node at (1.5,-2.1) {\normalsize$\Lambda'$};
\end{tikzpicture}
+ 2 \cdot
\begin{tikzpicture}[anchorbase,tinynodes, yscale=0.6, xscale = 0.7]
\draw[usual, blue] (0,-1.5)--(0,0.5) to[out=90,in=270] (1,2);
\draw[usual, crossline, blue] (0,2) to[out=270,in=90] (1,1);
\draw[usual, blue] (2,2) -- (2,1);
\begin{scope}[xshift = 0.6cm, yshift = 0.67cm, xscale = 0.22, yscale = 0.34]
\draw[usual, blue] (7,0) to[out=270,in=0] (4.5,-1.5) node[above = -1.9pt]{$\pjwm[{u{-}2}]$}to[out=180,in=270] (2,0);
\draw[usual, blue] (7,-4) to[out=90,in=0] (4.5,-2.5) node[below = 0.1pt]{$\pjwm[{u{-}2}]$}to[out=180,in=90] (2,-4);
\draw[usual, blue] (1,0) --++(0,-4);
\draw[usual, blue] (8,0) --++(0,-4);
\draw[usual, blue] (2,-5) --++(0,-2);
\draw[usual, blue] (7,-5) --++(0,-2);
\draw[JW] (0,0) rectangle (4,1);
\node at (2,0.3) {$\pjwm[{u{-}1}]$};
\draw[JW] (5,0) rectangle (9,1);
\node at (7,0.3) {$\pjwm[{u{-}1}]$};
\draw[JW] (0,-5) rectangle (4,-4);
\node at (2,-4.7) {$\pjwm[{u{-}1}]$};
\draw[JW] (5,-5) rectangle (9,-4);
\node at (7,-4.7) {$\pjwm[{u{-}1}]$};
\end{scope}
\draw[usual, blue] (2.9,-1.5) to (2.9,1.3);
\draw[usual, blue] (2.9,1.65) to (2.9,2);
\draw[JW] (2.5,1.3) rectangle (3.3, 1.65);
\node at (2.9, 1.42) {$\pjwm[w{-}1]$};
\draw[lJW] (-0.5,2) rectangle (3.3,2.8);
\node at (1.5,2.2) {\normalsize$\Lambda''$};
\draw[lJW] (-0.5,-1.5) rectangle (3.3,-2.3);
\node at (1.5,-2.1) {\normalsize$\Lambda'$};
\end{tikzpicture}
\,\right)\ .
\end{equation*}
By naturality of the cutting morphism \cite[Lemma 1.1(2)]{CGHPM}, we can slide any arbitrary morphism through. As a result, we obtain
\begin{equation*}    
\mathrm{M}_{v-1,w-1} = \sum_{u = [c_1,0]_{p,\ell}} 
-c_1\zeta^{\frac{1}{2}}e_{u,v}\cdot
\tilde{\mathsf t}\left(\zeta \cdot
\begin{tikzpicture}[anchorbase,tinynodes, yscale=0.6, xscale = 0.7]
\draw[usual, blue] (0,-1.5) to[out=90,in=180] (1.5,1) to[out = 0, in=90] (2,0);
\draw[usual, crossline, blue] (0,2) to[out=270,in=180] (1.5,-1) to[out = 0,in=270] (2,0);
\draw[usual, blue] (2.9,-1.5) to (2.9,0.1);
\draw[usual, blue] (2.9,0.45) to (2.9,2);
\draw[JW] (2.5,0.1) rectangle (3.3, 0.45);
\node at (2.9, 0.23) {$\pjwm[w{-}1]$};
\draw[lJW] (-0.5,2) rectangle (3.3,2.8);
\node at (1.5,2.2) {\normalsize$\Lambda''$};
\draw[lJW] (-0.5,-1.5) rectangle (3.3,-2.3);
\node at (1.5,-2.1) {\normalsize$\Lambda'$};
\end{tikzpicture}
+ 2 \cdot
\begin{tikzpicture}[anchorbase,tinynodes, yscale=0.6, xscale = 0.7]
\draw[usual, blue] (0,-1.5)--(0,-0.5) to[out=90,in=270] (1,1);
\draw[usual, crossline, blue] (0,2) -- (0,1) to[out=270,in=90] (1,-0.5);
\draw[usual, blue] (2,-0.5)--(2,1);
\draw[JW] (0.6,1) rectangle (1.4,1.3);
\node at (1,1.1) {$\pjwm[{u{-}1}]$};
\draw[JW] (1.6,1) rectangle (2.4,1.3);
\node at (2,1.1) {$\pjwm[{u{-}1}]$};
\draw[usual, blue] (2,1.3) to[out=90,in=0] (1.5,1.8) node[below = 0.1pt]{$\pjwm[{u{-}2}]$}to[out=180,in=90] (1, 1.3);
\draw[usual, blue] (2.2,1.3) -- ++(0,0.7);
\draw[usual, blue] (0.8,1.3) -- ++(0,0.7);
\draw[JW] (0.6,-0.5) rectangle (1.4,-0.8);
\node at (1,-0.7) {$\pjwm[{u{-}1}]$};
\draw[JW] (1.6,-0.5) rectangle (2.4,-0.8);
\node at (2,-0.7) {$\pjwm[{u{-}1}]$};
\draw[usual, blue] (2,-0.8) to[out=270,in=0] (1.5,-1.3) node[above = 0pt]{$\pjwm[{u{-}2}]$}to[out=180,in=270] (1, -0.8);
\draw[usual, blue] (2.2,-0.8) -- ++(0,-0.7);
\draw[usual, blue] (0.8,-0.8) -- ++(0,-0.7);
\draw[usual, blue] (2.9,-1.5) to (2.9,0.1);
\draw[usual, blue] (2.9,0.45) to (2.9,2);
\draw[JW] (2.5,0.1) rectangle (3.3, 0.45);
\node at (2.9, 0.23) {$\pjwm[w{-}1]$};
\draw[lJW] (-0.5,2) rectangle (3.3,2.8);
\node at (1.5,2.2) {\normalsize$\Lambda''$};
\draw[lJW] (-0.5,-1.5) rectangle (3.3,-2.3);
\node at (1.5,-2.1) {\normalsize$\Lambda'$};
\end{tikzpicture}
\,\right)\ .
\end{equation*}
For the first term, we use Proposition \ref{prop:twists} to identify the twist, so that we are only left with the following computation, which can be carried using Lemma \ref{lem:dimHom}:
\begin{equation*}
\tilde{\mathsf t}\left(\,
\begin{tikzpicture}[anchorbase, yscale=0.3, xscale = 0.5]
\draw[JW] (-0.2, -0.3) rectangle (1.2, 0.9);
\node at (0.5, 0.25) {\small$\pjwm[u{-}1]$};
\draw[usual, blue] (0.5,-1.5) to (0.5,-0.3);
\draw[usual, blue] (0.5,0.9) to (0.5,2);
\draw[JW] (1.6, -0.3) rectangle (3, 0.9);
\node at (2.3, 0.25) {\small$\pjwm[w{-}1]$};
\draw[usual, blue] (2.3,-1.5) to (2.3,-0.3);
\draw[usual, blue] (2.3,0.9) to (2.3,2);
\draw[lJW] (-0.5,1.5) rectangle (3.3,2.8);
\node at (1.5,2.1) {\normalsize$\Lambda''$};
\draw[lJW] (-0.5,-1) rectangle (3.3,-2.3);
\node at (1.5,-1.7) {\normalsize$\Lambda'$};
\end{tikzpicture}\, \right)
=
\operatorname{dim}\, \Hom(\Tz(u-1),\Tz(w-1)) = \delta_{u,w}\,.
\end{equation*}

For the second term, we can remove the two instances of the idempotent $\mathrm{E}_{u-1}$ on the right using Lemma \ref{lem:YisLtensorId1}. We therefore find
\begin{equation*}    
\begin{tikzpicture}[anchorbase,tinynodes, yscale=0.6, xscale = 0.7]
\draw[usual, blue] (0,-1.5)--(0,-0.5) to[out=90,in=270] (1,1);
\draw[usual, crossline, blue] (0,2) -- (0,1) to[out=270,in=90] (1,-0.5);
\draw[JW] (0.6,1) rectangle (1.4,1.3);
\node at (1,1.1) {$\pjwm[{u{-}1}]$};
\draw[usual, blue] (2,-0.8) --  (2,1.3) to[out=90,in=0] (1.5,1.8) node[below = 0.4pt]{$\pjwm[{u{-}2}]$}to[out=180,in=90] (1, 1.3);
\draw[usual, blue] (2.2,-1.5) -- ++(0,3.5);
\draw[usual, blue] (0.8,1.3) -- ++(0,0.7);
\draw[JW] (0.6,-0.5) rectangle (1.4,-0.8);
\node at (1,-0.7) {$\pjwm[{u{-}1}]$};
\draw[usual, blue] (2,-0.8) to[out=270,in=0] (1.5,-1.3)to[out=180,in=270] (1, -0.8);
\draw[usual, blue] (0.8,-0.8) -- ++(0,-0.7);
\draw[usual, blue] (2.9,-1.5) to (2.9,0.1);
\draw[usual, blue] (2.9,0.45) to (2.9,2);
\draw[JW] (2.5,0.1) rectangle (3.3, 0.45);
\node at (2.9, 0.23) {$\pjwm[w{-}1]$};
\draw[lJW] (-0.5,2) rectangle (3.3,2.8);
\node at (1.5,2.2) {\normalsize$\Lambda''$};
\draw[lJW] (-0.5,-1.5) rectangle (3.3,-2.3);
\node at (1.5,-2.1) {\normalsize$\Lambda'$};
\end{tikzpicture}
=
\begin{tikzpicture}[anchorbase, yscale=0.6, xscale = 0.7]
\draw[usual, blue] (0.1,-1.5)--(0.1,-0.5) to[out=90,in=270] (1,1.3);
\draw[usual, blue] (-0.2,-1.5)--(-0.2,-0.5) to[out=90,in=270] (0.8,1.3) -- ++(0,0.7);
\draw[usual, crossline, blue] (0.1,2) to  (0.1,1.2);
\draw[usual, crossline, blue] (0.1,0.85) to[out=270,in=90] (1,-0.5)--(1, -0.8);
\draw[usual, crossline, blue] (-0.2,2) to (-0.2,1.2);
\draw[usual, crossline, blue] (-0.2,0.85) to[out=270,in=90] (0.8,-0.5)--(0.8, -0.8);
\draw[usual, blue] (2,-0.8) --  (2,1.3) to[out=90,in=0] (1.5,1.8) node[below = 0.3pt, scale = 0.7]{\scriptsize$\pjwm[{u{-}2}]$}to[out=180,in=90] (1, 1.3);
\draw[usual, blue] (2.2,-1.5) -- ++(0,3.5);
\draw[usual, blue] (2,-0.8) to[out=270,in=0] (1.5,-1.3) to[out=180,in=270] (1, -0.8);
\draw[usual, blue] (0.8,-0.8) -- ++(0,-0.7);
\draw[JW] (-0.45,0.85) rectangle (0.35, 1.2);
\node at (-0.05,1.02) {\tiny  $\pjwm[u{-}1]$};
\draw[usual, blue] (2.9,-1.5) to (2.9,0.1);
\draw[usual, blue] (2.9,0.45) to (2.9,2);
\draw[JW] (2.5,0.1) rectangle (3.3, 0.45);
\node at (2.9, 0.25) {\tiny  $\pjwm[w{-}1]$};
\draw[lJW] (-0.5,2) rectangle (3.3,2.8);
\node at (1.5,2.4) {\normalsize$\Lambda''$};
\draw[lJW] (-0.5,-1.5) rectangle (3.3,-2.3);
\node at (1.5,-1.9) {\normalsize$\Lambda'$};
\end{tikzpicture}
= \zeta^{-\frac{(u-2)(u-4)}{2}}\zeta^{-\frac{2(u-2)}{2}} \cdot
\begin{tikzpicture}[anchorbase, yscale=0.6, xscale = 0.7]
\draw[usual, blue] (0.1,-1.5)--(0.1,-0.5) to[out=90,in=270] (1,1.3) -- (1,2);
\draw[usual, blue] (0.1,2) -- (0.1,1.2) ;
\draw[usual, crossline, blue] (0.1,0.85) to[out=270,in=90] (1,-0.5)--(1, -1.5);
\draw[usual, blue] (2.2,-1.5) -- ++(0,3.5);
\draw[usual, blue] (-0.2,-1.5) to (-0.2,0.85) node[left, pos = 0.3, xshift = -6pt]{\small$\pjwm[{u{-}2}]$};
\draw[usual, blue] (-0.2,1.2) to (-0.2,2);
\draw[JW] (-0.45,0.85) rectangle (0.35, 1.2);
\node at (-0.05,1.02) {\tiny  $\pjwm[u{-}1]$};
\draw[usual, blue] (2.9,-1.5) to (2.9,0.1);
\draw[usual, blue] (2.9,0.45) to (2.9,2);
\draw[JW] (2.5,0.1) rectangle (3.3, 0.45);
\node at (2.9, 0.25) {\tiny  $\pjwm[w{-}1]$};
\draw[lJW] (-0.5,2) rectangle (3.3,2.8);
\node at (1.5,2.4) {\normalsize$\Lambda''$};
\draw[lJW] (-0.5,-1.5) rectangle (3.3,-2.3);
\node at (1.5,-1.9) {\normalsize$\Lambda'$};
\end{tikzpicture}\,,
\end{equation*}
\noindent where the first equality arises from the naturality of the cutting morphism. Diagramamtically, this amounts to sliding the coupon labeled $\mathrm E_{u-1}$ through the cutting morphism. The second equality follows by repeatedly using Equation \eqref{eq:classicalbraidingabsorption} and counting the number of crossings that are eliminated each time.

Finally, we resolve the last remaining crossing, which produces two terms. On the one hand, the non-identity term is zero. Namely, if $c_1=1$,
this follows from the naturality of the cutting morphism. The general case may be reduced to that one using
\begin{equation}\label{eq:MazurNilpotentDie}
\begin{tikzpicture}[anchorbase, yscale=0.6, xscale = 0.7]
\draw[usual, blue] (2.2,-1.5) -- ++(0,3.5);
\draw[usual, blue] (1,-1.5) -- ++(0,3.5);
\ptru{1.8}{0.9}{\scriptsize $\!\!\mathrm{U}_{\pjwm[u{-}2]}^0$}{-2}{0.5};
\ptrd{1.8}{0.9}{\scriptsize $\mathrm{D}_{\pjwm[u]}^0$}{-2}{0};
\draw[usual, blue] (2.9,-1.5) to (2.9,0);
\draw[usual, blue] (2.9,0.6) to (2.9,2);
\draw[JW] (2.4,0) rectangle (3.4, 0.6);
\node at (2.9, 0.28) {\small  $\pjwm[w{-}1]$};
\draw[lJW] (0.5,2) rectangle (3.3,2.8);
\node at (2,2.4) {\normalsize$\Lambda''$};
\draw[lJW] (0.5,-1.5) rectangle (3.3,-2.3);
\node at (2,-1.9) {\normalsize$\Lambda'$};
\end{tikzpicture}
=
\begin{tikzpicture}[anchorbase, yscale=0.6, xscale = 0.7]
\draw[usual, blue] (2.2,-1.5) -- ++(0,3.5);
\draw[usual, blue] (1,-1.5) -- ++(0,3.5);
\ptru{1.8}{0.9}{\scriptsize $\!\!\mathrm{U}_{\pjwm[u{-}2]}^0$}{-2}{-1};
\ptrd{1.8}{0.9}{\scriptsize $\mathrm{D}_{\pjwm[u]}^0$}{-2}{1.2};
\draw[usual, blue] (2.9,-1.5) to (2.9,0);
\draw[usual, blue] (2.9,0.6) to (2.9,2);
\draw[JW] (2.4,0) rectangle (3.4, 0.6);
\node at (2.9, 0.28) {\small  $\pjwm[w{-}1]$};
\draw[lJW] (0.5,2) rectangle (3.3,2.8);
\node at (2,2.4) {\normalsize$\Lambda''$};
\draw[lJW] (0.5,-1.5) rectangle (3.3,-2.3);
\node at (2,-1.9) {\normalsize$\Lambda'$};
\end{tikzpicture}
=
\frac{-c_1}{c_1-1}\cdot
\begin{tikzpicture}[anchorbase, yscale=0.6, xscale = 0.7]
\draw[usual, blue] (2.2,-1.5) -- ++(0,3.5);
\draw[usual, blue] (1,-1.5) -- ++(0,3.5);
\ptru{1.8}{0.9}{\scriptsize $\!\!\mathrm{U}_{\pjwm[u{-}4]}^0$}{-2}{0.5};
\ptrd{1.8}{0.9}{\scriptsize $\mathrm{D}_{\pjwm[u{-}2]}^0$}{-2}{0};
\draw[usual, blue] (2.9,-1.5) to (2.9,0);
\draw[usual, blue] (2.9,0.6) to (2.9,2);
\draw[JW] (2.4,0) rectangle (3.4, 0.6);
\node at (2.9, 0.28) {\small  $\pjwm[w{-}1]$};
\draw[lJW] (0.5,2) rectangle (3.3,2.8);
\node at (2,2.4) {\normalsize$\Lambda''$};
\draw[lJW] (0.5,-1.5) rectangle (3.3,-2.3);
\node at (2,-1.9) {\normalsize$\Lambda'$};
\end{tikzpicture}\,,
\end{equation}
where the first equality follows from the naturality of the cutting morphisms, whereas the second is Equation \eqref{eq:partialtrace1strand}.
On the other hand, the identity term can be evaluated using Lemma \ref{lem:dimHom}, so that
\begin{equation*}
\tilde{\mathsf t}\left(\,
\begin{tikzpicture}[anchorbase, yscale=0.3, xscale = 0.5]
\draw[usual, blue] (0.5,-1.5) to (0.5,-0.3);
\draw[usual, blue] (0.5,0.9) to (0.5,2);
\draw[JW] (-0.3, -0.3) rectangle (1.1, 0.9);
\node at (0.45, 0.25) {\small $\pjwm[{u{-}1}]$};
\draw[usual, blue] (3.3,-1.5) to (3.3,-0.3);
\draw[usual, blue] (3.3,0.9) to (3.3,2);
\draw[JW] (2.5, -0.3) rectangle (3.9, 0.9);
\node at (3.25, 0.25) {\small $\pjwm[{w{-}1}]$};
\draw[usual, blue] (1.75,-1.5)--++(0,3.5);
\draw[usual, blue] (2.25,-1.5)--++(0,3.5);
\draw[lJW] (-0.5,1.5) rectangle (4.3,2.8);
\node at (2,2.1) {\normalsize$\Lambda''$};
\draw[lJW] (-0.5,-1) rectangle (4.3,-2.3);
\node at (2,-1.7) {\normalsize$\Lambda'$};
\end{tikzpicture}\, \right)
=
\operatorname{dim}\, \Hom(\Tz(u),\Tz(w)) = \delta_{u,w-2}+2\delta_{u,w}+\delta_{u,w+2}\ .
\end{equation*}
The first and the last terms are omitted when $b_1=1$ and $b_1 = p-1$, respectively. Note that in those case, we have that $c_1$ would be $0$ or $p$, so these term automatically disappear in $\mathbbm{k}$.

Bringing the above discussion together, we find that
\begin{align*}
    \mathrm{M}_{v-1,w-1} & = \sum_{u = [c_1,0]_{p,\ell}} \hskip-10pt
-c_1\zeta^{\frac{1}{2}}e_{u,v}\cdot
\bigg(
\zeta\cdot (-1)^{c_1}\zeta^{-\frac{3}{2}}\cdot \delta_{u,w} + 2 \cdot \zeta^{\frac{(u-2)^2}{2}} \cdot \zeta^{-\frac{1}{2}} \cdot (\delta_{u,w-2}+2\delta_{u,w}+\delta_{u,w+2}) 
\bigg)\\
& = \sum_{u = [c_1,0]_{p,\ell}} \hskip-10pt
-c_1\, e_{u,v}\cdot
\bigg((-1)^{c_1}\delta_{u,w} + 2 (-1)^{c_1-1} \cdot (\delta_{u,w-2}+2\delta_{u,w}+\delta_{u,w+2}) 
\bigg)\\
& = (-1)^{b_1-1}2(b_1-1)e_{w-2,v} + (-1)^{b_1}b_1e_{w,v}(-1+4) + (-1)^{b_1+1}2(b_1+1)e_{w+2,v}\,.
\end{align*}
\noindent This concludes the case $b_0=0$ as we have identified the scalar $e_{u,v}$ in Equation \eqref{eq:encircleCoeffSimples}.

\paragraph{Case 2: $w-1$ even.} In this case, we have that $w-1=0$ or $w=[b_1,1]_{p,\ell}$. We will further assume that either $v-1\neq 0$, or $v-1=w-1= 0$. We will make use of Equation \eqref{eq:redtoblue} on the skein labeled by $\Tz(v-1)$ if $v-1\neq 0$, and on the skein $\Gamma_0$ labeled by $\Tz(1)$ if $v-1=w-1= 0$. Going in this direction, we set 
\begin{equation*}
    e_{u,v} = \begin{tikzpicture}[baseline = -5pt, yscale = 0.5, xscale = 0.9]
        \draw[usual] (0.2,0) arc(180:0:0.9 and 1.6) node[pos = 0.6, above right = -2pt]{\tiny$\pjwm[v{-}2]$};
        \draw[usual] (0.5,0) arc(180:0:0.5 and 1.2);
        \draw[usual, crossline] (1,-2.2) -- (1,2);
        \draw[usual, crossline](0.2,0) arc(-180:0:0.9 and 1.6);
        \draw[usual, crossline] (0.5,0) arc(-180:0:0.5 and 1.2);
\draw[pJW] (0.5, -2.7) rectangle (1.5,-2.1);
\node at (1,-2.4){\tiny $\pjwm[u{-}1]$};
\draw[pJW] (1.2,-0.7) rectangle (2.2,-0.1);
\node at (1.75,-0.4){\tiny $\pjwm[v{-}1]$};
\draw[cJW] (0.8, 0.15) rectangle (1.7,0.75);
\node at (1.25,0.45) {\small $\mathsf{c}_{\pjwm[u{-}1]}$};
    \end{tikzpicture}\
    \quad\text{if }v-1 \neq 0 \quad\quad \text{or} \quad\quad
e_{u,v} = \begin{tikzpicture}[baseline = -5pt, yscale = 0.5, xscale = 0.9]
\draw[usual] (1,-1.5) -- (1,1.5);
\draw[usual](1.4,0) arc(-180:180:0.3 and 0.9);
\draw[pJW] (0.5, -2) rectangle (1.5,-1.4);
\node at (1.05,-1.7){\tiny $\pjwm[u{-}1]$};
\draw[cJW] (0.7, -0.5) rectangle (1.7,0.5);
\node at (1.2,0) {\small $\mathsf{c}_{\pjwm[u{-}1]}$};
\end{tikzpicture} \
\quad \text{if }v-1 =w-1=0 \,.
\end{equation*}
We therefore have
\begin{equation*}    \mathrm{M}_{v-1,w-1} = \sum_{u = [c_1,1]_{p,\ell}} \tilde{\mathsf{t}}\left(\,
\begin{tikzpicture}[anchorbase, yscale=0.6, xscale = 0.7]
\draw[usual, blue] (0,-1.5) to[out=90,in=270] (1,1);
\draw[usual, blue] (1.5,-1) to[out = 0, in=270] (2,1);
\draw[usual, blue] (1.5,-0.5) to[out = 180, in=90] (1,-1.5);
\draw[usual, crossline, blue] (1.5,-0.5) to[out = 0, in=90] (2,-1.5);
\draw[usual, crossline, blue](1.5,-1) to[out = 180, in=270] (0,1);
\draw[usual, crossline, blue] (0,1) --++(0,1) ;
\draw[usual, crossline, blue] (1,1) --++(0,1);
\draw[usual, crossline, blue] (2,1) --++(0,1);
\draw[usual, blue] (2.8,-1.5) --++(0,3.5);
\draw[pJW] (2.25,-1) rectangle (3.45,-0.3);
\node at (2.9,-0.7) {\tiny $\pjwm[w{-}1]$};
\draw[pJW] (1.4,0) rectangle (2.6,0.7);
\node at (2,0.35) {\tiny $\pjwm[u{-}1]$};
\draw[lJW] (-0.5,2) rectangle (3.3,2.8);
\node at (1.5,2.4) {\normalsize$\Lambda''$};
\draw[lJW] (-0.5,-1.5) rectangle (3.3,-2.3);
\node at (1.5,-1.9) {\normalsize$\Lambda'$};
\node[draw, rectangle, fill=white, minimum height=0.5cm, inner sep=1pt] at (-0.1,1.3){$e_{u,v}$};
\node[draw, rectangle, fill=white, minimum height=0.5cm, inner sep=1pt] at (1.2,1.3){\small $\theta_{\pjwm[u{-}1]}^3$};
\end{tikzpicture}\,\right)
= \sum_{u = [c_1,1]_{p,\ell}} \lambda_{u,v}\cdot\tilde{\mathsf{t}} \left(\,
\begin{tikzpicture}[anchorbase, yscale=0.6, xscale = 0.7]
\draw[usual, blue] (0,-1.5) to[out=90,in=270] (1,1);
\draw[usual, blue] (1.5,-1) to[out = 0, in=270] (2,1);
\draw[usual, blue] (1.5,-0.5) to[out = 180, in=90] (1,-1.5);
\draw[usual, crossline, blue] (1.5,-0.5) to[out = 0, in=90] (2,-1.5);
\draw[usual, crossline, blue](1.5,-1) to[out = 180, in=270] (0,1);
\draw[usual, crossline, blue] (0,1) --++(0,1) ;
\draw[usual, crossline, blue] (1,1) --++(0,1);
\draw[usual, crossline, blue] (2,1) --++(0,1);
\draw[usual, blue] (2.8,-1.5) --++(0,3.5);
\draw[pJW] (2.25,-1) rectangle (3.45,-0.3);
\node at (2.9,-0.7) {\tiny $\pjwm[w{-}1]$};
\draw[pJW] (1.4,0) rectangle (2.6,0.7);
\node at (2,0.35) {\tiny $\pjwm[u{-}1]$};
\draw[lJW] (-0.5,2) rectangle (3.3,2.8);
\node at (1.5,2.4) {\normalsize$\Lambda''$};
\draw[lJW] (-0.5,-1.5) rectangle (3.3,-2.3);
\node at (1.5,-1.9) {\normalsize$\Lambda'$};
\node[draw, rectangle, fill=white, minimum height=0.5cm, inner sep=1pt] at (1.2,1.3){\small $\theta_{\pjwm[u{-}1]}^3$};
\end{tikzpicture}
\,\right)\ \ .
\end{equation*}
Firstly, note that the sum only involves terms of the form $u=[c_1,1]_{p,\ell}$. This is because we are cutting through $3(u-1) + (w-1)$ strands so that the cutting morphism $\Lambda$ vanishes if $u-1$ and $w-1$ have different parity. The second equality follows by observing that the endomorphism $e_{u,v}$ can be written as $e_{u,v} = \lambda_{u,v}\cdot  \mathrm{E}_{u-1}+\mu_{u,v}\cdot \mathrm{L}^0_{u-1}$ for some scalars $\lambda_{u,v}$ and $\mu_{u,v}$. By the argument given around Equation \eqref{eq:MazurNilpotentDie}, the $\mathrm{L}^0_{u-1}$ terms does not contribute.

Now, the Hopf pairing of $\Tz(u-1)$ with itself has a very particular form by the fourth part of Theorem \ref{thm:HopfPairings}. The leading term consists of an unlinked cap and cup, whereas all the additional terms contain (at least) one instance of the pair of morphisms $\mathrm{U}^0_{u-1}$ and $\mathrm{D}^0_{u-1}$. Using a variant of the argument given around Equation \eqref{eq:MazurNilpotentDie}, we observe that the additional terms do not contribute. 
We therefore have that

\begin{equation*}
\mathrm{M}_{v-1,w-1} = \sum_{u = [c_1,1]_{p,\ell}} \lambda_{u,v}\cdot\tilde{\mathsf{t}}\left(\,
\begin{tikzpicture}[anchorbase, yscale=0.6, xscale = 0.7]
\draw[usual, blue] (0,-1.5) to[out=90,in=270] (1,1);
\draw[usual, blue] (1.5,-0.5) to[out = 0, in=270] (2,1);
\draw[usual, blue] (1.5,-1) to[out = 180, in=90] (1,-1.5);
\draw[usual, crossline, blue] (1.5,-1) to[out = 0, in=90] (2,-1.5);
\draw[usual, crossline, blue](1.5,-0.5) to[out = 180, in=270] (0,1);
\draw[usual, crossline, blue] (0,1) --++(0,1) ;
\draw[usual, crossline, blue] (1,1) --++(0,1);
\draw[usual, crossline, blue] (2,1) --++(0,1);
\draw[usual, blue] (2.8,-1.5) --++(0,3.5);
\draw[pJW] (2.25,-1) rectangle (3.45,-0.3);
\node at (2.9,-0.7) {\tiny $\pjwm[w{-}1]$};
\draw[pJW] (1.4,0) rectangle (2.6,0.7);
\node at (2,0.35) {\tiny $\pjwm[u{-}1]$};
\draw[lJW] (-0.5,2) rectangle (3.3,2.8);
\node at (1.5,2.4) {\normalsize$\Lambda''$};
\draw[lJW] (-0.5,-1.5) rectangle (3.3,-2.3);
\node at (1.5,-1.9) {\normalsize$\Lambda'$};
\node[draw, rectangle, fill=white, minimum height=0.5cm, inner sep=1pt] at (1.2,1.3){\small $\theta_{\pjwm[u{-}1]}^3$};
\end{tikzpicture}
\,\right)
= \sum_{u = [c_1,1]_{p,\ell}} \tilde{\mathsf{t}}\left(\ 
\begin{tikzpicture}[anchorbase, yscale=0.6, xscale = 0.7]
\draw[usual, blue] (2.8,-1.5) --++(0,3.5);
\draw[usual, blue] (0,-1.5) --++(0,3.5);
\draw[pJW] (2.25,-1) rectangle (3.45,-0.3);
\node at (2.9,-0.7) {\tiny $\pjwm[w{-}1]$};
\draw[pJW] (-0.5,-1.2) rectangle (0.7,-0.4);
\node at (0.1,-0.8) {\tiny $\pjwm[u{-}1]$};
\draw[lJW] (-0.5,2) rectangle (3.3,2.8);
\node at (1.5,2.4) {\normalsize$\Lambda''$};
\draw[lJW] (-0.5,-1.5) rectangle (3.3,-2.3);
\node at (1.5,-1.9) {\normalsize$\Lambda'$};
\node[draw, rectangle, fill=white, minimum height=0.5cm, inner sep=1pt] at (0,1.3){$e_{u,v}$};
\node[draw, rectangle, fill=white, minimum height=0.5cm, inner sep=1pt] at (0,0.3){\small $\theta_{\pjwm[u{-}1]}^2$};
\end{tikzpicture}
\right) \ .
\end{equation*}
The second equality follows from $e_{u,v} = \lambda_{u,v}\cdot  \mathrm{E}_{u-1}+\mu_{u,v}\cdot \mathrm{L}^0_{u-1}$ using the above argument.

At this point, it only remains to identify the morphism $e_{u,v}$. In lieu of doing this directly, we observe that the last expression for $\mathrm{M}_{v-1,w-1}$ obtained above corresponds to the invariant associated to another 4-dimensional 2-handlebofy, given by the Hopf link of a dotted component with an undotted component of framing 2. Pictorially, this gives
\begin{equation*}    \mathrm{M}_{v-1,w-1} = \SS^{\zeta^{1/2}}_{p^{(2)}}\left(
\begin{tikzpicture}[anchorbase,scale=0.7]
\draw[usual, blue] (2.5,0) arc (-180:0:0.5); 
\draw[usual, blue] (-0.5,0) arc (180:0:0.5 and 0.3)node[pos = 0.5, above left = -2pt]{\tiny $\mathbb{T}_{\zeta}(v-1)$}; 
\draw[usual] (1,0) arc (180:0:1 and 0.5); 
\draw[usual] (0,0) arc (-180:0:1); 
\draw[usual, crossline] (0,0) arc (180:0:1) node[midway, above]{2};
\draw[usual, crossline] (1,0) arc (-180:0:1 and 0.5) node{$\bullet$};
\draw[usual, crossline, blue] (-0.5,0) arc (-180:0:0.5 and 0.3); 
\draw[usual, crossline, blue] (2.5,0) arc (180:0:0.5)node[pos = 0.5, above right = -2pt]{\tiny $\mathbb{T}_{\zeta}(w-1)$}; 
\end{tikzpicture}
\right)\ \ .
\end{equation*}
These two handles cancel each other out, so this 2-handlebody is the 4-disk. Using sliding relations to get the skeins untangled from the handles, we obtain
\begin{equation*}    \mathrm{M}_{v-1,w-1} = \SS\left(
\begin{tikzpicture}[anchorbase,scale=0.7]
\draw[usual, blue] (5,0) arc (-180:0:1); 
\draw[usual, blue] (4,0) arc (180:0:1 and 0.6)node[pos = 0.5, above left = -2pt]{\tiny $\mathbb{T}_{\zeta}(v-1)$}; 
\draw[usual] (1,0) arc (180:0:1 and 0.5); 
\draw[usual] (0,0) arc (-180:0:1);
\draw[usual, crossline] (0,0) arc (180:0:1) node[midway, above]{2};
\draw[usual, crossline] (1,0) arc (-180:0:1 and 0.5) node{$\bullet$};
\draw[usual, crossline, blue] (4,0) arc (-180:0:1 and 0.6); 
\draw[usual, crossline, blue] (5,0) arc (180:0:1)node[pos = 0.5, above right = -2pt]{\tiny $\mathbb{T}_{\zeta}(w-1)$}; 
\node[draw, rectangle, fill=white, minimum height=0.5cm, inner sep=1pt] at (7,0){\small $\theta_{\pjwm[w{-}1]}^2$};
\end{tikzpicture}
\right)\
=
\tilde{\mathsf{t}}\left(
\begin{tikzpicture}[anchorbase,scale=0.7]
\draw[usual, blue] (5,0) arc (-180:0:1); 
\draw[usual, blue] (4,0) arc (180:0:1 and 0.6)node[pos = 0.5, above left = -2pt]{\tiny $\mathbb{T}_{\zeta}(v-1)$}; 
\draw[usual, crossline, blue] (4,0) arc (-180:0:1 and 0.6); 
\draw[usual, crossline, blue] (5,0) arc (180:0:1)node[pos = 0.5, above right = -2pt]{\tiny $\mathbb{T}_{\zeta}(w-1)$}; 
\node[draw, rectangle, fill=white, minimum height=0.5cm, inner sep=1pt] at (7,0){\small $\theta_{\pjwm[w{-}1]}^2$};
\end{tikzpicture}
\right)\ .
\end{equation*}
The right hand-side can be evaluated using Lemma \ref{lem:encirclingbyasimple} and Proposition \ref{prop:twists}, thereby concluding our computation.

Finally, the case when $v-1=0$, and $w-1$ is nonzero and even follows via a similar argument. This concludes the proof.
\end{proof}

\bibliography{bibliography.bib}

\end{document}